%% file: BE_boundeddomains_MaxwellBC.tex
\documentclass[12pt]{amsart}  

\usepackage{graphicx}
\usepackage{here} 
\usepackage{array} 

\usepackage{vmargin}
\usepackage{amssymb}
\usepackage{mathrsfs}
\usepackage[all]{xy}
\usepackage[usenames,dvipsnames]{color}
\usepackage{soul}
\RequirePackage[colorlinks,linkcolor=blue,citecolor=LimeGreen,urlcolor=red]{hyperref} 
\usepackage{amsmath}

\newtheorem{theorem}{Theorem}[section]

\newtheorem{cor}[theorem]{Corollary}

\newtheorem{lemma}[theorem]{Lemma}

\newtheorem{prop}[theorem]{Proposition}
\newtheorem{remark}[theorem]{Remark}

\numberwithin{equation}{section}

\newcommand{\R}{\mathbb{R}}

\newcommand{\Q}{\mathbb{Q}}
\newcommand{\N}{\mathbb{N}}

\renewcommand{\S}{\mathbb{S}}

\newcommand{\func}[3]{#1 : #2 \longrightarrow #3}

\newcommand{\disp}{\displaystyle}
\newcommand{\abs}[1]{\left|#1\right|}
\newcommand{\eps}{\varepsilon}
\newcommand{\norm}[1]{\left\|#1\right\|}

\renewcommand{\leq}{\leqslant}
\renewcommand{\geq}{\geqslant}
\renewcommand{\bar}{\overline}
\renewcommand{\tilde}{\widetilde}
\newcommand{\pa}[1]{\left(#1\right)}
\newcommand{\cro}[1]{\left[#1\right]}
\newcommand{\br}[1]{\left\{#1\right\}}
\newcommand\restr[2]{{
  \left.\kern-\nulldelimiterspace 
  #1 
  \right|_{#2} 
  }}

\def\signmb{\bigskip \begin{center} {\sc
Marc Briant\par\vspace{3mm}
Sorbonne Universit\'es, UPMC Univ. Paris 06/ CNRS\par
UMR 7598, Laboratoire Jacques-Louis Lions,\par
F-75005, Paris, France\par
\vspace{3mm}
e-mail:} \tt{briant.maths@gmail.com} \end{center}}

\def\signyan{\bigskip \begin{center} {\sc
Yan Guo\par\vspace{3mm}
Brown University\par
Division of Applied Mathematics\par
182 George Street, Box F,
Providence, RI 02192, USA\par
\vspace{3mm}
e-mail:} \tt{yan\_guo@brown.edu} \end{center}}

\begin{document} 

\title[Asymptotic Stability of Boltzmann Equation with Maxwell Boundaries]{Asymptotic Stability of the Boltzmann Equation with Maxwell Boundary Conditions}
\author{Marc Briant}
\author{Yan Guo}
\thanks{The authors would like to acknowledge the Division of Applied Mathematics at Brown University, where this work was achieved.}

\begin{abstract}
In a general $C^1$ domain, we study the perturbative Cauchy theory for the Boltzmann equation with Maxwell boundary conditions with an accommodation coefficient $\alpha$ in $(\sqrt{2/3},1]$, and discuss this threshold. We consider polynomial or stretched exponential weights $m(v)$ and prove existence, uniqueness and exponential trend to equilibrium around a global Maxwellian in $L^\infty_{x,v}(m)$. Of important note is the fact that the methods do not involve contradiction arguments.
\end{abstract}

\maketitle


\textbf{Keywords:} Boltzmann equation; Perturbative theory; Maxwell boundary conditions; Specular reflection boundary conditions; Maxwellian diffusion boundary conditions. 

%

\tableofcontents

\input{introduction}

\input{mainresults}

\input{L2theory}

\input{frequencycollision}

\input{Linftytheory}

\input{fullcauchy}

\input{qualitativestudy}



%
\bibliographystyle{acm}
\bibliography{bibliography_MaxwellBC}


\bigskip
\signmb
\signyan

\end{document}

%% file: introduction.tex
\section{Introduction} \label{sec:intro}

The Boltzmann equation rules the dynamics of rarefied gas particles moving in a domain $\Omega$ of $\R^3$ with velocities in $\R^3$ when the sole interactions taken into account are elastic binary collisions. More precisely, the Boltzmann equation describes the time evolution of $F(t,x,v)$, the distribution of particles in position and velocity, starting from an initial distribution $F_0(x,v)$. It reads
\begin{eqnarray}
\forall t \geq 0 &,& \:\forall (x,v) \in \Omega \times \R^3,\quad  \partial_t F + v\cdot \nabla_x F = Q(F,F),\label{BE}
\\ && \:\forall (x,v) \in \Omega \times \R^3,\quad F(0,x,v) = F_0(x,v). \nonumber
\end{eqnarray}
To which one have to add boundary conditions on $F$. Throughout this work we consider $C^1$ bounded domains which allows us to decompose the phase space boundary
$$\Lambda = \partial \Omega\times\R^3$$
into three sets
\begin{eqnarray*}
\Lambda^+ &=& \br{\pa{x,v}\in\partial\Omega\times\R^3, \quad n(x)\cdot v >0},
\\\Lambda^- &=& \br{\pa{x,v}\in\partial\Omega\times\R^3, \quad n(x)\cdot v <0},
\\\Lambda_0 &=& \br{\pa{x,v}\in\partial\Omega\times\R^3, \quad n(x)\cdot v =0},
\end{eqnarray*}
where $n(x)$ is the outward normal at a point $x$ on $\partial\Omega$. The set $\Lambda^+$ is the outgoing set, $\Lambda^-$ is the ingoing set and $\Lambda_0$ is called the grazing set.
\par In the present work, we consider the physically relevant case where the gas interacts with the boundary $\partial\Omega$ \textit{via} two phenomena. Part of the particles touching the wall elastically bounce against it like billiard balls (specular reflection boundary condition) whereas the other part are absorbed by the wall and then emitted back into the domain according to the thermodynamical equilibrium between the wall and the gas (Maxwellian diffusion boundary condition). This very general type of interactions will be referred to as Maxwell boundary condition and they mathematically translate into
\begin{equation}\label{mixedBC}
\begin{split}
\exists \alpha \in (0,1], \:\forall t > 0,&\:\forall (x,v) \in \Lambda^-,
\\& F(t,x,v) = (1-\alpha) F(t,x,\mathcal{R}_x(v)) + \alpha P_\Lambda(F(t,x,\cdot))(v)
\end{split}
\end{equation}
where the Maxwellian diffusion is given by
\begin{equation}\label{PLambda}
P_\Lambda(F(t,x,\cdot))(v) = c_\mu \mu(v)\left[\int_{v_*\cdot n(x)>0} F(t,x,v_*)\left(v_*\cdot n(x)\right)\:dv_*\right]
\end{equation}
with
$$\mu(v) = \frac{1}{\pa{2\pi}^{3/2}}e^{-\frac{\abs{v}^2}{2}} \quad\mbox{and}\quad c_\mu\int_{v\cdot n(x)>0} \mu(v)\left(v\cdot n(x)\right)\:dv=1.$$
\par Note that in our study we allow pure Maxwellian diffusion ($\alpha=1$) but not pure specular reflection ($\alpha=0$). The constant $\alpha$ is called the accommodation coefficient.

\bigskip
The operator $Q(F,F)$ encodes the physical properties of the interactions between two particles. This operator is quadratic and local in time and space. It is given by 
$$Q(F,F) =  \int_{\R^3\times \mathbb{S}^{2}}B\left(|v - v_*|,\mbox{cos}\:\theta\right)\left[F'F'_* - FF_*\right]dv_*d\sigma,$$
where $F'$, $F_*$, $F'_*$ and $F$ are the values taken by $F$ at $v'$, $v_*$, $v'_*$ and $v$ respectively. Define:
$$\left\{ \begin{array}{rl} \displaystyle{v'} & \displaystyle{= \frac{v+v_*}{2} +  \frac{|v-v_*|}{2}\sigma} \vspace{2mm} \\ \vspace{2mm} \displaystyle{v' _*}&\displaystyle{= \frac{v+v_*}{2}  -  \frac{|v-v_*|}{2}\sigma} \end{array}\right. \: \mbox{and} \quad \mbox{cos}\:\theta = \langle \frac{v-v_*}{\abs{v-v_*}},\sigma\rangle .$$
We recognise here the conservation of kinetic energy and momentum when two particles of velocities $v$ and $v_*$ collide to give two particles of velocities $v'$ and $v'_*$.
\par The collision kernel $B$ contains all the information about the interaction between two particles and is determined by physics. We mention, at this point, that one can derive this type of equations from Newtonian mechanics at least formally \cite{Ce}\cite{CIP}. The rigorous validity of the Boltzmann equation from Newtonian laws is known for short times (Landford's theorem \cite{La} or more recently \cite{GST,PSS}).
\par A very interesting \textit{a priori} property of the Boltzmann equation combined with Maxwell boundary condition is the preservation of mass. Indeed, standard properties of $Q$ (\cite{CIP}\cite{Ce}\cite{Vi2} among others) imply that if $F$ is solution to the Boltzmann equation then
\begin{equation}\label{massconservation}
\forall t\geq 0, \quad \int_{\Omega \times \R^3} F(t,x,v)\:dxdv = \int_{\Omega \times \R^3} F_0(x,v)\:dxdv,
\end{equation}
which physically means that the mass is preserved along time.

\bigskip
In the present paper we are interested in the well-posedness of the Boltzmann equation $\eqref{BE}$ for fluctuations around the global equilibrium 
$$\mu(v) = \frac{1}{\pa{2\pi}^{3/2}}e^{-\frac{\abs{v}^2}{2}}.$$
More precisely, in the perturbative regime $F=\mu + f$ we construct a Cauchy theory in $L^\infty_{x,v}$ spaces endowed with strech exponential or polynomial weights and study the continuity and the positivity of such solutions.
\par Under the perturbative regime, the Cauchy problem amounts to solving the perturbed Boltzmann equation
\begin{equation}\label{perturbedBE}
\partial_t f + v\cdot\nabla_x f = Lf + Q(f,f)
\end{equation}
with $L$ being the linear Boltzmann operator $Lf = 2Q(\mu,f)$ where we considered $Q$ as a symmetric bilinear operator
\begin{equation}\label{Qfg}
Q(f,g) = \frac{1}{2}\int_{\R^3\times \mathbb{S}^{2}}B\left(|v - v_*|,\mbox{cos}\:\theta\right)\left[f'g'_* + g'f'_* - fg_*-gf_*\right]dv_*d\sigma.
\end{equation}
Note that $f$ also satisfies the Maxwell boundary condition $\eqref{mixedBC}$ since $\mu$ does.
\bigskip


\subsection{Notations and assumptions}\label{subsec:notations}

We describe the assumptions and notations we shall use throughout the article.

\bigskip
\textbf{Function spaces.} 
Define for any $k \geq 0$ the functional
$$\forall\langle \cdot \rangle^k = \pa{1+\abs{\cdot}^k}.$$
\par The convention we choose is to index the space by the name of the concerned variable so we have, for $p$ in $[1,+\infty]$,
$$L^p_{[0,T]} = L^p\pa{[0,T]},\quad L^p_{t} = L^p \left(\R^+\right),\quad L^p_x = L^p\left(\Omega\right), \quad L^p_v = L^p\left(\R^3\right).$$
\par For $\func{m}{\R^3}{\R^+}$ a positive measurable function we define the following weighted Lebesgue spaces by the norms
\begin{eqnarray*}
\norm{f}_{L^\infty_{x,v}\pa{m}} &=& \sup\limits_{(x,v)\in\Omega\times\R^3}\cro{\abs{f(x,v)}\:m(v)}
\\\norm{f}_{L^1_vL^\infty_{x}\pa{m}} &=& \int_{\R^3}\sup\limits_{x\in\Omega}\abs{f(x,v)}\:m(v) \:dv
\end{eqnarray*}
and in general with $p$, $q$ in $[1,\infty)$: $\norm{f}_{L^p_vL^q_x\pa{m}} = \norm{\norm{f}_{L^q_x}m(v)}_{L^p_v}$.
\par We define the Lebesgue spaces on the boundary:
\begin{eqnarray*}
\norm{f}_{L^\infty_\Lambda\pa{m}} &=&\sup\limits_{(x,v)\in\Lambda}\cro{\abs{f(x,v)}\:m(v)}
\\\norm{f}_{L^1L^\infty_\Lambda\pa{m}} &=&\int_{\R^3}\sup\limits_{x:\: (x,v)\in\Lambda}\abs{f(x,v)v\cdot n(x)}\:m(v) \:dv
\end{eqnarray*}
with obvious equivalent definitions for $\Lambda^{\pm}$ or $\Lambda_0$. However, when we do not consider the $L^\infty$ setting in the spatial variable we define
$$\norm{f}_{L^2_{\Lambda}\pa{m}} = \cro{\int_{\Lambda} \abs{f(x,v)^2 m(v)^2 \abs{v\cdot n(x)}\:dS(x)dv}}^{1/2},$$
where $dS(x)$ is the Lebesgue measure on $\partial\Omega$. We emphasize here that when the underlying space in the velocity variable is $L^p$ with $p\neq \infty$, the measure we consider is $\abs{v\cdot n(x)}dS(x)$ as it is the natural one when one thinks about Green formula.

\bigskip
\textbf{Assumptions on the collision kernel.}
We assume that the collision kernel $B$ can be written as
\begin{equation}\label{B}
B(v,v_*,\theta) = \Phi\left(|v - v_*|\right)b\left( \mbox{cos}\:\theta\right),
\end{equation}
which covers a wide range of physical situations (see for instance \cite[Chapter 1]{Vi2}).
\par Moreover, we will only consider kernels with hard potentials, that is 
\begin{equation}\label{hardpot}
\Phi(z) = C_\Phi z^\gamma \:,\:\: \gamma \in [0,1],
\end{equation}
where $C_\Phi>0$ is a given constant. Of special note is the case $\gamma=0$ which is usually referred to as Maxwellian potentials.
We will assume that the angular kernel $b\circ \mbox{cos}$ is positive and continuous on $(0,\pi)$, and that it satisfies a strong form of Grad's angular cut-off:
\begin{equation}\label{cutoff}
b_\infty=\norm{b}_{L^\infty_{[-1,1]}}<\infty
\end{equation}
The latter property implies the usual Grad's cut-off \cite{Gr1}:
\begin{equation}\label{lb}
l_b = \int_{\mathbb{S}^{d-1}}b\left(\mbox{cos}\:\theta\right)d\sigma = \left|\mathbb{S}^{d-2}\right|\int_0^\pi b\left(\mbox{cos}\:\theta\right) \mbox{sin}^{d-2}\theta \:d\theta < \infty.
\end{equation}
Such requirements are satisfied by many physically relevant cases. The hard spheres case ($b=\gamma=1$) is a prime example.
\bigskip


\subsection{Comparison with previous studies} \label{subsec:previousresults}

Few results have been obtained about the perturbative theory for the Boltzmann equation with other boundary condition than the periodicity of the torus. On the torus we can mention \cite{Uk}\cite{Gu3}\cite{Gu4}\cite{MN}\cite{Bri1}\cite{GMM} for collision kernels with hard potentials with cutoff, \cite{GreStr} without the assumption of angular cutoff or \cite{Gu1}\cite{Kim} for soft potentials. A good review of the methods and techniques used can be found in the exhaustive \cite{UkYa}.
\par The study of the well-posedness of the Boltzmann equation, as well as the trend to equilibrium, when the spatial domain is bounded with non-periodic boundary condition is scarce and only focuses on hard potential kernels with angular cutoff. In \cite{Gu6}, exponential convergence to equilibrium in $L^\infty_{x,v}$ with the important weight $\langle v \rangle^\beta\mu(v)^{-1/2}$ was established. The boundary condition considered in \cite{Gu6} are pure specular reflections with $\Omega$ being strictly convex and analytic and pure Maxwellian diffusion with $\Omega$ being smooth and convex. Note that the arguments used in the latter work relied on a non-constructive $L^2_{x,v}$ theory.
\par More recently, the case of pure Maxwellian boundary condition has been resolved by \cite{EGKM} in $L^\infty_{x,v}\pa{\langle v \rangle^\beta\mu(v)^{-1/2}}$ in $\Omega$ smooth but not necessarily convex and, more importantly, with constructive arguments. They also deal with non-global Maxwellian diffusion and gave an explicit domain of continuity for the solutions. We also mention \cite{KimYun} for a perturbative study around a non-local and rotating Maxwellian. At last, a very recent work by the first author \cite{Bri6} extended the domain of validity of the previous study to $L^\infty_{x,v}(m)$ where $m$ is a less restrictive weight: a stretched exponential or a polynomial; both for specular reflection and Maxwellian diffusion. His methods are constructive from the results described above (but therefore still rely on the contradiction argument in $L^2_{x,v}$ and the analyticity of $\Omega$ for specular reflections).
\par We also mention some works in the framework of renormalized solutions in bounded domains. The existence of such solutions has been obtained in different settings \cite{Mis1}\cite{Mis2} with Maxwell boundary condition. The issue of asymptotic convergence for such solutions was investigated in \cite{DesVil1} where they proved a trend to equilibrium faster than any polynomial on condition that the solutions has high Sobolev regularity

\bigskip
The present work establishes the perturbative Cauchy theory for Maxwell boundary condition and exponential trend to equilibrium in $L^\infty_{x,v}$ with a stretched exponential and polynomial weight. There are four main contributions in this work. First, we allow mere polynomial weights for the perturbation, which is a significant improvement over the work \cite{Gu6}. Then we deal with more general, and more physically relevant, boundary conditions and we recover the existing results in the case of pure Maxwellian diffusion. Third, delicate uses of the diffusive part, since $\alpha >0$, gives constructive proofs and there are the first, to our knowledge, entirely constructive arguments when dealing with specular reflections. Finally, we propose a new method to establish an $L^2-L^\infty$ theory that simplifies both technically and conceptually the existing $L^2-L^\infty$ theory \cite{Gu6}\cite{EGKM}. We indeed estimate the action of the operator $K$ in between two consecutive rebounds against the wall and work with the different weight than all the previous studies, namely $\mu^{-1-0}$ where we prove that $K$ almost acts like $3\nu(v)$. Also, with such an estimate we get rid of the strict convexity and analyticity of $\Omega$ that was always required when dealing with some specular reflections. We only need $\Omega$ to be a $C^1$ bounded domain but as a drawback we require $\alpha>\sqrt{2/3}$ (this explicit threshold being obtained thanks to the precise control over $K$).
\par We conclude by mentioning that our results also give an explicit set of continuity of the aforementioned solutions. This was known only in the case of pure Maxwellian diffusion, in-flow and bounce-back boundary conditions \cite{Kim1}. In the case of $\Omega$ convex we recover the fact that the solutions are continuous away from the grazing set $\Lambda_0$ \cite{Gu6}. Concerning the regularity of solutions to the Boltzmann equation with boundary conditions we also refer to \cite{GKTT1}\cite{GKTT2}.
\bigskip


\subsection{Organisation of the article}\label{subsec:organization}

Section \ref{sec:mainresults} is dedicated to the statement and the description of the main results proved in this paper. We also describe our strategy, which mainly consists in four steps that make the skeleton of the present article.
\par Section \ref{sec:L2theory} is dedicated to the \textit{a priori} exponential decay of the solutions to the linear part of the perturbed equation in the $L^2$ setting.
\par In Section \ref{sec:frequencycollision} we start by giving a brief mathematical description of the specular characteristics. We then study the semigroup generated by the transport part and the collision frequency kernel $G_\nu=-v\cdot\nabla_x -\nu$ along with the Maxwell boundary condition.
\par We develop an $L^2-L^\infty$ theory in Section \ref{sec:Linftytheory} and we prove that $G=-v\cdot\nabla_x+L$ generates a $C^0$-semigroup in $L^\infty_{x,v}(\langle v \rangle^\beta\mu^{-1/2})$ that decays exponentially.
\par We prove the existence and uniqueness of solutions for the full Boltzmann equation $\eqref{BE}$ in the perturbative regime $F=\mu+f$ in Section \ref{sec:fullcauchy}.
\par At last, Section \ref{sec:qalitativestudy} deals with the positivity and the continuity of the solutions to the full Boltzmann equation that we constructed.
\bigskip

%% file: mainresults.tex
\section{Main results} \label{sec:mainresults}

The aim of the present work is to prove the following perturbative Cauchy theory for the full Boltzmann equation with Maxwell boundary condition.

\bigskip
\begin{theorem}\label{theo:perturbativeCauchy}
Let $\Omega$ be a $C^1$ bounded domain and let $\alpha$ in $(\sqrt{2/3},1]$. Define
\begin{equation}\label{kinfty}
k_\infty = 1+\gamma + \frac{16\pi b_\infty}{l_b}.
\end{equation}
Let $m=e^{\kappa_1\abs{v}^{\kappa_2}}$ with $\kappa_1 >0$ and $\kappa_2$ in $(0,2)$ or $m=\langle v \rangle^k$ with $k> k_\infty$.
\\There exists $\eta >0$ such that for any $F_0=\mu+f_0$ in $L^\infty_{x,v}(m)$ satisfying the conservation of mass $\eqref{massconservation}$ with 
$$\norm{F_0-\mu}_{L^\infty_{x,v}(m)}\leq \eta,$$
there exists a unique solution $F(t,x,v)= \mu(v) + f(t,x,v)$ in $L^\infty_{t,x,v}(m)$ to the Boltzmann equation $\eqref{BE}$ with Maxwell boundary condition $\eqref{mixedBC}$ and with $f_0$ as an initial datum. Moreover, 
\begin{itemize}
\item $F$ preserves the mass $\eqref{massconservation}$;
\item There exist $C$, $\lambda>0$ such that
$$\forall t\geq 0,\quad \norm{F(t)-\mu}_{L^\infty_{x,v}(m)}\leq Ce^{-\lambda t}\norm{f_0}_{L^\infty_{x,v}(m)};$$
\item If $F_0 \geq 0$ then $F(t)\geq 0$ for all $t$.
\end{itemize}
\end{theorem}
\bigskip

\begin{remark}
We make a few comments about the above theorem.
\begin{enumerate}
\item[(1)] Notice that we recover the case of pure diffusion \cite{Gu6}\cite{EGKM} since $\alpha=1$ is allowed.
\item[(2)] It is important to emphasize that the uniqueness holds in the pertubative sense, that is in the set of functions of the form $F=\mu+f$ with $f$ small. The uniqueness for the Boltzmann equation in $L^\infty_{t,x,v}(m)$ with Maxwell boundary condition in the general setting would be a very interesting problem to look at. 
\item[(3)] Recent results \cite{Bri2}\cite{Bri5} established a quantitative lower bound for the solutions in the case of pure specular reflections and pure diffusion respectively. We think that their methods could be directly applicable to the Maxwell boundary problem and the solutions described in the theorem above should have an exponential lower bound, at least when $\Omega$ is convex. However, we only give here a qualitative statement about the positivity.
\end{enumerate}
\end{remark}
\begin{remark}[Remarks about improvement over $\alpha$]
As we shall mention it in next sections, we can construct an explicit $L^2_{x,v}$ linear theory if $\alpha >0$ whereas we strongly need $\alpha > \sqrt{2/3}$ to develop an $L^\infty_{x,v}$ linear theory from the $L^2$ one. However, the $L^1_vL^\infty_x$ nonlinear theory only relies on the $L^\infty_{x,v}$ linear one. Decreasing our expectations on $\Omega$ would allow to increase the range for $\alpha$.
\begin{itemize}
\item \textbf{$\Omega$ smooth and convex: $\alpha >0$ and constructive.} Very recent result \cite{KimLee} managed to obtain an $L^\infty_{x,v}$ theory for sole specular reflections by iterating Duhamel's form three times (see later). Thus, a convex combination of their methods and ours allow to  derive an $L^2-L^\infty$ theory for any $\alpha$ in $[0,1]$ and it would be entirely constructive thanks to our explicit $L^2$ linear theory.
\item Unfortunately, a completely constructive $L^2_{x,v}$ theory for $\alpha=0$ is still missing at the moment.
\end{itemize} 
\end{remark}
\bigskip

In order to state our result about the continuity of the solutions constructed in Theorem \ref{theo:perturbativeCauchy} we need a more subtle description of $\partial\Omega$. As noticed by Kim \cite{Kim1}, some specific points on $\Lambda_0$ can offer continuity.
\par We define the inward inflection grazing boundary
$$\Lambda_0^{(I-)} = \Lambda_0 \cap \br{t_{min}(x,v)=0,\: t_{min}(x,-v)\neq 0 \:\mbox{and}\: \exists \delta>0, \:\forall \tau \in [0,\delta],\: x-\tau v \in \bar{\Omega}^c}$$
where $t_{min}(x,v)$ is the first rebound against the boundary of a particle starting at $x$ with a velocity $-v$ (see Subsection \ref{subsec:collisionfrequencycharacteristics} for rigorous definition). That leads to the boundary continuity set
$$\mathfrak{C}^-_\Lambda = \Lambda^- \cup \Lambda_0^{(I-)}.$$
As we shall see later, the continuity set $\mathfrak{C}^-_\lambda$ describes the set of boundary points in the phase space that lead to continuous specular reflections.
\par The key idea is to understand that the continuity of the specular reflection at each bounce against the wall will lead to continuity of the solution. We thus define the continuity set
\begin{equation*}
\begin{split}
\mathfrak{C} =& \Big\{\br{0}\times\br{\Omega\times\R^3\cup\pa{\Lambda^+\cup \mathfrak{C}^-_\Lambda}}\Big\} \cup \Big\{(0,+\infty)\times\mathfrak{C}^-_\Lambda \Big\}
\\&\cup \Big\{(t,x,v) \in (0,+\infty)\times\pa{\Omega\times\R^3\cup \Lambda^+}:
\\&\quad\quad \forall 1\leq k\leq N(t,x,v) \in\N, \pa{X_{k+1}(x,v),V_k(x,v)}\in \mathfrak{C}^-_\Lambda \Big\}.
\end{split}
\end{equation*}
The sequence $(T_k(x,v),X_k(x,v),V_k(x,v))_{k\in\N}$ is the sequence of footprints of the backward characteristic trajectory starting at $(x,v)$ and overcoming pure specular reflections; $N(t,x,v)$ is almost always finite and satisfies $T_{N(t,x,v)} \leq t < T_{N(t,x,v)+1}(x,v)$. We refer to Subsection \ref{subsec:collisionfrequencycharacteristics} for more details.

\bigskip
\begin{theorem}\label{prop:continuity}
Let $F(t,x,v) = \mu +f(t,x,v)$ be the solution associated to $F_0 =\mu+f_0$ described in Theorem \ref{theo:perturbativeCauchy}. Suppose that $F_0= \mu + f_0$ is continuous on $\Omega\times\R^3\cup\br{\Lambda^+\cup \mathfrak{C}^-_\Lambda}$ and satisfies the Maxwell boundary condition $\eqref{mixedBC}$ then $F=\mu + f$ is continuous on the continuity set $\mathfrak{C}$.
\end{theorem}
\bigskip

\begin{remark}
We emphasize here again that the above theorem holds only in the perturbative regime. We also point out the following properties of the continuity set.
\begin{enumerate}
\item[(1)] From \cite[Proposition A.4]{Bri2} we know that the set of points $(x,v)$ in $\bar{\Omega}\times\R^3$ that lead to problematic backward characteristics is of Lebesgue measure zero (see later for more details). We infer that $\mathfrak{C}$ is non-empty and when we only consider $t$ in $[0,T]$ for a given $T>0$, its complementary set is of measure zero.
\item[(2)] In the case of a convex domain $\Omega$, we recover the previous results \cite{Gu6} for both pure specular reflections and pure diffusion: $\mathfrak{C}=\R^+\times\pa{\bar{\Omega}\times\R^3 - \Lambda_0}$.
\end{enumerate}
\end{remark}
\bigskip


\subsection{Description of the strategy} \label{subsec:descriptionstrategy}

Our strategy can be decomposed into four main steps and we now describe each of them briefly.

\bigskip
\textbf{Step 1: A priori exponential decay in $L^2_{x,v}\pa{\mu^{-1/2}}$ for the full linear operator.} The first step is to prove that the existence of a spectral gap for $L$ in the sole velocity variable can be transposed to $L^2_{x,v}\pa{\mu^{-1/2}}$ when one adds the skew-symmetric transport operator $-v\cdot\nabla_x$. In other words, we prove that solutions to
$$\partial_t f= Gf = Lf - v\cdot\nabla_x f$$
in $L^2_{x,v}\pa{\mu^{-1/2}}$ decays exponentially fast. Basically, the spectral gap $\lambda_L$ of $L$ implies that for such a solution
$$\frac{d}{dt}\norm{f}^2_{L^2_{x,v}\pa{\mu^{-1/2}}} \leq -2\lambda_L \norm{f-\pi_L\pa{f}}^2_{L^2_{x,v}\pa{\mu^{-1/2}}},$$
where $\pi_L$ is the orthogonal projection in $L^2_v\pa{\mu^{-1/2}}$ onto the kernel of the operator $L$. This inequality exhibits the hypocoercivity of $L$. Therefore, one would like that the microscopic part $\pi_L^\bot(f)= f -\pi_L(f)$ controls the fluid part which has the following form 
$$\pi_L(f)(t,x,v) = \cro{a(t,x) + b(t,x)\cdot v + c(t,x)\abs{v}^2}\mu(v).$$
\par It is known \cite{Gu3}\cite{Gu4} that the fluid part has some elliptic regularity; roughly speaking one has
\begin{equation}\label{fluidmicroderivative}
\Delta \pi_L(f) \sim \partial^2 \pi_L^\bot f + \mbox{higher order terms},
\end{equation}
that can be used in Sobolev spaces $H^s$ to recover some coercivity. We follow the idea of \cite{EGKM} for Maxwellian diffusion and construct a weak version of the elliptic regularity of $a(t,x)$, $b(t,x)$ and $c(t,x)$ by multiplying these coordinates by test functions. Basically, the elliptic regularity of $\pi_{L}\pa{f}$ will be recovered thanks to the transport part applied to these test functions while, on the other side, $L$ will encode the control by $\pi_{L}^\bot\pa{f}$. The test functions we build works with specular reflections but the estimate for $b$ requires the integrability of the function on the boundary. Such a property holds for Maxwellian diffusion and this is why we cannot deal with the specific case $\alpha=0$.

\bigskip
\textbf{Step 2: Semigroup generated by the collision frequency kernel.} The collision frequency operator $G_\nu = -\nu(v) - v\cdot\nabla_x$ together with Maxwell boundary condition is proved to generate a strongly continuous semigroup with exponential decay in $L^\infty_{x,v}(m)$ with very general weights $m(v)$. The boundary operator associated with the Maxwell condition is of norm exactly one and therefore the standard theory of transport equation in bounded domains \cite{BePro} fails. The core idea is to obtain an implicit description of the solutions to $\partial_t f = G_\nu f$ along the characteristic trajectories and to prove that the number of trajectories that do not reach the initial plane $\{t=0\}$ after a large number of rebounds is very small. Such a method has been developed in \cite{Gu6} and extended in \cite{Bri6}; we adapt it to the case of Maxwell characteristics.

\bigskip
\textbf{Step 3: $L^\infty_{x,v}(\mu^{-\zeta})$ theory for the full nonlinear equation.} The underlying $L^2_{x,v}$-norm is not an algebraic norm for the nonlinear operator $Q$ whereas the $L^\infty_{x,v}$-norm is (see \cite{Ce}\cite{CIP} or \cite{Vi2} for instance). We therefore follow an $L^2-L^\infty$ theory \cite{Gu6} to pass on the previous semigroup property in $L^2$ to $L^\infty$ \textit{via} a change of variable along the flow of characteristics.
\par Basically, $L$ can be written as $L = -\nu(v) +K$ with $K$ a kernel operator. If we denote by $S_G(t)$ the semigroup generated by $G = L-v\cdot\nabla_x$ we have the following implicit Duhamel along the characteristics
$$S_G(t)= e^{-\nu(v)t} + \int_0^t e^{-\nu(v)(t-s)}K\cro{S_G(s)} \:ds.$$
The standard methods \cite{Vid}\cite{Gu6}\cite{EGKM} used an iterated version of this Duhamel's formula to recover some compactness property, thus allowing to bound the solution in $L^\infty$ by its $L^2$ norm. To do so they require to study the solution $f(t,x,v)$ along all the possible characteristic trajectories $(X_t(x,v),V_t(x,v))$. We propose here a less technical strategy by estimating the action of $K$ in between two consecutive collisions against $\partial\Omega$ thanks to trace theorems. The core contribution, which also gives the threshold $\alpha > \sqrt{2/3}$, is to work in $L^\infty_{x,v}(\mu^{-\zeta})$ as $\zeta$ goes to $1$ where $K$ is proven to act roughly like $3\nu(v)$.

\bigskip
\textbf{Step 4: Extension to polynomial weights.} To conclude the present study, we develop an analytic and nonlinear version of the recent work \cite{GMM}, also recently adapted in a nonlinear setting \cite{Bri6}. The main strategy is to find a decomposition of the full linear operator $G$ into $G_1+A$. We shall prove that $G_1$ acts like a small perturbation of the operator $G_{\nu} = -v\cdot\nabla_x - \nu(v)$ and is thus hypodissipative, and that $A$ has a regularizing effect. The regularizing property of the operator $A$ allows us to decompose the perturbative equation $\eqref{perturbedBE}$ into a system of differential equations
\begin{eqnarray}
\partial_t f_1 + v\cdot\nabla_x f_1 &=& G_1\pa{f_1} + Q(f_1+f_2,f_1+f_2)\label{introf1}
\\\partial_t f_2 + v\cdot\nabla_x f_2 &=& L\pa{f_2} + A\pa{f_1}.\label{introf2}
\end{eqnarray}
The first equation is solved in $L^\infty_{x,v}\pa{m}$ with the initial datum $f_0$ thanks to the hypodissipativity of $G_1$. The regularity of $A\pa{f_1}$ allows us to use Step 3 and thus solve the second equation with null initial datum in $L^\infty_{x,v}(\mu^{-\zeta})$.
\bigskip

%% file: L2theory.tex
\section{$L^2\pa{\mu^{-1/2}}$ theory for the linear part of the perturbed Boltzmann equation}\label{sec:L2theory}

This section is devoted to the study of the linear perturbed equation
$$\partial_t f + v\cdot\nabla_x f = L(f),$$
with the Maxwell boundary condition $\eqref{mixedBC}$ in the $L^2$ setting. Note that we only need $\alpha$ in $(0,1]$ in this section. As we shall see in Subsection \ref{subsec:LL2}, the space $L^2_v\pa{\mu^{-1/2}}$ is natural for the operator $L$. In order to avoid carrying the maxwell weight throughout the computations we look at the function $h(t,x,v) = f(t,x,v)\mu(v)^{-1/2}$. We thus study in this section the following equation in $L^2_{x,v}$

\begin{equation}\label{lineqL2}
\partial_t h + v\cdot\nabla_x h = L_\mu(h),
\end{equation}
with the associated boundary conditions
\begin{equation}\label{BCL2}
\forall t > 0,\:\forall (x,v) \in \Lambda^-,\quad  h(t,x,v) = (1-\alpha) h(t,x,\mathcal{R}_x(v)) + \alpha P_{\Lambda_\mu}(h)(t,x,v)
\end{equation}
where we defined 
$$L_\mu(h) = \frac{1}{\sqrt{\mu}}L\pa{\sqrt{\mu}h}$$
and $P_{\Lambda_\mu}$ can be viewed as a $L^2_v$-projection with respect to the measure $\abs{v\cdot n(x)}$:
\begin{equation}\label{PLambdamu}
\forall (x,v) \in \Lambda^-,\quad P_{\Lambda_\mu}(h) = c_\mu \sqrt{\mu(v)}\left[\int_{v_*\cdot n(x)>0} h(t,x,v_*)\sqrt{\mu(v_*)}\left(v_*\cdot n(x)\right)\:dv_*\right].
\end{equation}
We also use the shorthand notation $P_{\Lambda_\mu}^\bot = \mbox{Id}-P_{\Lambda_\mu}^\bot$.

\bigskip
For general domains $\Omega$, the Cauchy theory in $L^p_{x,v}$ ($1\leq p <+\infty$) of equations of the type
$$\partial_t f + v\cdot\nabla_x f = g$$
with boundary conditions
$$\forall (x,v) \in \Lambda^-, \quad f(t,x,v) = P(f)(t,x,v),$$
where $\func{P}{L^p_{\Lambda^+}}{L^p_{\Lambda^-}}$ is a bounded linear operator, is well-defined in $L^p_{x,v}$ when $\norm{P} <1$ \cite{BePro}. The specific case $\norm{P} = 1$ can still be dealt with (\cite{BePro} Section $4$) but even though the existence of solutions in $L^p_{x,v}$ can be proven, the uniqueness is not always given unless one can prove that the trace of $f$ belongs to $L^2_{\mbox{\scriptsize{loc}}}\pa{\R^+;L^p_{x,v}\pa{\Lambda}}$.
\par For Maxwell boundary conditions, the boundary operator $P$ is of norm exactly one and the general theory fails. The need of a trace in $L^2_{x,v}$ is essential to perform Green's identity and obtain the uniqueness of solutions. The pure Maxwellian boudary conditions with mass conservation can still be dealt with because one can show that $P_{\Lambda_\mu}^\bot(h)$ is in $L^2_{\Lambda^+}$ \cite{EGKM}. Unfortunately, in the case of specular reflections the uniqueness is not true in general due to a possible blow-up of the $L^2_{\mbox{\scriptsize{loc}}}\pa{\R^+;L^2_{x,v}\pa{\Lambda}}$ at the grazing set $\Lambda_0$ \cite{Uk3,BePro,CIP}.
\par Following ideas from \cite{Gu6}, a sole \textit{a priori} exponential decay of solutions is necessary to obtain a well-posed $L^\infty$ theory provided that we endow the space with a strong weight. This section is thus dedicated to the proof of the following theorem.

\bigskip
\begin{theorem}\label{theo:L2}
Let $\alpha >0$ and let $h_0$ be in $L^2_{x,v}$ such that $h_0$ satisfies the preservation of mass
$$\int_{\Omega\times\R^3} h_0(x,v)\sqrt{\mu(v)}\:dv=0.$$
Suppose that $h(t,x,v)$ in $L^2_{x,v}$ is a mass preserving solution to the linear perturbed Boltzmann equation $\eqref{lineqL2}$ with initial datum $h_0$ and satisfying the Maxwell boundary condition $\eqref{BCL2}$. Suppose also that $\restr{h}{\Lambda}$ belongs to $L^2_\Lambda$.
\\ Then there exist explicit $C_G$, $\lambda_G >0$, independent of $h_0$ and $h$, such that
$$\forall t \geq 0, \quad \norm{h(t)}_{L^2_{x,v}} \leq C_G e^{-\lambda_G t}\norm{h_0}_{L^2_{x,v}}.$$
\end{theorem}
\bigskip

In order to prove Theorem \ref{theo:L2} we first gather in Subsection \ref{subsec:LL2} some well-known properties about the linear operator $L$. Subsection \ref{subsec:controlfluidmicro} proves a very important lemma which allows to use the hypocoercivity of $L-v\cdot\nabla_x$ in the case of Maxwell boundary conditions. Finally, the exponential decay is proved in Subsection \ref{subsec:expodecayL2}.
\bigskip


\subsection{Preliminary properties of $L_\mu$ in $L^2_v$}\label{subsec:LL2}

\bigskip
\textbf{The linear Boltzmann operator.} We gather some well-known properties of the linear Boltzmann operator $L_\mu$ (see \cite{Ce}\cite{CIP}\cite{Vi2}\cite{GMM} for instance). 
\par $L_\mu$ is a closed self-adjoint operator in $L^{2}_v$ with kernel
$$\mbox{Ker}\left(L_\mu\right) = \mbox{Span}\left\{\phi_0(v),\dots,\phi_{4}(v)\right\}\sqrt{\mu} ,$$
where $\pa{\phi_i}_{0\leq i\leq 4}$ is an orthonormal basis of $\mbox{Ker}\left(L_\mu\right)$ in $L^2_v$. More precisely, if we denote $\pi_L$ to be the orthogonal projection onto $\mbox{Ker}\left(L_\mu\right)$ in $L^2_v)$:
\begin{equation}\label{piL}
\left\{\begin{array}{l} \disp{\pi_L(h) = \sum\limits_{i=0}^{4} \pa{\int_{\R^3} h(v_*)\phi_i(v_*)\sqrt{\mu(v_*)}\:dv_*} \phi_i(v)\sqrt{\mu(v)}} \vspace{2mm}\\\vspace{2mm} \disp{\phi_0(v)=1,\quad \phi_i(v) = v_i,\:1\leq i \leq 3,\quad \phi_{4}(v)=\frac{\abs{v}^2-3}{\sqrt{6}},}\end{array}\right.
\end{equation}
and we define $\pi_L^\bot = \mbox{Id} - \pi_L$. The projection $\pi_L(h(x,\cdot))(v)$ of $h(x,v)$ onto the kernel of $L_\mu$ is called its fluid part whereas $\pi_L^\bot(h)$ is its microscopic part.
\par Also, $L_\mu$ can be written under the following form
\begin{equation}\label{LnuK}
L_\mu = -\nu(v) +K,
\end{equation}
where $\nu(v)$ is the collision frequency
$$\nu(v) = \int_{\R^3\times\mathbb{S}^{2}} b\left(\mbox{cos}\:\theta\right)\abs{v-v_*}^\gamma \mu_*\:d\sigma dv_*$$
and $K$ is a bounded and compact operator in $L^2_v$.
\par Finally we remind that there exists $\nu_0,\:\nu_1 >0$ such that
\begin{equation}\label{nu0nu1}
\forall v \in \R^3,\quad \nu_0(1+\abs{v}^\gamma)\leq \nu(v)\leq \nu_1(1+\abs{v}^\gamma),
\end{equation}
and that $L_\mu$ has a spectral gap $\lambda_L >0$ in $L^2_{x,v}$ (see \cite{BM}\cite{Mo1} for explicit proofs)
\begin{equation}\label{spectralgapL}
\forall g \in L^2_v, \quad \langle L_\mu(g),g\rangle_{L^2_v} \leq -\lambda_L \norm{\pi_L^\bot(g)}_{L^2_v}^2.
\end{equation}

\bigskip
\textbf{The linear perturbed Boltzmann operator.}
The linear perturbed Boltzmann operator is the full linear part of the perturbed Boltzmann equation $\eqref{perturbedBE}$:
$$G = L - v\cdot \nabla_x$$
or, in our $L^2$ setting,
$$G_\mu = L_\mu - v\cdot\nabla_x.$$
An important point is that the same computations as to show the \textit{a priori} conservation of mass implies that in $L^2_{x,v}$ the space $\pa{\mbox{Span}\br{\sqrt{\mu}}}^\bot$ is stable under the flow
$$\partial_t h = G_\mu(h)$$
with Maxwell boundary conditions $\eqref{BCL2}$. Coming back to our general setting $f=h\sqrt{\mu}$ we thus define the $L^2_{x,v}\pa{\mu^{-1/2}}$ projection onto that space
\begin{equation}\label{PiG}
\Pi_G(f) = \left(\int_{\Omega\times\R^3} f(x,v_*)\:dxdv_*\right)\mu(v),
\end{equation}
and its orthogonal projection $\Pi_G^\bot = \mbox{Id} - \Pi_G$.  Note that $\Pi_G^\bot(f)=0$ amounts to saying that $f$ satisfies the preservation of mass.
\bigskip


\subsection{A priori control of the fluid part by the microscopic part}\label{subsec:controlfluidmicro}

As seen in the previous section, the operator $L_\mu$ is only coercive on its orthogonal part. The key argument is to show that we recover the full coercivity on the set of solutions to the differential equation. Namely, that for these specific functions, the microscopic part controls the fluid part. This is the purpose of the next lemma.

\bigskip
\begin{lemma}\label{lem:controlfluidmicro}
Let $h_0(x,v)$ and $g(t,x,v)$ be in $L^2_{x,v}$ such that $\Pi_G(h_0)= \Pi_G(g) =0$ and let $h(t,x,v)$ in $L^2_{x,v}$ be a mass preserving solution to
\begin{equation}\label{lineqL2lem}
\partial_t h + v\cdot\nabla_x h = L_\mu(h) + g
\end{equation}
with initial datum $h_0$ and satisfying the boundary condition $\eqref{BCL2}$. Suppose that $\restr{h}{\Lambda}$ belongs to $L^2_\Lambda$. Then there exists an explicit $C_\bot >0$ and a function $N_h(t)$ such that for all $t\geq 0$
\begin{itemize}
\item[(i)] $\abs{N_h(t)}\leq C_\bot \norm{h(t)}^2_{L^2_{x,v}}$;
\item[(ii)]  
\begin{equation*}
\begin{split}
\int_0^t\norm{\pi_L(h)}^2_{L^2_{x,v}}\:ds \leq & N_h(t)-N_h(0) + C_\bot \int_0^t \cro{\norm{\pi^\bot_L(h)}^2_{L^2_{x,v}} +\norm{P_{\Lambda_\mu}^\bot(h)}^2_{L^2_{\Lambda^+}}}\:ds
\\ &+ C_\bot \int_0^t\norm{g}^2_{L^2_{\Lambda^+}}\:ds.
\end{split}
\end{equation*}
\end{itemize}
The constant $C_\bot$ is independent of $h$.
\end{lemma}
\bigskip

The methods of the proof are a technical adaptation of the methods proposed in \cite{EGKM} in the case of purely diffusive boundary condition.

\bigskip
\begin{proof}[Proof of Lemma \ref{lem:controlfluidmicro}]

We recall the definition of $\pi_L$ $\eqref{piL}$ and we define the function $a(t,x)$, $b(t,x)$ and $c(t,x)$ by
\begin{equation}\label{piLL2}
\pi_L(h)(t,x,v) = \cro{a(t,x) + b(t,x)\cdot v + c(t,x)\frac{\abs{v^2}-3}{2}}\sqrt{\mu(v)}.
\end{equation}
The key idea of the proof is to choose suitable test function $\psi$ in $H^1_{x,v}$ that will catch the elliptic regularity of $a$, $b$ and $c$ and estimate them. Note that for $a$ we strongly use the fact that $h$ preserves the mass.
\par For a test function $\psi=\psi(t,x,v)$ integrated against the differential equation $\eqref{lineqL2}$ we have by Green's formula
\begin{eqnarray*}
&&\int_0^t\frac{d}{dt}\int_{\Omega\times\R^3}\psi h\:dxdvds = \int_{\Omega\times\R^3} \psi(t) h(t)\:dxdv -\int_{\Omega\times\R^3} \psi_0 h_0\:dxdv   
\\&&\quad\quad= \int_0^t \int_{\Omega\times\R^3} h \partial_t\psi\:dxdvds + \int_0^t \int_{\Omega\times\R^3} L_\mu[h]\psi\:dxdvds
\\&&\quad\quad\quad+\int_0^t \int_{\Omega\times\R^3} h v\cdot\nabla_x\psi \:dxdvds-\int_0^t\int_\Lambda \psi h v\cdot n(x)\:dS(x)dvds
\\&&\quad\quad\quad + \int_0^t \int_{\Omega\times\R^3}\psi g \:dxdvds. 
\end{eqnarray*}
We decompose $h = \pi_L(h) + \pi_L^\bot(h)$ in the term involving $v\cdot\nabla_x$ and use the fact that $L_\mu[h] = L_\mu[\pi_L^\bot(h)]$ to obtain the weak formulation
\begin{equation}\label{eqpsi}
-\int_0^t \int_{\Omega\times\R^3} \pi_L(h) v\cdot\nabla_x\psi \:dxdvds = \Psi_1(t) + \Psi_2(t)+\Psi_3(t) + \Psi_4(t) +\Psi_5(t)+\Psi_6(t)
\end{equation}
with the following definitions
\begin{eqnarray}
\Psi_1(t) &=& \int_{\Omega\times\R^3} \psi_0 h_0\:dxdv -\int_{\Omega\times\R^3} \psi(t) h(t)\:dxdv, \label{Psi1}
\\\Psi_2(t) &=& \int_0^t \int_{\Omega\times\R^3} \pi^\bot_L(h) v\cdot\nabla_x\psi \:dxdvds, \label{Psi2}
\\\Psi_3(t) &=& \int_0^t \int_{\Omega\times\R^3} L_\mu\cro{\pi_L^\bot(h)}\psi\:dxdvds, \label{Psi3}
\\\Psi_4(t) &=& -\int_0^t\int_\Lambda \psi h v\cdot n(x)\:dS(x)dvds,\label{Psi4}
\\\Psi_5(t) &=& \int_0^t \int_{\Omega\times\R^3} h \partial_t\psi\:dxdvds, \label{Psi5}
\\\Psi_6(t) &=& \int_0^t \int_{\Omega\times\R^3}\psi g \:dxdvds.\label{Psi6}
\end{eqnarray}

\par For each of the functions $a$, $b$ and $c$, we shall construct a $\psi$ such that the left-hand side of $\eqref{eqpsi}$ is exactly the $L^2_x$-norm of the function and the rest of the proof is estimating the six different terms $\Psi_i(t)$. Note that $\Psi_1(t)$ is already under the desired form 
\begin{equation}\label{Psi1all}
\Psi_1(t) = N_h(t)-N_h(0)
\end{equation}
with $\abs{N_h(s)} \leq C\norm{h}^2_{L^2_{x,v}}$ if $\psi(x,v)$ is in $L^2_{x,v}$ and its norm is controlled by the one of $h$ (which will be the case for our choices).

\begin{remark}
The linear perturbed equation $\eqref{lineqL2lem}$, the Maxwell boundary condition $\eqref{BCL2}$, and the conservation of mass are invariant under standard time mollification. We therefore consider for simplicity in the rest of the proof that all functions are smooth in the variable $t$. Exactly the same estimates can be derived for more general functions: study time mollified equation and then take the limit in the smoothing parameter. 
\end{remark}
For clarity, every positive constant not depending on $h$ will be denoted by $C_i$.


\bigskip
\textbf{Estimate for $a$.} By assumption $h$ is mass-preserving which is equivalent to
$$0 = \int_{\Omega\times\R^3}h(t,x,v)\sqrt{\mu(v)}\:dxdv = \int_\Omega a(t,x)\:dx.$$
We can thus choose the following test function
$$\psi_a(t,x,v) = \pa{\abs{v}^2-\alpha_a}\sqrt{\mu}v\cdot\nabla_x\phi_a(t,x)$$
where
$$-\Delta_x\phi_a(t,x) = a(t,x) \quad\mbox{and}\quad \restr{\partial_n\phi_a}{\partial\Omega}=0,$$
and $\alpha_a>0$ is chosen such that for all $1\leq i \leq 3$
$$\int_{\R^3} \pa{\abs{v}^2-\alpha_a}\frac{\abs{v}^2-3}{2}v_i^2\mu(v)\:dv = 0.$$
The differential operator $\partial_n$ denotes the tangential derivative at the boundary. The fact that the integral over $\Omega$ of $a(t,\cdot)$ is null allows us to use standard elliptic estimate \cite{Eva}:
\begin{equation}\label{phiaH2}
\forall t \geq 0,\quad\norm{\phi_a(t)}_{H^2_x}\leq C_0\norm{a(t)}_{L^2_x}.
\end{equation}
The latter estimate provides the control of $\Psi_1 = N^{(a)}_h(t) - N^{(a)}_h(0)$, as discussed before, and the control of $\eqref{Psi6}$, using Cauchy-Schwarz and Young's inequalities,
\begin{eqnarray}
 \abs{\Psi_6(t)} &\leq& C\int_0^t \norm{\phi_a}^2_{L^2_x}\norm{g}_{L^2_{x,v}}\:ds\nonumber
 \\&\leq& \frac{C_1}{4} \int_0^t\norm{a}_{L^2_x}\:ds + C_6\int_0^t \norm{g}^2_{L^2_{x,v}}\:ds,\label{Psi6a}
\end{eqnarray}
where $C_1>0$ is given by $\eqref{RHSa}$ below.

\bigskip
Firstly we compute the term on the right-hand side of $\eqref{eqpsi}$.
\begin{eqnarray*}
&&-\int_0^t \int_{\Omega\times\R^3} \pi_L(h) v\cdot\nabla_x\psi_a \:dxdvds 
\\&&\quad = -\sum\limits_{1\leq i,j \leq 3}\int_0^t \int_{\Omega} a(s,x)\pa{\int_{\R^3}\pa{\abs{v}^2-\alpha_a}v_iv_j\mu(v)\:dv}\partial_{x_i}\partial_{x_j}\phi_a(s,x)\:dxds
\\&&\quad\quad -\sum\limits_{1\leq i,j \leq 3}\int_0^t \int_{\Omega} b(s,x)\cdot \pa{\int_{\R^3}v\pa{\abs{v}^2-\alpha_a}v_iv_j\mu(v)\:dv}\partial_{x_i}\partial_{x_j}\phi_a(s,x)\:dxds
\\&&\quad\quad -\sum\limits_{1\leq i,j \leq 3}\int_0^t \int_{\Omega} c(s,x)\pa{\int_{\R^3}\pa{\abs{v}^2-\alpha_a}\frac{\abs{v}^2-3}{2}v_iv_j\mu(v)\:dv}\partial_{x_i}\partial_{x_j}\phi_a(s,x).
\end{eqnarray*}
By oddity the second term is null, as well as the first and last ones are when $i\neq j$. When $i=j$ in the last term we recover exactly our choice of $\alpha_a$ which makes the last term being null too. It only remains the first term when $i=j$
\begin{eqnarray}
-\int_0^t \int_{\Omega\times\R^3} \pi_L(h) v\cdot\nabla_x\psi_a \:dxdvds &=& -C_1 \int_0^t\int_{\Omega}a(s,x)\Delta_x \phi_a(s,x)\:dxds\nonumber
\\&=& C_1 \int_0^t\norm{a}^2_{L^2_x}\:ds. \label{RHSa}
\end{eqnarray}
Direct computations show $\alpha_a = 10$ and $C_1>0$.

\bigskip
We recall $L_\nu = -\nu(v)+K$ where $K$ is a bounded operator and that the $H^2_x$-norm of $\phi_a(t,x)$ is bounded by the $L^2_x$-norm of $a(t,x)$. For the terms $\Psi_2$ $\eqref{Psi2}$ and $\Psi_3$ $\eqref{Psi3}$ a mere Cauchy-Schwarz inequality yields
\begin{equation}\label{Psi23a}
\begin{split}
\forall i\in\br{2,3}, \quad \abs{\Psi_i(t)} &\leq C \int_0^t\norm{a}_{L^2_x}\norm{\pi_L^\bot(h)}_{L^2_{x,v}}\:ds
\\&\leq  \frac{C_1}{4}\int_0^t\norm{a}^2_{L^2_x}\:ds + C_2\int_0^t\norm{\pi_L^\bot(h)}^2_{L^2_{x,v}}\:ds.
\end{split}
\end{equation}
We used Young's inequality for the last inequality, with $C_1$ defined in $\eqref{RHSa}$.

\bigskip
The term $\Psi_4$ $\eqref{Psi4}$ deals with boundary so we decompose it into $\Lambda^+$ and $\Lambda^-$. In the $\Lambda^-$ integral we apply the Maxwell boundary condition satisfied by $h$ and use the change of variable $v \mapsto \mathcal{R}_x(v)$. Since $\abs{v}^2$, $\mu(v)$, $\phi_a(s,x)$, the specular part and $P_{\Lambda_\mu}$ $\eqref{PLambdamu}$ are invariant by this isometric change of variable we get
\begin{equation*}
\begin{split}
\Psi_4(t) =& -\int_0^t \int_{\Lambda^+} h \pa{\abs{v}^2-\alpha_a}\abs{v\cdot n(x)}\nabla_x\phi_a(s,x)\cdot v \sqrt{\mu}\:dS(x)dvds 
\\& + (1-\alpha) \int_0^t\int_{\Lambda^+} h\pa{\abs{v}^2-\alpha_a}\abs{v\cdot n(x)}\nabla_x\phi_a\cdot \mathcal{R}_x(v)\sqrt{\mu}\:dS(x)dvds 
\\& + \alpha\int_0^t\int_{\Lambda^+} P_{\Lambda_\mu}(h)\pa{\abs{v}^2-\alpha_a}\abs{v\cdot n(x)}\nabla_x\phi_a\cdot \mathcal{R}_x(v) \sqrt{\mu}\:dS(x)dvds.
\end{split}
\end{equation*}
so
\begin{equation}\label{startPsi4a}
\begin{split}
\Psi_4(t)=& -(1-\alpha) \int_0^t \int_{\Lambda^+} h \pa{\abs{v}^2-\alpha_a}\abs{v\cdot n(x)}\nabla_x\phi_a\cdot \cro{v-\mathcal{R}_x(v)} \sqrt{\mu}\:dS(x)dvds 
\\& - \alpha\int_0^t\int_{\Lambda^+}\pa{\abs{v}^2-\alpha_a}\abs{v\cdot n(x)}\nabla_x\phi_a\cdot \cro{vh - \mathcal{R}_x(v)P_{\Lambda_\mu}(h)}\sqrt{\mu} 
\end{split}
\end{equation}
By definition of the specular reflection and the tangential derivative
\begin{eqnarray*}
\abs{v\cdot n(x)}\nabla_x\phi_a(s,x)\cdot\pa{v-\mathcal{R}_x(v)} &=& 2\pa{v\cdot n(x)}n\cdot\nabla_x\phi_a(s,x)
\\&=& 2\pa{v\cdot n(x)}\partial_n\phi_a(s,x).
\end{eqnarray*}
The contribution of the specular reflection part is therefore null since $\phi_a$ was chosen such that $\restr{\partial_n\phi_a}{\partial\Omega}=0$. For the diffusive part we compute
$$vh - \mathcal{R}_x(v)P_{\Lambda_\mu}(h) = vP_{\Lambda_\mu}^\bot(h) + 2P_{\Lambda_\mu}(h)\pa{v\cdot n(x)}n(x)$$
and again the term in the direction of $n(x)$ gives a zero contribution since $\restr{\partial_n\phi_a}{\partial\Omega}=0$. It only remains
$$\Psi_4(t)= - \alpha\int_0^t\int_{\Lambda^+}\cro{\pa{\abs{v}^2-\alpha_a}\abs{v\cdot n(x)}\sqrt{\mu}\:v\cdot\nabla_x\phi_a} P_{\Lambda_\mu}^\bot(h)\:dS(x)dvds.$$
We apply Cauchy-Schwarz inequality and the control on the $H^2$ norm of $\phi_a$ to finally obtain the following estimate
\begin{equation}\label{Psi4a}
\begin{split}
\abs{\Psi_4(t)} &\leq C \int_0^t\norm{a}_{L^2_x}\norm{P_{\Lambda_\mu}^\bot(h)}_{L^2_{\Lambda^+}}\:ds
\\&\leq  \frac{C_1}{4}\int_0^t\norm{a}^2_{L^2_x}\:ds + C_4\int_0^t\norm{P_{\Lambda_\mu}^\bot(h)}^2_{L^2_{\Lambda^+}}\:ds,
\end{split}
\end{equation}
where we used Young's inequality with $C_1$ defined in $\eqref{RHSa}$.

\bigskip
It remains to estimate the term with time derivatives $\eqref{Psi5}$. It reads
\begin{eqnarray*}
\Psi_5(t) &=&\int_0^t \int_{\Omega\times\R^3}h\pa{\abs{v}^2-\alpha_a} v\cdot\cro{\partial_t\nabla_x\phi_a}\sqrt{\mu}\:dxdvds
\\&=& \sum\limits_{i=1}^3\int_0^t \int_{\Omega\times\R^3}\pi_L(h)\pa{\abs{v}^2-\alpha_a} v_i\sqrt{\mu}\:\partial_t\partial_{x_i}\phi_a\:dxdvds 
\\ &\:&+ \int_0^t \int_{\Omega\times\R^3}\pi_L^\bot(h)\pa{\abs{v}^2-\alpha_a} \sqrt{\mu}\:v\cdot\cro{\partial_t\nabla_x\phi_a}\:dxdvds
\end{eqnarray*}
Using oddity properties for the first integral on the right-hand side and then Cauchy-Schwarz and the following bound
$$\int_{\R^3}\pa{\abs{v}^2-\alpha_a}^2 \abs{v}^2 \mu(v)\:dv <+\infty$$
we get
\begin{equation}\label{Psi5astart}
\abs{\Psi_5(t)}\leq C\int_0^t\cro{\norm{b}_{L^2_x}+\norm{\pi_L^\bot(h)}_{L^2_{x,v}}}\norm{\partial_t\nabla_x\phi_a}_{L^2_x}\:ds.
\end{equation}

\par The estimation on $\norm{\partial_t\nabla_x\phi_a}_{L^2_x}$ will come from elliptic estimates in negative Sobolev spaces. We use the decomposition of the weak formulation $\eqref{eqpsi}$ between $t$ and $t+\eps$ (instead of between $0$ and $t$) with $\psi(t,x,v) = \phi(x)\sqrt{\mu} \in H^1_x$ with the integral of $\phi$ on $\Omega$ being zero. $\psi(x)\mu(v)$ and $v\psi(x)\mu(v)$ are in $\mbox{Ker}(L_\mu)$ and therefore are orthogonal to $\pi_L(h)$ and $L_\mu[h]$. Moreover, $\psi$ does not depend on time. Hence,
$$\Psi_2(t)=\Psi_3(t)=\Psi_5(t)=0.$$
At last, with the same computations as before the boundary term is
$$\Psi_4 = -\alpha\int_0^t\int_{\partial\Omega}\phi(x)\int_{v\cdot n(x)>0}\pa{\mbox{Id} - P_{\Lambda_\mu}}(h)\abs{v\cdot n(x)}\sqrt{\mu}\:dS(x)dvds = 0.$$
The weak formulation associated to $\phi(x)\sqrt{\mu}$ is therefore
\begin{equation*}
\begin{split}
&\int_{\Omega\times\R^3} \phi(x) h(t+\eps)\sqrt{\mu}\:dxdv -\int_{\Omega\times\R^3} \phi(x) h(t)\sqrt{\mu}\:dxdv 
\\&\quad\quad\quad=  \int_{t}^{t+\eps} \int_{\Omega\times\R^3}\cro{\pi_L(h)v\cdot\nabla_x\phi(x)+ g\phi(x)}\sqrt{\mu}\:dxdvds,
\end{split}
\end{equation*}
which is equal to
\begin{equation*}
\begin{split}
&\int_{\Omega} \cro{a(t+\eps)-a(t)}\phi(x)\:dx 
\\&\quad\quad\quad= C\int_t^{t+\eps}\cro{\int_{\Omega} b(s,x)\cdot \nabla_x\phi(x)\:dxds+ \int_{\Omega\times\R^3}g\phi\sqrt{\mu}\:dxdv}ds.
\end{split}
\end{equation*}
Dividing by $\eps$ and taking the limit as $\eps$ goes to $0$ yields the following estimates, thanks to a Cauchy-Schwarz inequality,
$$\int_{\Omega}\partial_ta(t,x)\phi(x)\:dx \leq C \cro{\norm{b(t)}_{L^2_x}\norm{\nabla_x\phi}_{L^2_x} + \norm{g}_{L^2_{x,v}}\norm{\phi}_{L^2_x}}.$$
Since $\phi$ has a null integral on $\Omega$ we can apply Poincar\'e inequality.
$$\int_{\Omega}\partial_ta(t,x)\phi(x)\:dx \leq C \cro{\norm{b(t)}_{L^2_x} + \norm{g}_{L^2_{x,v}}}\norm{\nabla_x\phi}_{L^2_x}.$$
The latter inequality is true for all $\phi$ in $\mathcal{H}^1_x$ the set of functions in $H^1_x$ with a null integral. Therefore, for all $t\geq 0$
\begin{equation}\label{H1*a}
\norm{\partial_ta(t,x)}_{\pa{\mathcal{H}^1_x}^*} \leq C \cro{\norm{b(t)}_{L^2_x} + \norm{g}_{L^2_{x,v}}}
\end{equation}
where $\pa{\mathcal{H}^1_x}^*$ is the dual of $\mathcal{H}^1_x$.
\par We fix $t$ and thanks to the conservation of mass we have that the integral of $\partial_t a$ is null on $\Omega$. We can construct $\phi(t,x)$ such that
$$-\Delta_x\phi(t,x) = \partial_ta(t,x) \quad\mbox{and}\quad \restr{\partial_n\phi}{\partial\Omega}=0.$$

and by standard elliptic estimates \cite{Eva} and $\eqref{H1*a}$:
$$\norm{\phi}_{\mathcal{H}^1_x} \leq \norm{\partial_t a}_{\pa{\mathcal{H}^1_x}^*}\leq C \cro{\norm{b(t)}_{L^2_x} + \norm{g}_{L^2_{x,v}}}.$$

\par Combining this estimate with 
$$\norm{\partial_t\nabla_x\phi_a}_{L^2_x} = \norm{\nabla_x\Delta^{-1}\partial_t a}_{L^2_x} \leq  \norm{\Delta^{-1}\partial_t a}_{H^1_x} = \norm{\phi}_{\mathcal{H}^1_x}$$
we can further control $\Psi_5$ in $\eqref{Psi5astart}$
\begin{equation}\label{Psi5a}
\abs{\Psi_5(t)} \leq C_5\int_0^t\cro{\norm{b}^2_{L^2_x} + \norm{\pi_L^\bot(h)}^2_{L^2_{x,v}}+\norm{g}^2_{L^2_x}}\:ds.
\end{equation}

\bigskip
We now gather $\eqref{RHSa}$, $\eqref{Psi1all}$, $\eqref{Psi23a}$, $\eqref{Psi4a}$, $\eqref{Psi5a}$ and $\eqref{Psi6a}$ into $\eqref{eqpsi}$
\begin{equation}\label{afinal}
\begin{split}
\int_0^t \norm{a}^2_{L^2_x}\:ds \leq& N^{(a)}_h(t)-N^{(a)}_h(0) +C_{a,b}\int_0^t \norm{b}^2_{L^2_x}\:ds
\\&+ C_a\int_0^t \cro{\norm{P_{\Lambda_\mu}^\bot(h)}^2_{L^2_{\Lambda^+}}+\norm{\pi_L^\bot(h)}^2_{L^2_{x,v}} + \norm{g}^2_{L^2_{x,v}}}\:ds.
\end{split}
\end{equation}


\bigskip
\textbf{Estimate for $b$.} The choice of function to integrate against to deal with the $b$ term is more involved.
\par We emphasize that $b(t,x)$ is a vector $\pa{b_1(t,x),b_2(t,x),b_3(t,x)}$. Fix $J$ in $\br{1,2,3}$ and define
$$\psi_J(t,x,v) = \sum\limits_{i=1}^3\varphi^{(J)}_i(t,x,v),$$
with
$$\varphi^{(J)}_i(t,x,v) = \left\{\begin{array}{l} \disp{\abs{v}^2v_iv_J\sqrt{\mu}\partial_{x_i}\phi_J(t,x) - \frac{7}{2}\pa{v_i^2-1}\sqrt{\mu}\partial_{x_J}\phi_J(t,x), \quad\mbox{if}\:\:i\neq J} \vspace{2mm}\\\vspace{2mm} \disp{\frac{7}{2}\pa{v_J^2-1}\sqrt{\mu}\partial_{x_J}\phi_J(t,x), \quad\mbox{if}\:\: i=J.} \end{array}\right.$$
where
$$-\Delta_x\phi_J(t,x) = b_J(t,x) \quad\mbox{and}\quad \restr{\phi_J}{\partial\Omega}=0.$$
Since it will be important we emphasize here that for all $i \neq k$
\begin{equation}\label{contribution0}
\int_{\R^3}\pa{v_i^2-1}\mu(v)\:dv =0 \quad\mbox{and}\quad \int_{\R^3}\pa{v_i^2-1}v_k^2\mu(v)\:dv =0.
\end{equation}
The vanishing of $\phi_J$ at the boundary implies, by standard elliptic estimate \cite{Eva},
\begin{equation}\label{phibH2}
\forall t \geq 0,\quad\norm{\phi_J(t)}_{H^2_x}\leq C_0\norm{b_J(t)}_{L^2_x}.
\end{equation}
Again, this estimate provides the control of $\Psi_1 = N^{(J)}_h(t) - N^{(J)}_h(0)$ and of $\Psi_6(t)$ as in $\eqref{Psi6a}$:
\begin{equation}\label{Psi6b}
 \abs{\Psi_6(t)} \leq \frac{7}{4} \int_0^t\norm{b_J}^2_{L^2_x}\:ds + C_6\int_0^t \norm{g}^2_{L^2_{x,v}}\:ds.
\end{equation}

\bigskip
We start by the right-hand side of $\eqref{eqpsi}$. By oddity, there is neither contribution from $a(s,x)$ nor from $c(s,x)$. Hence,
\begin{eqnarray*}
&&-\int_0^t \int_{\Omega\times\R^3} \pi_L(h) v\cdot\nabla_x\psi_J \:dxdvds 
\\&& = -\sum\limits_{1\leq j,k\leq 3}\sum\limits_{\overset{i=1}{i\neq J}}^3\int_0^t \int_{\Omega} b_k(s,x)\pa{\int_{\R^3}\abs{v^2}v_kv_iv_jv_J\mu(v)\:dv}\partial_{x_j}\partial_{x_i}\phi_J(s,x)\:dxds
\\&&\quad + \frac{7}{2}\sum\limits_{1\leq j,k\leq 3}\sum\limits_{\overset{i=1}{i\neq J}}^3\int_0^t \int_{\Omega} b_k(s,x) \pa{\int_{\R^3}\pa{v_i^2-1}v_kv_j\mu(v)\:dv}\partial_{x_j}\partial_{x_J}\phi_J(s,x)\:dxds
\\&&\quad -\frac{7}{2}\sum\limits_{1\leq j,k\leq 3}\int_0^t \int_{\Omega} b_k(s,x)\pa{\int_{\R^3}\pa{v_J^2-1}v_jv_k\mu(v)\:dv}\partial_{x_j}\partial_{x_J}\phi_J(s,x)\:dxds.
\end{eqnarray*}
The last two integrals on $\R^3$ are zero if $j\neq k$. Moreover, when $j=k$ and $j\neq J$ it is also zero by $\eqref{contribution0}$. We compute directly for $j= J$
$$\int_{\R^3}\pa{v_J^2-1}v_J^2\mu(v)\:dv = 2.$$
The first term is composed by integrals in $v$ of the form
$$\int_{\R^3}\abs{v}^2v_kv_iv_jv_J\mu(v)\:dv$$
which are always null unless two indices are equals to the other two. Therefore if $i=j$ then $k=J$ and if $i\neq j$ we only have two options: $k=i$ and $j=J$ or $k=j$ and $i=J$. Hence,
\begin{eqnarray*}
&&-\int_0^t \int_{\Omega\times\R^3} \pi_L(h) v\cdot\nabla_x\psi_J \:dxdvds
\\&& =  -\sum\limits_{\overset{i=1}{i\neq J}}^3\int_0^t\int_{\Omega}b_J(s,x)\partial_{x_ix_i}\phi_J \pa{\int_{\R^3}\abs{v}^2v_i^2v_J^2\mu(v)\:dv}dxds
\\&&\quad -\sum\limits_{\overset{i=1}{i\neq J}}^3\int_0^t\int_{\Omega}b_i(s,x)\partial_{x_ix_J}\phi_J \pa{\int_{\R^3}\abs{v}^2v_i^2v_J^2\mu(v)\:dv}dxds
\\&&\quad + 7 \sum\limits_{\overset{i=1}{i\neq J}}^3\int_0^t\int_{\Omega}b_i(s,x)\partial_{x_ix_J}\phi_J\:dxds
- 7\int_0^t \int_{\Omega} b_J(s,x)\partial_{x_J}\partial_{x_J}\phi_J(s,x)\:dxds.
\end{eqnarray*} 
To conclude we compute $\int_{\R^3}\abs{v^2}v_i^2v_J^2 \mu(v)\:dv = 7$ whenever $i\neq J$ and it thus only remains the following equality 
\begin{eqnarray}
-\int_0^t \int_{\Omega\times\R^3} \pi_L(h) v\cdot\nabla_x\psi_a \:dxdvds &=& -7 \int_0^t\int_{\Omega}b_J(s,x)\Delta_x \phi_J(s,x)\:dxds\nonumber
\\&=& 7 \int_0^t\norm{b_J}^2_{L^2_x}\:ds. \label{RHSb}
\end{eqnarray}

\bigskip
Then the term $\Psi_2$ and $\Psi_3$ are dealt with as in $\eqref{Psi23a}$
\begin{equation}\label{Psi23b}
\forall i\in\br{2,3}, \quad \abs{\Psi_i(t)} \leq   \frac{7}{4}\int_0^t\norm{b_J}^2_{L^2_x}\:ds + C_2\int_0^t\norm{\pi_L^\bot(h)}^2_{L^2_{x,v}}\:ds.
\end{equation}

\bigskip
The boundary term $\Psi_4$ is divided into $\Lambda^+$ and $\Lambda^-$, we apply the Maxwell boundary condition $\eqref{BCL2}$ and the change of variable $v\mapsto \mathcal{R}_x(v)$ on the $\Lambda^-$ part
\begin{equation*}
\begin{split}
\Psi_4(t) =& -\sum\limits_{i=1}^3\int_0^t \int_{\Lambda^+} h \abs{v\cdot n(x)}\varphi^{(J)}_i(s,x,v) \:dS(x)dvds 
\\& + (1-\alpha) \sum\limits_{i=1}^3\int_0^t\int_{\Lambda^+} h\abs{v\cdot n(x)}\varphi^{(J)}_i(s,x,\mathcal{R}_x(v))\:dS(x)dvds 
\\& + \alpha\sum\limits_{i=1}^3\int_0^t\int_{\Lambda^+} P_{\Lambda_\mu}(h)\abs{v\cdot n(x)}\varphi^{(J)}_i(s,x,\mathcal{R}_x(v))\:dS(x)dvds. 
\end{split}
\end{equation*}
We decompose $h = P_{\Lambda_\mu}(h) + P_{\Lambda_\mu}^\bot$ to obtain
\begin{equation}\label{Psi4bstart}
\begin{split}
\Psi_4(t)=&  -\sum\limits_{i=1}^3\int_0^t \int_{\Lambda^+} P_{\Lambda_\mu}(h) \pa{v\cdot n(x)}\cro{\varphi^{(J)}_i(s,x,v)-\varphi^{(J)}_i(s,x,\mathcal{R}_x(v))}\:dS(x)dvds 
\\& - \sum\limits_{i=1}^3\int_0^t\int_{\Lambda^+}P_{\Lambda_\mu}^\bot(h)\abs{v\cdot n(x)}
\\&\quad\quad\quad\quad\quad\quad\quad\times\cro{\varphi^{(J)}_i(s,x,v)-(1-\alpha)\varphi^{(J)}_i(s,x,\mathcal{R}_x(v))}\:dS(x)dvds.
\end{split}
\end{equation}
We apply Cauchy-Schwarz inequality and the elliptic estimate on $\phi_J$ $\eqref{phibH2}$ to the second integral obtain the following estimate
\begin{equation}\label{1stboundaryb}
\begin{split}
&\abs{- \sum\limits_{i=1}^3\int_0^t\int_{\Lambda^+}P_{\Lambda_\mu}^\bot(h)\abs{v\cdot n(x)}\cro{\varphi^{(J)}_i(s,x,v)-(1-\alpha)\varphi^{(J)}_i(s,x,\mathcal{R}_x(v))}\:dS(x)dvds}
\\&\quad\quad\quad\quad\quad\leq C \int_0^t\norm{b_J}\norm{P_{\Lambda_\mu}^\bot(h)}_{L^2_{\Lambda^+}}\:ds
\\&\quad\quad\quad\quad\quad\leq  \frac{7}{4}\int_0^t\norm{b_J}^2_{L^2_x}\:ds + C_4\int_0^t\norm{\pi_L^\bot(h)}^2_{L^2_{\Lambda^+}}\:ds,
\end{split}
\end{equation}
where we also used Young's inequality.

\par The term involving $P_{\Lambda_\mu}(h)$ in $\eqref{Psi4bstart}$ is computed directly by a change of variable $v\mapsto \mathcal{R}_x(v)$ to come back to the full boundary $\Lambda$ and the property $\eqref{PLambdamu}$ that is
$$P_{\Lambda_\mu}(h)(s,x,v) = z(s,x)\sqrt{\mu(v)}.$$
We also have $\varphi^{(J)}_i$ in the following form
$$\varphi^{(J)}_i(t,x,v) = \tilde{\varphi}^{(J)}_i(v)\sqrt{\mu(v)} \partial\phi_J(t,x)$$
where $\partial_i$ begin a certain derivative in $x$ and $\tilde{\varphi}^{(J)}_i$ is an even function. We thus get
\begin{equation*}
\begin{split}
&\int_0^t \int_{\Lambda^+} P_{\Lambda_\mu}(h) \pa{v\cdot n(x)}\cro{\varphi^{(J)}_i(s,x,v)-\varphi^{(J)}_i(s,x,\mathcal{R}_x(v))}\:dS(x)dvds
\\&\quad\quad\quad= \int_0^t \int_{\Lambda} P_{\Lambda_\mu}(h) \pa{v\cdot n(x)}\varphi^{(J)}_i(s,x,v)\:dS(x)dvds
\\&\quad\quad\quad = \sum\limits_{k=1}^3\int_0^t\int_{\Omega}z(s,x)n_k(x)\partial_i\phi_J(s,x)\pa{\int_{\R^3} \tilde{\varphi}^{(J)}_i(v)v_k\mu(v)\:dv}\:dS(x)ds
\\&\quad\quad\quad = 0,
\end{split}
\end{equation*}
by oddity. Combining the latter with $\eqref{1stboundaryb}$ inside $\eqref{Psi4bstart}$ yields
\begin{equation}\label{Psi4b}
\abs{\Psi_4(t)}\leq  \frac{C_1}{4}\int_0^t\norm{b_J}^2_{L^2_x}\:ds + C_4\int_0^t\norm{P_{\Lambda_\mu}^\bot(h)}^2_{L^2_{\Lambda^+}}\:ds.
\end{equation}

\bigskip
It remains to estimate $\Psi_5$ which involves time derivative $\eqref{Psi5}$:
\begin{eqnarray*}
\Psi_5(t) &=&\sum\limits_{i=1}^3\int_0^t \int_{\Omega\times\R^3}h \partial_t\varphi^{(J)}_i(s,x,v)\:dxdvds
\\&=& \sum\limits_{i=1}^3\int_0^t \int_{\Omega\times\R^3}\pi_L^\bot(h) \partial_t\varphi^{(J)}_i(s,x,v)\:dxdvds
\\&\:& +\sum\limits_{\overset{i=1}{i\neq J}}^3\int_0^t \int_{\Omega\times\R^3}\pi_L(h)\abs{v}^2v_iv_J\sqrt{\mu}\partial_{x_i}\phi_J\:dxdvds
\\&\:& +\sum\limits_{i=1}^3 \pm \frac{7}{2}\int_0^t \int_{\Omega\times\R^3} \pi_L(h)\pa{v_i^2-1}\sqrt{\mu}\partial_{x_J}\phi_J\:dxdvds.
\end{eqnarray*}
By oddity arguments, only terms in $a(s,x)$ and $c(s,x)$ can contribute to the last two terms on the right-hand side. However, $i\neq J$ implies that the second term is zero as well as the contribution of $a(s,x)$ in the third term thanks to $\eqref{contribution0}$. Finally, a Cauchy-Schwarz inequality on both integrals yields as in $\eqref{Psi5astart}$
\begin{equation}\label{Psi5bstart}
\abs{\Psi_5(t)}\leq C\int_0^t\cro{\norm{c}_{L^2_x}+\norm{\pi_L^\bot(h)}_{L^2_{x,v}}}\norm{\partial_t\nabla_x\phi_J}_{L^2_x}\:ds.
\end{equation}

\par To estimate $\norm{\partial_t\nabla_x\phi_J}_{L^2_x}$ we follow the idea developed for $a(s,x)$ about negative Sobolev regularity. We apply the weak formulation $\eqref{eqpsi}$ to a specific function between $t$ and $t+\eps$. The test function is $\psi(x,v) = \phi(x)v_J\sqrt{\mu}$ with $\phi$ in $H^1_x$ and null on the boundary. Note that $\psi$ does not depend on $t$, vanishes at the boundary and belongs to $\mbox{Ker}(L)$. Hence,
$$\Psi_3(t)=\Psi_4 = \Psi_5(t)=0.$$
It remains
\begin{eqnarray*}
C\int_{\Omega} \cro{b_J(t+\eps)-b_J(t)}\phi(x)\:dx &=& \int_{t}^{t+\eps} \int_{\Omega\times\R^3}\pi_L(h)v_Jv\cdot\nabla_x\phi(x)\sqrt{\mu}\:dxdvds
\\&\:& + \int_{t}^{t+\eps} \int_{\Omega\times\R^3}\pi_L^\bot(h)v_Jv\cdot\nabla_x\phi(x)\sqrt{\mu}\:dxdvds
\\&\:& + \int_{t}^{t+\eps} \int_{\Omega\times\R^3}g\phi(x)v_J\sqrt{\mu}\:dxdvds.
\end{eqnarray*}
As for $a(t,x)$ we divide by $\eps$ and take the limit as $\eps$ goes to $0$. By oddity, the first integral on the right-hand side only gives terms with $a(s,x)$ and $c(s,x)$. The second term is dealt with by a Cauchy-Schwarz inequality. Finally, we apply a Cauchy-Schwarz inequality for the last integral with a Poincar\'e inequality for $\phi(x)$ ($\phi$ is null on the boundary). This yields
\begin{equation}\label{negativesobolevb}
\abs{\int_{\Omega} \partial_t b_J(t,x)\phi(x)\:dx} \leq C \cro{\norm{a}_{L^2_x}+\norm{c}_{L^2_x}+\norm{\pi_L^\bot(h)}_{L^2_{x,v}}+\norm{g}_{L^2_x}}\norm{\nabla_x\phi}_{L^2_x}.
\end{equation}

\par The latter is true for all $\phi(x)$ in $H^1_x$ vanishing on the boundary. We thus fix $t$ and apply the inequality above to 
$$-\Delta_x\phi(t,x) = \partial_tb_J(t,x) \quad\mbox{and}\quad \restr{\phi}{\partial\Omega}=0,$$
and obtain
$$\norm{\partial_t\nabla_x\phi_J}_{L^2_x}^2 = \norm{\nabla_x\Delta^{-1}\partial_t b_J}_{L^2_x}^2 = \int_{\Omega}\pa{\nabla_x\Delta^{-1}\partial_t b_J}\nabla_x\phi(x)\:dx.$$
We integrate by parts (the boundary term vanishes because of our choice of $\phi$).
$$\norm{\partial_t\nabla_x\phi_J}_{L^2_x}^2 = \int_{\Omega} \partial_t b_J(t,x)\phi(x)\:dx$$
At last, we use $\eqref{negativesobolevb}$
\begin{equation*}
\begin{split}
\norm{\partial_t\nabla_x\phi_J}_{L^2_x}^2 &\leq C \cro{\norm{a}_{L^2_x}+\norm{c}_{L^2_x}+\norm{\pi_L^\bot(h)}_{L^2_{x,v}}+\norm{g}_{L^2_x}}\norm{\nabla_x\phi}_{L^2_x} 
\\&= C\cro{\norm{a}_{L^2_x}+\norm{c}_{L^2_x}+\norm{\pi_L^\bot(h)}_{L^2_{x,v}}+\norm{g}_{L^2_x}}\norm{\nabla_x \Delta_x^{-1}\partial_t b_J}_{L^2_x}
\\&= C\cro{\norm{a}_{L^2_x}+\norm{c}_{L^2_x}+\norm{\pi_L^\bot(h)}_{L^2_{x,v}}+\norm{g}_{L^2_x}}\norm{\partial_t\nabla_x \phi_J}_{L^2_x}
\end{split}
\end{equation*}

\par Combining this estimate with $\eqref{Psi5bstart}$ and using Young's inequality with any $\eps_b>0$
\begin{equation}\label{Psi5b}
\abs{\Psi_5(t)} \leq \eps_b \int_0^t\norm{a}^2_{L^2_x}\:ds + C_5(\eps_b)\int_0^t\cro{ \norm{c}^2_{L^2_x}+ \norm{\pi_L^\bot(h)}^2_{L^2_{x,v}}+\norm{g}^2_{L^2_x}}\:ds. 
\end{equation}

\bigskip
We now gather $\eqref{RHSb}$, $\eqref{Psi1all}$, $\eqref{Psi23b}$, $\eqref{Psi4b}$ and $\eqref{Psi5b}$
\begin{equation*}
\begin{split}
\int_0^t \norm{b_J}^2_{L^2_x}\:ds \leq & N^{(J)}_h(t)-N^{(J)}_h(0) + \eps_b\int_0^t \norm{a}^2_{L^2_x}\:ds+ C_{J,c}(\eps_b)\int_0^t \norm{c}^2_{L^2_x}\:ds
\\&+C_{J}(\eps_b)\int_0^t \cro{\norm{P_{\Lambda_\mu}^\bot(h)}^2_{L^2_{\Lambda^+}}+\norm{\pi_L^\bot(h)}^2_{L^2_{x,v}}+\norm{g}^2_{L^2_x}}\:ds.
\end{split}
\end{equation*}
Finally, summing over all $J$ in $\br{1,2,3}$
\begin{equation}\label{bfinal}
\begin{split}
\int_0^t \norm{b}^2_{L^2_x}\:ds \leq & N^{(b)}_h(t)-N^{(b)}_h(0) + \eps_b\int_0^t \norm{a}^2_{L^2_x}\:ds+ C_{b,c}(\eps_b)\int_0^t \norm{c}^2_{L^2_x}\:ds
\\&+C_b(\eps_b)\int_0^t \cro{\norm{P_{\Lambda_\mu}^\bot(h)}^2_{L^2_{\Lambda^+}}+\norm{\pi_L^\bot(h)}^2_{L^2_{x,v}}+\norm{g}^2_{L^2_x}}\:ds.
\end{split}
\end{equation}


\bigskip
\textbf{Estimate for $c$.} The handling of $c(t,x)$ is quite similar to the one of $a(t,x)$ but it involves a more intricate treatment of the boundary terms as $h$ does not preserves the energy. We choose the following test function
$$\psi_c(t,x,v) = \pa{\abs{v}^2-\alpha_c}v\cdot\nabla_x\phi_c(t,x)\sqrt{\mu(v)}$$
where
$$-\Delta_x\phi_c(t,x) = c(t,x) \quad\mbox{and}\quad \restr{\phi_c}{\partial\Omega}=0,$$
and $\alpha_c>0$ is chosen such that for all $1\leq i \leq 3$
$$\int_{\R^3} \pa{\abs{v}^2-\alpha_c}v_i^2\mu(v)\:dv = 0.$$
The vanishing of $\phi_c$ at the boundary implies, by standard elliptic estimate \cite{Eva},
\begin{equation}\label{phicH2}
\forall t \geq 0,\quad\norm{\phi_c(t)}_{H^2_x}\leq C_0\norm{c(t)}_{L^2_x}.
\end{equation}
Again, this estimate provides the control of $\Psi_1 = N^{(c)}_h(t) - N^{(c)}_h(0)$ and of $\Psi_6(t)$ as in $\eqref{Psi6a}$:
\begin{equation}\label{Psi6c}
 \abs{\Psi_6(t)} \leq \frac{C_1}{4} \int_0^t\norm{c}^2_{L^2_x}\:ds + C_6\int_0^t \norm{g}^2_{L^2_{x,v}}\:ds,
\end{equation}
where $C_1$ is given by $\eqref{RHSc}$ below.

\bigskip
We start by the right-hand side of $\eqref{eqpsi}$.
\begin{eqnarray*}
&&-\int_0^t \int_{\Omega\times\R^3} \pi_L(h) v\cdot\nabla_x\psi_c \:dxdvds 
\\&&\quad = -\sum\limits_{1\leq i,j \leq 3}\int_0^t \int_{\Omega} a(s,x)\pa{\int_{\R^3}\pa{\abs{v}^2-\alpha_c}v_iv_j\mu(v)\:dv}\partial_{x_i}\partial_{x_j}\phi_c(s,x)\:dxds
\\&&\quad\quad -\sum\limits_{1\leq i,j \leq 3}\int_0^t \int_{\Omega} b(s,x)\cdot \pa{\int_{\R^3}v\pa{\abs{v}^2-\alpha_c}v_iv_j\mu(v)\:dv}\partial_{x_i}\partial_{x_j}\phi_c(s,x)\:dxds
\\&&\quad\quad -\sum\limits_{1\leq i,j \leq 3}\int_0^t \int_{\Omega} c(s,x)\pa{\int_{\R^3}\pa{\abs{v}^2-\alpha_c}\frac{\abs{v}^2-3}{2}v_iv_j\mu(v)\:dv}\partial_{x_i}\partial_{x_j}\phi_c(s,x).
\end{eqnarray*}
By oddity, the second integral vanishes, as well as all the others if $i\neq j$. Our choice of $\alpha_c$ makes the first integral vanish even for $i=j$. It only remains the last integral with terms $i=j$ and therefore the definition of $\Delta_x\phi_c(t,x)$ gives
\begin{equation}\label{RHSc}
-\int_0^t \int_{\Omega\times\R^3} \pi_L(h) v\cdot\nabla_x\psi_c \:dxdvds = C_1 \int_0^t\norm{c}^2_{L^2_x}\:ds. 
\end{equation}
Again, direct computations show $\alpha_c = 5$ and hence $C_1>0$.

\bigskip
Then the term $\Psi_2$ and $\Psi_3$ are dealt with as for $a(t,x)$ and $b(t,x)$.
\begin{equation}\label{Psi23c}
\forall i\in\br{2,3}, \quad \abs{\Psi_i(t)} \leq   \frac{C_1}{4}\int_0^t\norm{c}^2_{L^2_x}\:ds + C_2\int_0^t\norm{\pi_L^\bot(h)}^2_{L^2_{x,v}}\:ds,
\end{equation}
where $C_1$ is defined in $\eqref{RHSc}$.

\bigskip
The term $\Psi_4$ involves integral on the boundary $\Lambda$. Again, we divide it into $\Lambda^+$ and $\Lambda^-$, we use the Maxwell boundary condition $\eqref{BCL2}$ satisfied by $h$ and we make the change of variable $v\mapsto \mathcal{R}_x(v)$ on $\Lambda^-$. As for $\eqref{startPsi4a}$ dealing with $a(t,x)$ we obtain
\begin{equation*}
\begin{split}
\Psi_4(t) &= -2(1-\alpha) \int_0^t \int_{\Lambda^+} h \pa{\abs{v}^2-\alpha_c}\pa{v\cdot n(x)}^2\partial_n\phi_c\sqrt{\mu}\:dS(x)dvds 
\\& - \alpha\int_0^t\int_{\Lambda^+}\pa{\abs{v}^2-\alpha_c}\abs{v\cdot n}\nabla_x\phi_c\cdot \cro{vP_{\Lambda_\mu}^\bot(h) + 2P_{\Lambda_\mu}(h)\pa{v\cdot n)n}}\sqrt{\mu}.
\end{split}
\end{equation*}
We decompose $h = P_{\Lambda_\mu}(h) + P_{\Lambda_\mu}^\bot(h)$ in the first integral and use $\eqref{PLambdamu}$ which says that $P_{\Lambda_\mu}(h)(t,x,v)=z(t,x)\sqrt{\mu(v)}$.
\begin{equation*}
\begin{split}
\Psi_4(t) =& -2\int_0^t \int_{\partial\Omega} \partial_n\phi_c z(s,x)\pa{\int_{v\cdot n(x)>0}\pa{\abs{v}^2-\alpha_c}\pa{v\cdot n(x)}^2\mu(v)\:dv}\:dS(x)ds 
\\&-2(1-\alpha) \int_0^t \int_{\Lambda^+}  \cro{\pa{\abs{v}^2-\alpha_c}\pa{v\cdot n(x)}^2\partial_n\phi_c\sqrt{\mu}}P_{\Lambda_\mu}^\bot(h)\:dS(x)dvds
\\& - \alpha\int_0^t\int_{\Lambda^+}\cro{\pa{\abs{v}^2-\alpha_c}\abs{v\cdot n}\nabla_x\phi_c\cdot v\sqrt{\mu}} P_{\Lambda_\mu}^\bot(h) \:dS(x)dvds.
\end{split}
\end{equation*}
Because
$$(v\cdot n(x))^2 = \sum\limits_{1\leq i,j\leq 3} v_iv_jn_in_j$$
the first term is null when $i\neq j$ and vanishes for $i=j$ thanks to our choice of $\alpha_c$. The last two integrals are dealt with by applying Cauchy-Schwarz inequality and the elliptic estimate on $\phi_c$ in $H^2$ $\eqref{phicH2}$. As for the case of $a(t,x)$, we obtain
\begin{equation}\label{Psi4c}
\abs{\Psi_4(t)}\leq  \frac{C_1}{4}\int_0^t\norm{c}^2_{L^2_x}\:ds + C_4\int_0^t\norm{P_{\Lambda_\mu}^\bot(h)}^2_{L^2_{\Lambda^+}}\:ds.
\end{equation}

\bigskip
As for $a(t,x)$ and $b(t,x)$, the estimate on $\Psi_5$ $\eqref{Psi5}$ will follow elliptic arguments in negative Sobolev spaces. With exactly the same computations as in $\eqref{Psi5astart}$ we have
\begin{equation}\label{Psi5cstart}
\abs{\Psi_5(t)}\leq C\int_0^t\norm{\pi_L^\bot(h)}_{L^2_{x,v}}\norm{\partial_t\nabla_x\phi_c}_{L^2_x}\:ds.
\end{equation}
Note that the contribution of $\pi_L$ vanishes by oddity on the terms involving $a(t,x)$ and $c(t,x)$ and also on the terms involving $b(t,x)$  thanks to our choice of $\alpha_c$.
\par To estimate $\norm{\partial_t\nabla_x\phi_c}_{L^2_x}$ we use the decomposition of the weak formulation $\eqref{eqpsi}$ between $t$ and $t+\eps$ (instead of between $0$ and $t$) with $\psi(t,x,v) =\sqrt{\mu} \pa{\abs{v}^2-3}\phi(x)/2$ where $\phi$ belongs to $H^1_x$ and is null at the boundary. $\psi$ does not depend on $t$, vanishes at the boundary and $\psi(x)\mu(v)$ is in $\mbox{Ker}(L)$. Hence,
$$\Psi_3(t)=\Psi_4 = \Psi_5(t)=0.$$
It remains
\begin{eqnarray*}
C\int_{\Omega} \cro{c(t+\eps)-c(t)}\phi(x)\:dx &=& \int_{t}^{t+\eps} \int_{\Omega\times\R^3}\pi_L(h)\frac{\abs{v}^2-3}{2}v\cdot\nabla_x\phi(x)\sqrt{\mu}\:dxdvds
\\&\:& + \int_{t}^{t+\eps} \int_{\Omega\times\R^3}\pi_L^\bot(h)\frac{\abs{v}^2-3}{2}v\cdot\nabla_x\phi(x)\sqrt{\mu}\:dxdvds
\\&\:&\int_{t}^{t+\eps} \int_{\Omega\times\R^3}g\frac{\abs{v}^2-3}{2}\sqrt{\mu}\phi(x)\:dxdvds .
\end{eqnarray*}
As for $a(t,x)$ we divide by $\eps$ and take the limit as $\eps$ goes to $0$. By oddity, the first integral on the right-hand side only gives terms with $b(s,x)$. The second and third terms are dealt with by a Cauchy-Schwarz inequality and we apply on $\phi$ a Poincar\'e inequality. This yields
$$\abs{\int_{\Omega} \partial_t c(t,x)\phi(x)\:dx} \leq C \cro{\norm{b}_{L^2_x}+\norm{\pi_L^\bot(h)}_{L^2_{x,v}}+\norm{g}_{L^2_x}}\norm{\nabla_x\phi}_{L^2_x}.$$

\par The latter is true for all $\phi(x)$ in $H^1_x$ vanishing on the boundary. We thus fix $t$ and apply the inequality above to 
$$-\Delta_x\phi(t,x) = \partial_tc(t,x) \quad\mbox{and}\quad \restr{\phi}{\partial\Omega}=0.$$
Exactly the same computation as for $b_J$ we obtain for any $\eps_c >0$
\begin{equation}\label{Psi5c}
\begin{split}
\abs{\Psi_5(t)} &\leq C\int_0^t\cro{\norm{b}_{L^2_x} + \norm{\pi_L^\bot(h)}_{L^2_{x,v}}+\norm{g}_{L^2_x}}\norm{\pi_L^\bot(h)}_{L^2_{x,v}}\:ds 
\\&\leq \eps_c\int_0^t\norm{b}^2_{L^2_x}\:ds + C_5(\eps_c)\int_0^t \cro{\norm{\pi_L^\bot(h)}^2_{L^2_{x,v}}+\norm{g}^2_{L^2_x}}\:ds.
\end{split}
\end{equation}

\bigskip
We now gather $\eqref{RHSc}$, $\eqref{Psi1all}$, $\eqref{Psi23c}$, $\eqref{Psi4c}$, $\eqref{Psi5c}$ and $\eqref{Psi6c}$
\begin{equation}\label{cfinal}
\begin{split}
\int_0^t \norm{c}^2_{L^2_x}\:ds \leq & N^{(c)}_h(t)-N^{(c)}_h(0) + \eps_c\int_0^t \norm{b}^2_{L^2_x}\:ds
\\&+ C_c(\eps_c)\int_0^t \cro{\norm{P_{\Lambda_\mu}^\bot(h)}^2_{L^2_{\Lambda^+}}+\norm{\pi_L^\bot(h)}^2_{L^2_{x,v}}+\norm{g}^2_{L^2_x}}\:ds.
\end{split}
\end{equation}


\bigskip
\textbf{Conclusion of the proof.} We gather the estimates we derived for $a$, $b$ and $c$. We compute the linear combination $\eqref{afinal} + \eta \times \eqref{bfinal} + \beta \times \eqref{cfinal}$. For all $\eps_b >0$ and $\eps_c >0$ this implies
\begin{equation*}
\begin{split}
&\int_0^t \cro{\norm{a}^2_{L^2_x} + \eta\norm{b}^2_{L^2_x} + \beta \norm{c}^2_{L^2_x}} \:ds 
\\&\quad\quad\quad\leq N_h(t)-N_h(0) + C_\bot\int_0^t \cro{\norm{P_{\Lambda_\mu}^\bot(h)}^2_{L^2_{\Lambda^+}}+\norm{\pi_L^\bot(h)}^2_{L^2_{x,v}}+\norm{g}^2_{L^2_x}}\:ds
\\&\quad\quad\quad\quad+\int_0^t \cro{\eta\eps_b\norm{a}^2_{L^2_x} + \pa{C_{a,b}+\beta\eps_c}\norm{b}^2_{L^2_x} + \eta C_{b,c}(\eps_b) \norm{c}^2_{L^2_x}} \:ds.
\end{split}
\end{equation*}

\par We first choose $\eta > C_{a,b}$, then $\eps_b$ such that $\eta\eps_b < 1$ and then $\beta > \eta C_{b,c}(\eps_b) $. Finally, we fix $\eps_c$ small enough such that $C_{a,b}+\beta\eps_c < \eta$ . With such choices we can absorb the last term on the right-hand side by the left-hand side. This concludes the proof of Lemma \ref{lem:controlfluidmicro}.
\end{proof}

\bigskip


\subsection{Exponential decay of the solution}\label{subsec:expodecayL2}

In this section we show that a solution to $\eqref{lineqL2}$ that preserves mass and has its trace in $L^2_\Lambda$ decays exponentially fast.

\bigskip
\begin{proof}[Proof of Theorem \ref{theo:L2}]
Let $h$ be a solution described in the statement of the theorem and define for $\lambda>0$, $\tilde{h}(t,x,v)=e^{\lambda t}h(t,x,v)$. $\tilde{h}$ satisfies the conservation of mass and is solution to
$$\partial_t \tilde{h} + v\cdot\nabla_x \tilde{h} = L_\mu(\tilde{h}) + \lambda \tilde{h}$$
with the Maxwell boundary condition. Moreover, since $\restr{h}{\Lambda}$ belongs to $L^2_\Lambda\pa{\mu^{-1/2}}$ so does $\restr{\tilde{h}}{\Lambda}$. We can use Green formula and get
\begin{equation*}
\frac{1}{2}\frac{d}{dt}\norm{\tilde{h}}^2_{L^2_{x,v}} = -\frac{1}{2}\int_{\Omega\times\R^3} v\cdot\nabla_x\pa{\tilde{h}^2}\:dxdv + \int_\Omega\langle L_\mu(\tilde{h})(t,x,\cdot),\tilde{h}(t,x,\cdot) \rangle_{L^2_v}\:dx + \lambda \norm{\tilde{h}}^2_{L^2_{x,v}}
\end{equation*}
Therefore, thnaks to the spectral gap $\eqref{spectralgapL}$ of $L$ in $L^2_v$ we get
\begin{equation} \label{1stexpodecayL2}
\frac{1}{2}\frac{d}{dt}\norm{\tilde{h}}^2_{L^2_{x,v}}\leq  -\frac{1}{2}\int_\Lambda \tilde{h}^2 \:v\cdot n(x)\:dS(x)dv -\lambda_L \norm{\pi_L^\bot(\tilde{h})}^2_{L^2_{x,v}}+ \lambda \norm{\tilde{h}}^2_{L^2_{x,v}}.
\end{equation} 
\par As we did in previous section, we divide the integral over the boundary and we apply the boundary condition $\eqref{BCL2}$ followed by the change of variable $v \mapsto \mathcal{R}_x(v)$ that sends $\Lambda^-$ to $\Lambda^+$. At last, we decompose $\restr{h}{\Lambda^+}$ into $P_{\Lambda_\mu}(h)+P_{\Lambda_\mu}^\bot(h)$ and this yields
\begin{eqnarray}
-\int_\Lambda \tilde{h}^2 \:v\cdot n(x)\:dS(x)dv &=& -\int_{\Lambda^+} \cro{\tilde{h}^2 - \pa{(1-\alpha) \tilde{h} + \alpha P_{\Lambda_\mu}(\tilde{h})}^2}v\cdot n(x)\:dS(x)dv \nonumber
\\&=& -(1-(1-\alpha)^2)\norm{P_{\Lambda_\mu}^\bot(\tilde{h})}^2_{L^2_{\Lambda^+}} \nonumber
\\&\:&\:+ 2\alpha\int_{\Lambda^+}P_{\Lambda_\mu}(\tilde{h})P_{\Lambda_\mu}^\bot(\tilde{h})v\cdot n(x)\:dS(x)dv \nonumber
\\&=& -(1-(1-\alpha)^2)\norm{P_{\Lambda_\mu}^\bot(\tilde{h})}^2_{L^2_{\Lambda^+}}.\label{boundarytermL2}
\end{eqnarray}
Combining $\eqref{1stexpodecayL2}$ and $\eqref{boundarytermL2}$ and integrating from $0$ to $t$ we get
\begin{equation}\label{2ndexpodecqyL2}
\norm{\tilde{h}(t)}^2_{L^2_{x,v}} + C\int_0^t\cro{\norm{P_{\Lambda_\mu}^\bot(\tilde{h})}_{L^2_{\Lambda^+}} + \norm{\pi_L^\bot(\tilde{h})}_{L^2_{x,v}}}\:ds \leq \norm{h_0}^2_{L^2_{x,v}} + 2\lambda \int_0^t\norm{\tilde{h}}^2_{L^2_{x,v}}\:ds.
\end{equation}

\bigskip
To conclude we use Lemma \ref{lem:controlfluidmicro} for $\tilde{h}$ with $g = \lambda \tilde{h}$:
\begin{equation}\label{expodecaylemma}
\begin{split}
\int_0^t\norm{\pi_L(\tilde{h})}^2_{L^2_{x,v}}\:ds \leq &  N_{\tilde{h}}(t)-N_{\tilde{h}}(0) 
\\ &+ C_\bot \int_0^t \cro{\norm{\pi^\bot_L(\tilde{h})}^2_{L^2_{x,v}}+\norm{P_{\Lambda_\mu}^\bot(\tilde{h})}^2_{L^2_{\Lambda^+}}+\lambda^2 \norm{\tilde{h}}^2_{L^2_{x,v}}}\:ds
\end{split}
\end{equation}
and we combine $\eps \times \eqref{expodecaylemma} + \eqref{2ndexpodecqyL2}$ for $\eps>0$.
\begin{equation*}
\begin{split}
&\cro{\norm{\tilde{h}}^2_{L^2_{x,v}}-\eps N_{\tilde{h}}(t)} + C_\eps\int_0^t \pa{\norm{\pi_L(\tilde{h})}^2_{L^2_{x,v}} + \norm{\pi_L^\bot (\tilde{h})}^2_{L^2_{x,v}}}\:ds
\\&\quad\quad\quad\quad\quad +\pa{C-\eps C_\bot}\int_0^t \norm{P_{\Lambda_\mu}^\bot(\tilde{h})}^2_{L^2_{\Lambda^+}} \:ds
\\&\quad\quad\quad \leq \norm{h_0}^2_{L^2_{x,v}\pa{\mu^{-1/2}}}-\eps N_{\tilde{h}}(0) + \pa{\eps C_\bot\lambda^2 + 2\lambda}\int_0^t \norm{\tilde{h}}^2_{L^2_{x,v}}\:ds
\end{split}
\end{equation*}
with $C_\eps = \min\br{\eps C_\bot, C-\eps C_\bot}$. Thanks to the control $\abs{N_{\tilde{h}}(s)} \leq C\norm{\tilde{h}(s)}^2_{L^2_{x,v}}$ and the fact that 
$$\norm{\pi_L(\tilde{h})}^2_{L^2_{x,v}} + \norm{\pi_L^\bot (\tilde{h})}^2_{L^2_{x,v}} = \norm{\tilde{h}}^2_{L^2_{x,v}}$$
we can choose $\eps$ small enough such that $C_\eps>0$ and then $\lambda>0$ small enough such that $\pa{\eps C_\bot\lambda^2 + 2\lambda} < C_\eps$. Such choices imply that $\norm{\tilde{h}}^2_{L^2_{x,v}}$ is uniformly bounded in time by $C\norm{h_0}^2_{L^2_{x,v}}$.
\par By definition of $\tilde{h}$, this shows an exponential decay for $h$ and concludes the proof of Theorem \ref{theo:L2}.
\end{proof}
\bigskip

%% file: frequencycollision.tex
\section{Semigroup generated by the collision frequency}\label{sec:frequencycollision}

This section is devoted to proving that the following operator
$$G_\nu = -\nu(v) -v\cdot\nabla_x$$
with the Maxwell boundary condition generates a semigroup $S_{G_\nu}(t)$ with exponential decay in $L^\infty_{x,v}$ endowed with different weights. Such a study has been done for pure specular reflections ($\alpha=0$) whereas a similar result has been obtained in the purely diffusive case ($\alpha=1$) (see \cite{Gu6} for maxwellian weights and \cite{Bri6} for more general weights and $L^1_vL^\infty_x$ framework). We adapt the methods of \cite{Gu6}\cite{Bri6} in order to fit our boundary condition. They consist in deriving an implicit formulation for the semigroup along the characteristics and then we need to control the characteristic trajectories that do not reach the plane $\br{t=0}$ in a time $t$. As we shall see, this number of problematic trajectories is small when the number of rebounds is large and so can be controlled for long times.

\bigskip
\begin{theorem}\label{theo:semigroupGnu}
Let $m(v)=m(\abs{v})\geq 0$ be such that
\begin{equation}\label{assumptionm}
\frac{\pa{1+\abs{v}}\nu(v)}{m(v)}\in L^1_v \quad\mbox{and}\quad m(v)\mu(v)\in L^\infty_v.
\end{equation}
Then for any $f_0$ in $L^\infty_{x,v}(m)$ there exists a unique solution $S_{G_\nu}(t)f_0$ in $L^\infty_{x,v}(m)$ to
\begin{equation}\label{eqGnu}
\cro{\partial_t + v\cdot\nabla_x + \nu(v)}\pa{S_{G_\nu}(t)f_0}=0
\end{equation}
such that $\restr{\pa{S_{G_\nu}(t)f_0}}{\Lambda} \in L^\infty_\Lambda(m)$ and satisfying the Maxwell boundary condition $\eqref{mixedBC}$ with initial datum $f_0$. Moreover it satisfies
$$\forall \nu_0' <\nu_0,\:\exists \: C_{m,\nu_0'}>0,\:\forall t\geq 0, \quad \norm{S_{G_\nu}(t)f_0}_{L^\infty_{x,v}(m)} \leq C_{m,\nu_0'}e^{-\nu'_0t}\norm{f_0}_{L^\infty_{x,v}(m)},$$
with $\nu_0 = \inf\br{\nu(v)}>0$.
\end{theorem}
\bigskip

A corollary of the proof of this theorem is a gain of weight when one integrates in the time variable. This will be of core importance to control the nonlinear operator.

\bigskip
\begin{cor}\label{cor:gainweightGnu}
Let $m$ be such that $m(v)\nu(v)^{-1}$ satisfies the requirements of Theorem \ref{theo:semigroupGnu}. Then there exists $C_0>0$ such that for any $\pa{f_s}_{s\in\R^+}$ in $L^\infty_{x,v}(m)$, any $\eps$ in $(0,1)$ and all $t\geq 0$,
$$ \norm{\int_0^t S_{G_\nu}(t-s)f_s(x,v)\:ds}_{L^\infty_{x,v}(m)} \leq \frac{C_0}{1-\eps}e^{-\eps\nu_0t} \sup\limits_{s\in[0,t]}\cro{e^{\eps\nu_0s}\norm{f_s}_{L^\infty_{x,v}\pa{m\nu^{-1}}}}.$$
\end{cor}
\bigskip

The rest of this Section is entirely devoted to the proof of these results.
\bigskip


\subsection{Brief description of characteristic trajectories}\label{subsec:collisionfrequencycharacteristics}

The characteristic trajectories of the free transport equation
$$\partial_t f(t,x,v) + v\cdot\nabla_xf(t,x,v) = 0$$
with purely specular reflection boundary condition will play an important role in our proof. Their study has been done in \cite[Appendix A]{Bri2} and we describe here the results that we shall use later on.

\bigskip
The description of backward characteristics relies on the time of first rebound against the boundary of $\Omega$. For $x$ in $\bar{\Omega}$ and $v \neq 0$ define
$$t_{min}(x,v) = \max\left\{t\geq 0: \: x-vs \in \bar{\Omega}, \: \forall \:0\leq s \leq t\right\}.$$
Note that for all $(x,v)\notin \Lambda_0 \cup \Lambda^-$, $t_{min}(x,v)>0$. The characteristic trajectories are straight lines in between two rebounds against the boundary, where the velocity then undergo a specular reflection.
\par From \cite[Appendix A.2]{Bri2}, starting from $(x,v)$ in $\bar{\Omega}\times\pa{\R^3-\br{0}}$, one can construct $T_1(x,v) = t_{min}(x,v)$ and the footprint $X_1(x,v)$ on $\partial\Omega$ of the backward trajectory starting from $x$ with velocity $v$ has well as its resulting velocity $V_1(x,v)$:
$$X_1(x,v) = x-T_1(x,v)v \quad\mbox{and}\quad V_1(x,v) = \mathcal{R}_{X_1(x,v)}\pa{v},$$
where we recall that $\mathcal{R}_y(v)$ is the specular reflection of $v$ at a point $y \in \partial\Omega$.
One can iterate the process and construct the second collision with the boundary at time $T_2(x,v) = T_1(x,v) + t_{min}(X_1(x,v),V_1(x,v))$, at the footprint $X_2(x,v)= X_1(X_1(x,v),V_1(x,v))$ and the second reflected velocity $V_2(x,v) = V_1(X_1,V_1)$ and so on so forth to construct a sequence $(T_k(x,v),X_k(x,v),V_k(x,v))$ in $\partial\Omega\times\R^3$. More precisely we have, for almost every $(x,v)$,
$$T_{k+1}(x,v) = T_k + t_{min}(X_k,V_k),\: X_{k+1}(x,v) = X_k -t_{min}(X_k,V_k)V_k, \: V_{k+1} = \mathcal{R}_{X_{k+1}}\pa{V_k}.$$

\par Thanks to \cite[Proposition A.4]{Bri2}, for a fixed time $t$ and for almost every $(x,v)$ there are a finite number of rebounds. In other terms, there exists $N(t,x,v)$ such that the backward trajectories starting from $(x,v)$ and running for a time $t$ is such that
$$T_{N(t,x,v)}(x,v) \leq t < T_{N(t,x,v)+1}(x,v).$$ 
\bigskip

We conclude this subsection by stating a continuity result about the footprints of characteristics. This is a rewriting of \cite[Lemmas 1 and 2]{Kim1}.

\bigskip
\begin{lemma}\label{lem:boundarycontinuityset}
Let $\Omega$ be a $C^1$ bounded domain.
\begin{enumerate}
\item[(1)] the backward exit time $t_{min}(x,v)$ is lower semi-continuous;
\item[(2)] if $v\cdot n(X_1(x,v)) <0$ then $t_{min}(x,v)$ and $X_1(x,v)$ are continuous functions of $(x,v)$;
\item[(3)] let $(x_0,v_0)$ be in $\Omega \times\R^3$ with $v_0 \neq 0$ and $t_{min}(x_0,v_0)<\infty$, if $(X_1(x_0,v_0),v_0)$ belongs to $\Lambda_0^{I-}$ then $t_{min}(x,v)$ is continuous around $(x_0,v_0)$.
\end{enumerate}
\end{lemma}
\bigskip

Note that \cite[Lemma 2]{Kim1} also gives that if $(X_1(x_0,v_0),v_0)$ belongs to $\Lambda_0^{I+}$ then $t_{min}(x,v)$ is not continuous around $(x_0,v_0)$. Therefore, points $(2)$ and $(3)$ in Lemma \ref{lem:boundarycontinuityset} imply that $\mathfrak{C}^-_\Lambda = \Lambda^-\cup \Lambda_0^{I-}$ is indeed the boundary continuity set.
\bigskip


\subsection{Proof of Theorem \ref{theo:semigroupGnu}: uniqueness}\label{subsec:collisionfrequencyuniqueness}

Assume that there exists a solution $f$ of $\eqref{eqGnu}$ in $L^\infty_{x,v}(m)$ satisfying the Maxwell boundary condition and such that $\restr{f}{\Lambda}$ belongs to $L^\infty_\Lambda(m)$. With the assumptions on the weight $m(v)$ and the following inequalities
$$\norm{f}_{L^1_{x,v}} \leq \pa{\int_{\R^3}\frac{dv}{m(v)}} \norm{f}_{L^\infty_{x,v}\pa{m}} \quad\mbox{and}\quad \norm{f}_{L^1_\Lambda} \leq \pa{\int_{\R^3}\frac{\abs{v}}{m(v)}\:dv} \norm{f}_{L^\infty_\Lambda\pa{m}}$$
we see that $f$ belongs to $L^1_{x,v}$ and its restriction $\restr{f}{\Lambda}$ belongs to $L^1_\Lambda$. 
\par We can therefore use the divergence theorem and the fact that $\nu(v)\geq \nu_0>0$:
\begin{eqnarray*}
\frac{d}{dt}\norm{f}_{L^1_{x,v}} &=& \int_{\Omega\times\R^3}\mbox{sgn}(f(t,x,v))\cro{-v\cdot\nabla_x - \nu(v)}f(t,x,v)\:dxdv
\\&=& -\int_{\Omega\times\R^3} v\cdot\nabla_x\pa{\abs{f}}\:dxdv - \norm{\nu(v)f}_{L^1_{x,v}}
\\&\leq& -\int_{\Lambda}\abs{f(t,x,v)}\pa{v\cdot n(x)}\:dS(x)dv - \nu_0\norm{f}_{L^1_{x,v}}.
\end{eqnarray*}

\bigskip
Using the Maxwell boundary condition $\eqref{mixedBC}$ and then applying the change of variable $v\to\mathcal{R}_x(v)$, which has a unit jacobian since it is an isometry, gives
\begin{equation*}
\begin{split}
&\int_{\Lambda^-}\abs{f(t,x,v)}\pa{v\cdot n(x)}\:dS(x)dv 
\\&\quad= - \int_{\Lambda^-}\abs{(1-\alpha) f(t,x,\mathcal{R}_x(v))+\alpha P_\Lambda\pa{f(t,x,\cdot)}(v)}\pa{v\cdot n(x)}\:dS(x)dv
\\&\quad = - \int_{\Lambda^+}\abs{(1-\alpha) f(t,x,v)+\alpha P_\Lambda\pa{f(t,x,\cdot)}(v)}\pa{v\cdot n(x)}\:dS(x)dv
\\&\quad\leq \int_{\Lambda^+}\abs{f(t,x,v)}\pa{v\cdot n(x)}\:dS(x)dv.
\end{split}
\end{equation*}
We used the fact that $P_\Lambda(f)(\mathcal{R}_x(v))=P_\Lambda(f)(v)$.
\par The integral over the boundary $\Lambda$ is therefore positive and so uniqueness follows from a Gr\"onwall lemma.
\bigskip


\subsection{Proof of Theorem \ref{theo:semigroupGnu}: existence and exponential decay}\label{subsec:collisionfrequencyexistence}

Let $f_0$ be in $L^\infty_{x,v}(m)$. Define the following iterative scheme:
$$\cro{\partial_t + v\cdot\nabla_x +\nu} f^{(n)} =0 \quad\mbox{and}\quad f^{(n)}(0,x,v)= f_0(x,v)\mathbf{1}_{\br{\abs{v}\leq n}}$$
with a damped version of the Maxwell boundary condition for $t>0$ and $(x,v)$ in $\Lambda^-$
\begin{equation}\label{mixedBCnu}
f^{(n)}(t,x,v) = (1-\alpha) f^{(n)}(t,x,\mathcal{R}_x(v)) + \alpha\pa{1-\frac{1}{n}}P_\Lambda(\restr{f^{(n)}}{\Lambda^+})(t,x,v).
\end{equation}
Denote by $\func{P^{(n)}}{\Lambda^+}{\Lambda^-}$ the boundary operator associated to $\eqref{mixedBCnu}$.
\par Note that $\mu(v)^{-1/2}f_0(x,v)\mathbf{1}_{\br{\abs{v}\leq n}}$ is in $L^\infty_{x,v}$ and
$$\forall\: \mu^{-1/2}f \in L^\infty_{\Lambda^+}, \quad \norm{P^{(n)}(\mu^{-1/2}f)}_{L^\infty_{\Lambda^-}} \leq \pa{1-\frac{\alpha}{n}}\norm{\mu^{-1/2}f}_{L^\infty_{\Lambda^+}}.$$
The norm of the operator $P^{(n)}$ is thus strictly smaller than one. We can apply \cite[Lemma 14]{Gu6} which implies that $f^{(n)}$ is well-defined in $L^\infty_{x,v}\pa{\mu^{-1/2}}$ with $\restr{f^{(n)}}{\Lambda}$ in $L^\infty_{\Lambda}\pa{\mu^{-1/2}}$.
\par We shall prove that in fact $f^{(n)}$ decays exponentially fast in $L^\infty_{x,v}(m)$ and that its restriction to the boundary is in $L^\infty_\Lambda(m)$. Finally, we will prove that $f^{(n)}$ converges, up to a subsequence, towards $f$ the desired solution of Theorem \ref{theo:semigroupGnu}. The proof of Theorem \ref{theo:semigroupGnu} consists in the following three steps developed in Subsections \ref{subsubsec:implicit}, \ref{subsubsec:IpRp} and \ref{subsubsec:CVfn}.
\bigskip


\subsubsection{Step 1: Implicit formula for $f^{(n)}$.}\label{subsubsec:implicit}

We use the conservation property that $e^{\nu(v)t}f^{(n)}(t,x,v)$ is constant along the characteristic trajectories. We apply it to the first collision with the boundary (recall Subsection \ref{subsec:collisionfrequencycharacteristics} for notations) and obtain that for all $(x,v)\notin \Lambda_0\cup\Lambda^-$
\begin{equation*}
\begin{split}
f^{(n)}(t,x,v) =& \mathbf{1}_{\br{t-t_{min}(x,v) \leq 0}} e^{-\nu(v)t}f_0(x-tv,v)\mathbf{1}_{\br{\abs{v}\leq n}} 
\\&+  \mathbf{1}_{\br{t-t_{min}(x,v) > 0}} e^{-\nu(v)t_{min}(x,v)} \restr{f^{(n)}}{\Lambda^-}(t-t_{min}(x,v),X_1(x,v),v).
\end{split}
\end{equation*}
Indeed, either the backward trajectory hits the boundary at $X_1(x,v)$ before time $t$ ($t_{min} < t$) or it reaches the origin plane $\br{t=0}$ before it hits the boundary ($t_{min} \geq 0$). Defining $t_1 = t_1(t,x,v) = t-t_{min}(x,v)$, and recalling the first footprint $X_1=X_1(x,v)$ and the first change of velocity $V_1(x,v))$, we apply the boundary condition $\eqref{mixedBCnu}$ and obtain the following implicit formula.
\begin{equation}\label{fncharacteristics}
\begin{split}
&f^{(n)}(t,x,v) = \mathbf{1}_{\br{t_1(x,v) \leq 0}} e^{-\nu(v)t}f_0(x-tv,v)\mathbf{1}_{\br{\abs{v}\leq n}} 
\\&\quad\quad\quad+ (1-\alpha)\:\mathbf{1}_{\br{t_1(x,v) > 0}} e^{-\nu(v)(t-t_1)}f^{(n)}(t_1,X_1,V_1)
\\&\quad\quad\quad+\mathbf{1}_{\br{t_1(x,v) > 0}}\cro{\Delta_n\mu(v)e^{-\nu(v)(t-t_1)}\int_{v_{1*}\cdot n(x_1) >0} \frac{1}{\mu(v_{1*})}f^{(n)}(t_1,X_1,v_{1*})\:d\sigma_{x_1}(v_{1*})},
\end{split}
\end{equation}
where we denoted $\Delta_n = \alpha(1-1/n)$ and we defined the probability measure on $\Lambda^+$
\begin{equation}\label{dsigmax}
d\sigma_x(v) = c_\mu \mu(v) \abs{v\cdot n(x)}\:dv.
\end{equation}
Moreover, once at $(t_1,X_1,v_1)$ with $v_1$ being either $V_1(x,v)$ or $v_{1*}$ either $t_2\leq 0$ (where $t_2=t_1(t_1,X_1,v_1)\leq 0$) and the trajectory reaches the initial plane after the first rebound or $t_2>0$ and it can still overcome a collision against the boundary in the time $t$. Again, the fact that $e^{\nu(v)t}f^{(n)}(t,x,v)$ is constant along the characteristics implies 
\begin{equation}\label{representationfn1}
f^{(n)}(t,x,v) = I_1\pa{f^{(n)}_0}(t,x,v) + R_1\pa{f^{(n)}}(t,x,v)
\end{equation}
with $I_1$ accounting for all the trajectories reaching the initial plane in at most $1$ rebound in time $t$
\begin{equation}\label{I1}
\begin{split}
&I_1\pa{f^{(n)}_0} = \mathbf{1}_{\br{t_1 \leq 0}} e^{-\nu(v)t}f^{(n)}_0(x-tv,v) 
\\&\quad+ \mathbf{1}_{\br{t_1 >0}}\mathbf{1}_{\br{t_2\leq 0}} (1-\alpha) e^{-\nu(v)(t-t_1)}e^{-\nu(V_1)t_1}f_0(X_1-t_1V_1,V_1)
\\&\quad+ \mathbf{1}_{\br{t_1 > 0}}\cro{\Delta_n\mu(v)e^{-\nu(v)(t-t_1)}\int_{v_{1*}\cdot n(x_1) >0}\mathbf{1}_{\br{t_2\leq 0}} \frac{e^{-\nu(v_{1*})t_1}}{\mu(v_{1*})}f^{(n)}_0(x_1-t_1v_{1*},v_{1*})\:d\sigma_{x_1}},
\end{split}
\end{equation}
and $R_1\pa{f^{(n)}}$ encodes the contribution of all the characteristics that after one rebound are still able to generate new collisions against $\partial\Omega$
\begin{equation}\label{R1}
\begin{split}
&R_1\pa{f^{(n)}}(t,x,v) = (1-\alpha)\:\mathbf{1}_{\br{t_2 > 0}} e^{-\nu(v)(t-t_1)}f^{(n)}(t_1,X_1,V_1)
\\&\quad\quad\quad+\Delta_n\mu(v)e^{-\nu(v)(t-t_1)}\int_{v_{1*}\cdot n(x_1) >0}\mathbf{1}_{\br{t_2 > 0}} \frac{1}{\mu(v_{1*})}f(t_1,X_1,v_{1*})\:d\sigma_{x_1}(v_{1*}).
\end{split}
\end{equation}
Of important note, to lighten computations, in each term the value $t_2$ refers to $t_1$ of the preceding triple $(t_1,x_1,v_1)$ where $v_1$ is $V_1(x,v)$ in the first term and $v_{1*}$ in the second. As we are about to iterate $\eqref{representationfn1}$, we shall generate sequences $(t_{k+1},x_{k+1},v_{k+1})$ which have to be understood as $(t_1(t_k^{(l)},x_k^{(l)},v_k^{(l)}),x_1(x_k^{(l)},v_k^{(l)}), v_{k+1})$ and $v_{k+1}$ being either $V_1(x_k^{(l)},v_k^{(l)})$ or an integration variable $v_{(k+1)*}$.

\bigskip
By a straightforward induction we obtain an implicit form for $f^{(n)}$ when one takes into account the contribution of the characteristics reaching $\br{t=0}$ in at most $p\geq 1$ rebounds
\begin{equation}\label{representationfn}
f^{(n)}(t,x,v) = I_p\pa{f^{(n)}_0}(t,x,v) + R_p\pa{f^{(n)}}(t,x,v).
\end{equation}
$I_p\pa{f^{(n)}_0}$ contains all the trajectories reaching the initial plane in at most $p$ rebounds whereas $R_p\pa{f^{(n)}}$ gathers the contributions of all the trajectories still coming from a collision against the boundary. A more careful induction gives an explicit formula for $R_p$ and this is the purpose of the next Lemma. The main idea is to look at every possible combination of the specular reflections among all the collisions against the boundary, represented by the set $\vartheta$ defined below.

\bigskip
\begin{lemma}\label{lem:IpRp}
For $p\geq 1$ and $i$ in $\br{1,\dots,p}$ define $\vartheta_p(i)$ the set of strictly increasing functions from $\br{1,\dots,i}$ into $\br{1,\dots,p}$. Let $(t_0,x_0,v_0)=(t,x,v)$ in $\R^+\times\Omega\times\R^3$ and $(v_{1*},\dots,v_{p*})$ in $R^{3p}$. For $l$ in $\vartheta_p(i)$ we define the sequence $(t_k^{(l)},x_k^{(l)},v_k^{(l)})_{1\leq k \leq p}$ by induction
$$t_{k} = t_{k-1}^{(l)}-t_{min}(x_{k-1}^{(l)},v_{k-1}^{(l)}) \quad,\: x_k^{(l)}=X_1(x_{k-1}^{(l)},v_{k-1}^{(l)})$$
$$v_k^{(l)}=\left\{\begin{array}{ll} \disp{V_1(x_{k-1}^{(l)},v_{k-1}^{(l)})\quad\quad}&\disp{\mbox{if}\quad k\in l\cro{\br{1,\dots,i}},} \vspace{1mm}\\\vspace{1mm} \disp{v_{k*}}& \disp{\mbox{otherwise}.} \end{array}\right.$$
At last, for $1\leq k \leq p$ define the following measure on $\R^{3k}$
$$d\Sigma^k_{l}\pa{v_{1*},\dots,v_{p*}} = \frac{\mu(v)}{\mu(v_k^{(l)})}\cro{\prod\limits_{j=0}^{k-1}e^{-\nu(v_j^{(l)})(t_j^{(l)}-t_{j+1}^{(l)})}} \:d\sigma_{x_1}(v_{1*})\dots d\sigma_{x_{p}}(v_{p*})$$
and the following sets
\begin{equation}\label{definition Vj}
\mathcal{V}^{(l)}_j = \br{v_{j*} \in \R^3, \quad v_{j*}\cdot n(x_j^{(l)}) > 0}.
\end{equation}
Then we have the following identities
\begin{equation}\label{Ip}
\begin{split}
&I_p\pa{f^{(n)}_0}(t,x,v) = \sum\limits_{k=0}^p\sum\limits_{i=0}^k (1-\alpha)^i\Delta_n^{k-i}
\\&\quad\times\sum\limits_{l\in\vartheta_k(i)}\int_{\prod\limits_{1\leq j \leq p}\mathcal{V}^{(l)}_j}\mathbf{1}_{\br{t_k^{(l)}>0,\:t_{k+1}^{(l)}\leq 0}}e^{-\nu(v_k^{(l)})t_k^{(l)}}f^{(n)}_0(x_k^{(l)}-t_k^{(l)}v_k^{(l)},v_k^{(l)})\:d\Sigma^k_{l}
\end{split}
\end{equation}
and
\begin{equation}\label{Rp}
R_p\pa{f^{(n)}}(t,x,v) = \sum\limits_{i=0}^p \sum\limits_{l\in\vartheta_p(i)}(1-\alpha)^i\Delta_n^{p-i}\int_{\prod\limits_{1\leq j \leq p}\mathcal{V}^{(l)}_j}\mathbf{1}_{\br{t_{p+1}^{(l)}>0}} f^{(n)}(t_p^{(l)},x_p^{(l)},v_p^{(l)})\:d\Sigma^p_{l},
\end{equation}
where we defined $t_{p+1}^{(l)} = t_p^{(l)} - t_{min}(x_p^{(l)},v_p^{(l)})$ and also by convention $l \in \vartheta_p(0)$ means that $l=0$.
\end{lemma}
\bigskip

\begin{proof}[Proof of Lemma \ref{lem:IpRp}]
The proof is done by induction on $p$ and we start with the formula for $R_p$.
\par By definition of $R_1\pa{f^{(n)}}(t,x,v)$ $\eqref{R1}$, the property holds for $p=1$ since on the pure reflection part $\mu(v) = \mu(v_1)$ and $d\sigma_{x_1}$ is a probability measure.
\par Suppose that the property holds at $p\geq 1$. Then we can apply the property $\eqref{representationfn1}$ at rank one to $f^{(n)}(t_p^{(l)},x_p^{(l)},v_p^{(l)})$. In other terms this amounts to applying the preservation of $e^{\nu(v)t}f^{(n)}(t,x,v)$ along characteristics and to keep only the contribution of trajectories still able to generate rebounds. Using the notations $t_{p+1}^{(l)} = t_p^{(l)} - t_{min}(x_p^{(l)},v_p^{(l)})$, $x_{p+1}^{(l)}=X_1(x_p^{(l)},v_p^{(l)})$, and the definition $\eqref{R1}$, it reads
\begin{equation*}
\begin{split}
&R_1\pa{f^{(n)}}(t_p^{(l)},x_p^{(l)},v_p^{(l)})
\\&= (1-\alpha)\:\mathbf{1}_{\br{t_2(t_p^{(l)},x_p^{(l)},v_p^{(l)}) > 0}} e^{-\nu(v_p^{(l)})(t_p^{(l)}-t_{p+1}^{(l)})}f^{(n)}(t_{p+1}^{(l)},x_{p+1}^{(l)},V_1(x_p^{(l)},v_p^{(l)}))
\\&\quad+\Delta_n\mu(v_p^{(l)})e^{-\nu(v_p^{(l)})(t_p^{(l)}-t_{p+1}^{(l)})}\int_{\mathcal{V}^{(l)}_{p+1}}\mathbf{1}_{\br{t_2(t_p^{(l)},x_p^{(l)},v_{(p+1)*)} > 0}} \frac{d\sigma_{x_{p+1}^{(l)}}}{\mu(v_{(p+1)*})}f(t_{p+1}^{(l)},x_{p+1}^{(l)},v_{(p+1)*}).
\end{split}
\end{equation*}
Since $\mu$ only depends on the norm and since $d\sigma_{x_{p+1}^{(l)}}$ is a probability measure, the specular part above can be rewritten as
\begin{equation*}
\begin{split}
(1-\alpha) \mu(v_p^{(l)})\int_{\mathcal{V}^{(l)}_{p+1}}&\mathbf{1}_{\br{t_2 > 0}}\frac{e^{-\nu(v_p^{(l)})(t_p^{(l)}-t_{p+1}^{(l)})}}{\mu(V_1(x_p^{(l)},v_p^{(l)}))} f^{(n)}(t_{p+1}^{(l)},x_{p+1}^{(l)},V_1(x_p^{(l)},v_p^{(l)}))\:d\sigma_{x_{p+1}^{(l)}}(v_{(p+1)*}).
\end{split}
\end{equation*}

\bigskip
For each $l$ in $\vartheta_p(i)$ we can generate $l_1$ in $\vartheta_{p+1}(i+1)$ with $l_1(p+1)=p+1$ (representing the specular reflection case) and  $l_2 = l$ in $\vartheta_{p+1}(i)$. Plugging $R_1\pa{f^{(n)}}(t_p^{(l)},x_p^{(l)},v_p^{(l)})$ into $R_p\pa{f^{(n)}}(t,x,v)$ we obtain for each $l$ the desired integral for $l_1$ and $l_2$
$$\mathcal{I}_{l_{1,2}} = \int_{\prod\limits_{1\leq j \leq p+1}\mathcal{V}_j^{(l_{1,2})}}\mathbf{1}_{\br{t_{p+1}^{(l_{1,2})}>0}} f^{(n)}(t_{p+1}^{(l_{1,2})},x_{p+1}^{(l_{1,2})},v_{p+1})\:d\Sigma^p_{l_{1,2}}\pa{v_{1*},\dots,v_{(p+1)*}}.$$
\par Our computations thus lead to
\begin{equation*}
R_p\pa{f^{(n)}}(t,x,v) = \sum\limits_{i=0}^p \sum\limits_{\overset{l\in\vartheta_{p+1}(i+1)}{l(i+1)=p+1}}(1-\alpha)^{i+1}\Delta_n^{p-i} \mathcal{I}_l + \sum\limits_{i=0}^p \sum\limits_{\overset{l\in\vartheta_{p+1}(i)}{l(i)\neq p+1}}(1-\alpha)^{i}\Delta_n^{p+1-i} \mathcal{I}_l
\end{equation*}
which can be rewritten as
\begin{equation*}
R_p\pa{f^{(n)}}(t,x,v)= \sum\limits_{i=1}^{p+1} \sum\limits_{\overset{l\in\vartheta_{p+1}(i)}{l(i)=p+1}}(1-\alpha)^{i}\Delta_n^{p+1-i} \mathcal{I}_l + \sum\limits_{i=0}^p \sum\limits_{\overset{l\in\vartheta_{p+1}(i)}{l(i)\neq p+1}}(1-\alpha)^{i}\Delta_n^{p+1-i} \mathcal{I}_l
\end{equation*}
For $i=0$ there can be no $l$ such that $l(i)=p+1$ and for $i=p+1$ the only $l$ in $\vartheta_{p+1}(i)$ is the identity and so there is no $l$ such that $l(i)\neq p+1$. In the end, for $0\leq i \leq p+1$, we are summing exactly once every function $l$ in $\vartheta_{p+1}(i)$. This concludes the proof of the lemma for $R_p$.

\bigskip
At last, $I_p$ could be derived explicitely by the same kind of induction. However, $I_p$ contains all the contributions from characteristics reaching $\br{t=0}$ in at most $p$ collisions against the boundary. It follows that $I_p$ is the sum of all the possible $R_k$ with $k$ from $0$ to $p$ such that $\mathbf{1}_{\br{t_{k+1}\leq 0}}$ to which we apply the preservation of $e^{\nu(v)t}f^{(n)}(t,x,v)$ along the backward characteristics starting at $(t_k^{(l)},x_k^{(l)},v_k^{(l)})$ up to $t$. And since $d\sigma_{x_{k+1}}(v_{(k+1)*})\dots d\sigma_{x_{p}}(v_{p*})$ is a probability measure on $\R^{3(p-k)}$ we can always have an integral against 
$$d\sigma_{x_1}(v_{1*})\dots d\sigma_{x_{p}}(v_{p*}).$$
This concludes the proof for $I_p$.
\end{proof}
\bigskip


\subsubsection{Step 2: Estimates on the operators $I_p$ and $R_p$.}\label{subsubsec:IpRp}

The next two lemmas give estimates on the operator $I_p$ and $R_p$. Note that we gain a weight of $\nu(v)$ which will be of great importance when dealing with the bilinear operator.

\bigskip
\begin{lemma}\label{lem:controlIp}
There exists $C_m>0$ only depending on $m$ such that for all $p\geq 1$ and all $h$ in $L^\infty_{x,v}(m)$,
$$\norm{I_p(h)(t)}_{L^\infty_{x,v}(m)} \leq p C_m e^{-\nu_0 t}\norm{h}_{L^\infty_{x,v}(m)}.$$
Moreover we also have the following inequality for all $(t,x,v)$ in $\R^+\times\bar{\Omega}\times\R^3$
$$m(v)\abs{I_p(h)}(t,x,v) \leq p C_m\pa{\nu(v)e^{-\nu(v)t}+e^{-\nu_0 t}}\norm{h}_{L^\infty_{x,v}(m\nu^{-1})}.$$
\end{lemma}
\bigskip

\begin{proof}[Proof of Lemma \ref{lem:controlIp}]
We only prove the second inequality as the first one follows exactly the same computations without multiplying and dividing by $\nu(v_k^{(l)})$.
\par Bounding by the $L^\infty_{x,v}(m\nu^{-1})$-norm out of the definition $\eqref{Ip}$ gives
\begin{equation}\label{ineqI1}
\begin{split}
\abs{I_p(h)(t,x,v)}\leq \norm{h}_{L^\infty_{x,v}(m\nu^{-1})}&\sum\limits_{k=0}^p\sum\limits_{i=0}^k (1-\alpha)^i\Delta_n^{k-i}
\\&\:\sum\limits_{l\in\vartheta_k(i)}\int_{\prod\limits_{1\leq j \leq p}\mathcal{V}_j^{(l)}}\mathbf{1}_{\br{t_k^{(l)}>0,\:t_{k+1}^{(l)}\leq 0}} \frac{\nu(v_k^{(l)})}{m(v_k^{(l)})}e^{-\nu(v_k^{(l)})t_k^{(l)}}\:d\Sigma^k_{l}.
\end{split}
\end{equation}
Fix $k$, $i$ and $l$. Then by definition of $(v_k^{(l)})$: either $v_k^{(l)}=V_1(\dots (V_1(x,v))))$ $k$ iterations (case of $k$ specular reflections which means that $l$ is the identity) or there exists $J$ in $\br{1,\dots,p}$ such that $v_k^{(l)}=V_1(\dots (V_1(x_j,v_{J*}))))$ $k-J$ iterations. Since $m$, $\nu$ and $\mu$ are radially symmetric this yields
\begin{equation}\label{ineqI1sigma}
\begin{split}
&\int_{\prod\limits_{1\leq j \leq p}\mathcal{V}_j^{(l)}} \frac{1}{m(v_k^{(l)})}e^{-\nu(v_k^{(l)})t_k^{(l)}}\:d\Sigma^k_{l} 
\\&\quad\quad\quad= \int_{\prod\limits_{1\leq j \leq J}\mathcal{V}_j^{(l)}} \frac{\mu(v)\nu(v_k^{(l)})}{m(v_{J*})\mu(v_{J*})}\cro{\prod\limits_{j=0}^{k-1}e^{-\nu(v_j^{(l)})(t_j^{(l)}-t_{j+1}^{(l)})}}e^{-\nu(v_J^{(l)})t_k^{(l)}}\:d\sigma_{x_1}\dots d\sigma_{x_J}.
\end{split}
\end{equation}
We use the convention that $v_{0*}=v$ so that this formula holds in both cases.
\par In the case $J=0$, all the collisions against the boundary were specular reflections and so for any $j$, $v_j^{(l)}$ is a rotation of $v$ and $t_k^{(l)}$ does not depend on any $v_{j*}$. As $\nu$ is rotation invariant the exponential decay inside the integral is exactly $e^{-\nu(v)(t-t_k^{(l)})}e^{-\nu(v)t_k^{(l)}}$. The $d\sigma_{x_j}$ are probability measures and therefore in the case when $J$ is zero
$$\eqref{ineqI1sigma} = \frac{\nu(v)}{m(v)}e^{-\nu(v)t}.$$
\par In the case $J\neq 0$ we directly bound the exponential decay by $e^{-\nu_0 t}$ and integrate all the variable but $v_{J*}$. Therefore, by definition $\eqref{dsigmax}$ of $d\sigma_x$
$$\eqref{ineqI1sigma} \leq c_\mu e^{-\nu_0 t} \mu(v)\int_{v_{J*}\cdot n(x^{(l)}_{J})}\frac{\nu(v_{J*})}{m(v_{J*})}\abs{v_{J*}\cdot n(x^{(l)}_{J})}dv_{J*}\leq \frac{C_m}{m(v)} e^{-\nu_0 t},$$
where we used the boundedness and integrability assumptions on $m$ $\eqref{assumptionm}$.

\bigskip
To conclude we plug our upper bounds on $\eqref{ineqI1sigma}$ inside $\eqref{ineqI1}$ and use 
$$\sum\limits_{k=0}^p\sum\limits_{i=0}^k\sum\limits_{l\in\vartheta_k(i)} (1-\alpha)^i\Delta_n^{k-i} = \sum\limits_{k=0}^{p} \pa{1-\frac{\alpha}{n}}^k$$
to finally get
$$m(v)\abs{I_p(h)(t,x,v)}\leq p\cro{\nu(v)e^{-\nu(v)t}+C_me^{-\nu_0t}}\norm{h}_{L^\infty_{x,v}(m\nu^{-1})}$$
which concludes the proof.
\end{proof}
\bigskip

The estimate we derive on $R_p$ needs to be more subtle. The main idea behind it is to differentiate the case when the characteristics come from a majority of pure specular reflections, and therefore has a small contribution because of the multiplicative factor $(1-\alpha)^{k}$, from the case when they come from a majority of diffusions, and therefore has a small contribution because of the small number of such possible composition of diffusive boundary condition.

\bigskip
\begin{lemma}\label{lem:controlRp}
There exists $C_m>0$ only depending on $m$ and $\bar{N}$, $\bar{C} >0$ only depending on $\alpha$ and the domain $\Omega$ such that for all $T_0 >0$, if
$$p = \bar{N}\pa{\cro{\bar{C}T_0}+1}$$,
where $\cro{\cdot}$ stands for the floor function; then for all $h = h(t,x,v)$ and for all $t$ in $[0,T_0]$
$$\sup\limits_{s\in[0,t]}\cro{e^{\nu_0 s}\norm{R_p(h)(s)}_{L^\infty_{x,v}(m)} }\leq C_m\pa{\frac{1}{2}}^{\cro{\bar{C}T_0}}\sup\limits_{s \in [0,t]}\cro{e^{\nu_0s}\norm{\mathbf{1}_{\br{t_1>0}}h(s)}_{L^\infty_{\Lambda^+}(m)}}.$$
Moreover, the following inequality holds for all $(t,x,v)$ in $\R^+\times\bar{\Omega}\times\R^3$ and all $\eps$ in $[0,1]$,
\begin{equation*}
\begin{split}
m(v)\abs{R_p(h)}(t,x,v) \leq& C_me^{-\eps\nu_0t}\pa{\frac{1}{2}}^{\cro{\bar{C}T_0}}\pa{\nu(v)e^{-\nu(v)(1-\eps)t}+e^{-\nu_0(1-\eps) t}}
\\&\quad\times\sup\limits_{s \in [0,t]}\cro{e^{\eps\nu_0s}\norm{h(s)}_{L^\infty_{\Lambda^+}(m\nu^{-1})}}.
\end{split}
\end{equation*}
\end{lemma}
\bigskip

\begin{proof}[Proof of Lemma \ref{lem:controlRp}]
Let $(t,x,v)$ in $\R^\times\Omega\times\R^3$. Again, we shall only prove the second inequality, the first one being dealt with exactly the same way.
\par  First, the exponential decay inside $d\Sigma_l^p$ (see Lemma \ref{lem:IpRp}) is bounded by $e^{-\nu_0(t-t_p^{(l)})}$ if there is at least one diffusion or by $e^{-\nu(v)(t-t_p^{(l)})}$ if only specular reflections occur in the $p$ rebounds (because then the reflection preserves $\abs{v}$ and $\nu$ only depends on the norm ), that is $i=p$ and $l=\mbox{Id}$. Second, by definition of $(t_{k}^{(l)},x_k^{(l)},v_k^{(l)})$ (see Lemma \ref{lem:IpRp}) we can bound
\begin{eqnarray*}
\mathbf{1}_{\br{t_{p+1}^{(l)}>0}}\abs{h(t_p^{(l)},x_p^{(l)},v_p^{(l)})}m(v_p^{(l)}) &=& \mathbf{1}_{\br{t_1(t_p^{(l)},x_p^{(l)},v_p^{(l)})>0}}\abs{h(t_p^{(l)},x_p^{(l)},v_p^{(l)})}m(v_p^{(l)})
\\&\leq& \mathbf{1}_{\br{t_p^{(l)}>0}}\nu(v_p^{(l)})\norm{\mathbf{1}_{\br{t_1>0}}h(t_p^{(l)})}_{L^\infty_{\Lambda^+}(m\nu^{-1})}.
\end{eqnarray*}
We thus obtain the following bound
\begin{equation*}
\begin{split}
\abs{R_p(h)(t,x,v)}\leq& \sum\limits_{i=0}^{p_1} \sum\limits_{l\in\vartheta_p(i)}(1-\alpha)^i\Delta_n^{p-i}\int_{\prod\limits_{1\leq j \leq p}\mathcal{V}_j^{(l)}} \mathbf{1}_{\br{t_p^{(l)}>0}}\frac{\mu(v)\nu(v_p^{(l)})}{m(v_p^{(l)})\mu(v_p^{(l)})}
\\&\quad\quad\quad\times e^{-\nu_0(t-t_p^{(l)})}\norm{\mathbf{1}_{\br{t_1>0}}h(t_p^{(l)})}_{L^\infty_{\Lambda^+}(m\nu^{-1})}\:d\sigma_{x_1^{(l)}}\dots d\sigma_{x_p^{(l)}}
\\& +(1-\alpha)^p\mathbf{1}_{\br{t_p^{(\mbox{\scriptsize{Id}})}>0}}\nu(v)e^{-\nu(v)(t-t_p^{(\mbox{\scriptsize{Id}})})}\norm{\mathbf{1}_{\br{t_1>0}}h(t_p^{(\mbox{\scriptsize{Id}})})}_{L^\infty_{\Lambda^+}(m\nu^{-1})}.
\end{split}
\end{equation*}
Which implies for $0\leq \eps \leq 1$
\begin{equation}\label{ineqR1}
\begin{split}
&e^{\eps\nu_0t}m(v)\abs{R_p(h)(t,x,v)}
\\&\quad\quad \leq \pa{\sum\limits_{i=0}^p \sum\limits_{l\in\vartheta_p(i)}(1-\alpha)^i\Delta_n^{p-i}\int_{\prod\limits_{1\leq j \leq p}\mathcal{V}_j^{(l)}} \mathbf{1}_{\br{t_p^{(l)}>0}}\frac{\mu(v)m(v)}{m(v_p^{(l)})\mu(v_p^{(l)})}d\sigma_{x_1^{(l)}}.. d\sigma_{x_p^{(l)}}}
\\&\quad\quad\quad\quad\times\pa{\nu(v)e^{-\nu(v)(1-\eps)t}+e^{-\nu_0(1-\eps)t}}\sup\limits_{s \in [0,t]}\cro{e^{\eps\nu_0s}\norm{\mathbf{1}_{\br{t_1>0}}h(s)}_{L^\infty_{\Lambda^+}(m\nu^{-1})}}.
\end{split}
\end{equation}

\bigskip
By definition, Lemma \ref{lem:IpRp}, $t_p^{(l)} = t_p^{(l)}(t,x,v,v^{(l)}_1,v^{(l)}_2,\dots,v_p^{(l)})$ and thus for all $j$ in $\br{1,\dots,p}$,
$$\mathbf{1}_{\br{t_p^{(l)}>0}}\leq \mathbf{1}_{\br{t_{p-j}^{(l)}>0}}.$$
Following the reasoning of the proof of Lemma \ref{lem:controlIp}, for fixed $i$ and $l$, there exists $J$ in $\br{0,\dots,p}$ such that $v_p^{(l)}=V_1(\dots (V_1(x_J,v_{J*}))))$ $p-J$ iterations, with the convention that $v_{0*}=v$. The measures $d\sigma_x$ are probability measures and the functions $m$, $\nu$ and $\mu$ are rotation invariant. Therefore
\begin{equation*}
\begin{split}
&\int_{\prod\limits_{1\leq j \leq p}\mathcal{V}_j^{(l)}} \mathbf{1}_{\br{t_p^{(l)}>0}}\frac{\mu(v)m(v)}{m(v_p^{(l)})\mu(v_p^{(l)})}d\sigma_{x_1^{(l)}}\dots d\sigma_{x_p^{(l)}} 
\\&\:\leq \int_{\mathcal{V}_j^{(l)}}\frac{\mu(v)m(v)}{m(v_{J*}^{(l)})\mu(v_{J*}^{(l)})}\:d\sigma_{x_{J*}^{(l)}}(v_{j*})\pa{\int_{\prod\limits_{1\leq j \leq J-1}\br{v_{j*}\cdot n(x^{(l)}_j)>0}} \mathbf{1}_{\br{t_{(J-1)*}^{(l)}>0}}d\sigma_{x_1^{(l)}}.. d\sigma_{x_{J-1}^{(l)}}}.
\end{split}
\end{equation*}
In the case $J=0$ we have $v_{J*}^{(l)}=v$ and therefore the above is exactly one. In the case $J\geq 1$, assumption $\eqref{assumptionm}$ on $m$ implies that the integral over $v_{J*}^{(l)}$ is bounded uniformly by $C_m$. So we have
\begin{equation}\label{ineqR1sigma}
\begin{split}
&\int_{\prod\limits_{1\leq j \leq p}\mathcal{V}_j^{(l)}} \mathbf{1}_{\br{t_p^{(l)}>0}}\frac{\mu(v)m(v)}{m(v_p^{(l)})\mu(v_p^{(l)})}\:d\sigma_{x_1^{(l)}}.. d\sigma_{x_p^{(l)}}\leq C_m\int_{\prod\limits_{1\leq j \leq J-1}\mathcal{V}_j^{(l)}} \mathbf{1}_{\br{t_{(J-1)*}^{(l)}>0}}d\sigma_{x_1^{(l)}}.. d\sigma_{x_{J-1}^{(l)}}.
\end{split}
\end{equation}
Plugging $\eqref{ineqR1sigma}$ into $\eqref{ineqR1}$ gives
\begin{equation}\label{finalRp}
\begin{split}
e^{\eps\nu_0t}m(v)\abs{R_p(h)(t,x,v)} \leq& C_mF_p(t)\pa{\nu(v)e^{-\nu(v)(1-\eps)t}+e^{-\nu_0(1-\eps)t}}
\\&\quad\quad\quad\times\sup\limits_{s \in [0,t]}\cro{e^{\nu_0s}\norm{\mathbf{1}_{\br{t_1>0}}h(s)}_{L^\infty_{\Lambda^+}(m\nu^{-1})}}
\end{split}
\end{equation}
with
$$F_p(t) = \sup\limits_{x,v}\cro{\sum\limits_{i=0}^p \sum\limits_{l\in\vartheta_p(i)}(1-\alpha)^i\Delta_n^{p-i}\int_{\prod\limits_{1\leq j \leq J-1}\mathcal{V}_j^{(l)}} \mathbf{1}_{\br{t_{(J-1)*}^{(l)}>0}}d\sigma_{x_1^{(l)}}.. d\sigma_{x_{J-1}^{(l)}}}.$$

\bigskip
It remains to prove an upper bound on $F_p(t)$ for $0\leq t \leq T_0$ when $T_0$ and $p$ are large. Let $T_0>0$, $p$ in $\N$ and $0<\delta<1$ to be determined later.
\par For any given $i$ in $\br{1,\dots,p}$ and $l$ in $\vartheta_p(i)$ we define the non-grazing sets for all $j$ in $\br{1,\dots,p}$ as
$$\Lambda^{(l),\delta}_j =  \br{v_j^{(l)}\cdot n(x_j^{(l)})\geq \delta} \cap \br{\abs{v_j^{(l)}}\leq \frac{1}{\delta}}.$$
By definition of the backward characteristics we have $x_j^{(l)}-x_{j+1}^{(l)} = (t_{j}^{(l)}-t_{j+1}^{(l)})v_{j}^{(l)}$. Since $\Omega$ is a $C^1$ bounded it is known \cite[Lemma 2]{Gu6} that there exists $C_\Omega>0$ such that
$$\forall v_{j}^{(l)} \in \Lambda^{(l),\delta}_j,\quad \abs{t_{j}^{(l)}-t_{j+1}^{(l)}} \geq \frac{\abs{v_{j}^{(l)}\cdot n(x_{j}^{(l)})}}{C_\Omega \abs{v_{j}^{(l)}}} \geq \frac{\delta^3}{C_\Omega}.$$
Therefore, for $t$ in $[0,T_0]$, if $t^{(l)}_{(J-1)*}(t,x,v,v^{(l)}_1,v^{(l)}_2,\dots,v_{J-1}^{(l)})>0$ then there can be at most $\cro{C_\Omega T_0 \delta^{-3}}+1$ velocities $v^{(l)}_j$ in $\Lambda^{(l),\delta}_j$. Among these, we have exactly $k$ velocities $v^{(l)}_j$ that are integration variables $v_{j*}$ and the rest are specular reflections. Since $i$ represents the total number of specular reflections, it remains exactly $p-i-k$ integration variables that are not in any $\Lambda^{(l),\delta}_j$. Recalling that $d\sigma_x$ is a probability measure, if $v_j^{(l)}$ is a specular reflection we bound the integral in $v_{j*}$ by one. All these thoughts yield
\begin{equation*}
\begin{split}
F_p(t) \leq& \sup\limits_{x,v}\Big[\sum\limits_{i=0}^p \sum\limits_{l\in\vartheta_p(i)}(1-\alpha)^i\Delta_n^{p-i} 
\\&\quad\quad\times\sum\limits_{j=0}^{\cro{\frac{C_\Omega T_0}{\delta^{3}}}+1}\sum\limits_{k=0}^j\int_{\br{\begin{array}{l} \mbox{exactly $k$ of $v_* \in \Lambda^{(l),\delta}$,} \\ \mbox{$j-k$ of specular in $\Lambda^{(l),\delta}$,} \\ \mbox{$p-i-k$ of $v_*$ not in $\Lambda^{(l),\delta}$}\end{array}}}\:\prod\limits_{1\leq m \leq J-1} d\sigma_{x_m^{(l)}}(v_{m*})\Big]
\\\leq & \sum\limits_{i=0}^p \sum\limits_{l\in\vartheta_p(i)}(1-\alpha)^i\Delta_n^{p-i}\sum\limits_{j=0}^{\cro{\frac{C_\Omega T_0}{\delta^{3}}}+1}{j \choose k}
\\&\times\sum\limits_{k=0}^j\pa{\sup\limits_{t,x,v,i,l,j}\int_{\Lambda^{(l),\delta}}d\sigma_{x_j^{(l)}}(v_*)}^{k}\pa{\sup\limits_{t,x,v,i,l,j}\int_{v_*\notin\Lambda^{(l),\delta}}d\sigma_{x_j^{(l)}}(v_*)}^{p-i-k}.
\end{split}
\end{equation*}

\par In what follows we denote by $C$ any positive constant independent of $t$, $x$, $v$, $i$, $l$ and $j$. We bound first
$$\int_{v_*\notin\Lambda^{(l),\delta}}d\sigma_{x_j^{(l)}}(v_*) \leq \int_{0<v_*\cdot n(x_j^{(l)})\leq \delta}d\sigma_{x_j^{(l)})} + \int_{\abs{v_*}\geq \delta^{-1}}d\sigma_{x_j^{(l)})} \leq C\delta$$
and second we bound by one the integrals on $\Lambda^{(l),\delta}$. With $C\delta <1$ we end up with
\begin{eqnarray*}
F_p(t) &\leq& \sum\limits_{i=0}^p \sum\limits_{l\in\vartheta_p(i)}(1-\alpha)^i(C\Delta_n\delta)^{p-i}\sum\limits_{j=0}^{\cro{\frac{C_\Omega T_0}{\delta^{3}}}+1}\sum\limits_{k=0}^j {j \choose k} (C\delta)^{-k}
\\ &\leq& \sum\limits_{i=0}^p \sum\limits_{l\in\vartheta_p(i)}(1-\alpha)^i(C\Delta_n\delta)^{p-i}\sum\limits_{j=0}^{\cro{\frac{C_\Omega T_0}{\delta^{3}}}+1}\pa{1+\frac{1}{C\delta}}^j
\\&\leq& C\delta\pa{1+\frac{1}{C\delta}}^{\cro{\frac{C_\Omega T_0}{\delta^{3}}}+2}\sum\limits_{i=0}^p {p \choose i}(1-\alpha)^i(C\alpha\delta)^{p-i}
\\&\leq& 2\pa{1+\frac{1}{C\delta}}^{\cro{\frac{C_\Omega T_0}{\delta^{3}}}+1}\pa{(1-\alpha) + C\delta\alpha}^p.
\end{eqnarray*}

\par Since $\alpha>0$ we can choose $\delta>0$ small enough such that $(1-\alpha) + C\delta\alpha = \alpha_0 <1$. Then choose $N$ in $\N$ large enough such that
$$\pa{1+\frac{1}{C\delta}}\alpha_0^N \leq \frac{1}{2}.$$
Finally choose $p = N(\cro{\frac{C_\Omega T_0}{\delta^{3}}}+1)$. It follows that
$$F_p(t) \leq 2\cro{\pa{1+\frac{1}{C\delta}}\alpha_0^N}^{\cro{\frac{C_\Omega T_0}{\delta^{3}}}+1}\leq \pa{\frac{1}{2}}^{\cro{\frac{C_\Omega T_0}{\delta^{3}}}}.$$
This inequality with $\eqref{finalRp}$ concludes the proof of the lemma.
\end{proof}
\bigskip


\subsubsection{Step 3: Exponential decay and convergence of $f^{(n)}$.}\label{subsubsec:CVfn}

Fix $T_0 >0$ to be chosen later and choose $p=p_R(T_0)$ defined in Lemma \ref{lem:controlRp}. We have that for all $n$ in $\N$,
\begin{itemize}
\item by $\eqref{representationfn1}$, for every $(t,x,v)$ in $\R^+\times\bar{\Omega}\times\R^3$,
$$\mathbf{1}_{\br{t_1(t,x,v)}\leq 0}f^{(n)}(t,x,v)= e^{-\nu(v)t}f^{(n)}_0(x-tv,v)$$
 and hence
\begin{equation}\label{fnt1<0}
\sup\limits_{s\in [0,t]}\cro{e^{\nu_0 s}\norm{\mathbf{1}_{\br{t_1(t,x,v)}\leq 0}f^{(n)}(t,x,v)}_{L^\infty_{x,v}\pa{m}}} \leq \norm{f_0^{(n)}}_{L^\infty_{x,v}(m)} \leq \norm{f_0}_{L^\infty_{x,v}(m)};
\end{equation}
\item by Lemmas \ref{lem:IpRp}, \ref{lem:controlIp} and \ref{lem:controlRp}, for every $(t,x,v)$ in $[0,T_0]\times\bar{\Omega}\times\R^3$
\begin{equation}\label{fnt1>0}
\begin{split}
&\sup\limits_{s\in [0,t]}\cro{e^{\nu_0 s}\norm{\mathbf{1}_{\br{t_1(s,x,v)> 0}}f^{(n)}(s,x,v)}_{L^\infty_{x,v}\pa{m}}} 
\\&\quad\quad\quad\leq \sup\limits_{s\in [0,t]}\cro{e^{\nu_0 s}\norm{I_p(f^{(n)}_0)(s)}_{L^\infty_{x,v}(m)}} + \sup\limits_{s\in [0,t]}\cro{e^{\nu_0 s}\norm{R_p(h)(s)}_{L^\infty_{x,v}(m)}}
\\&\quad\quad\quad\leq  p C_m \norm{f_0}_{L^\infty_{x,v}(m)} 
\\&\quad\quad\quad\quad+ C_m\pa{\frac{1}{2}}^{\cro{\bar{C}T_0}}\sup\limits_{s\in [0,t]}\cro{e^{\nu_0 s}\norm{\mathbf{1}_{\br{t_1(s,x,v)> 0}}f^{(n)}(s,x,v)}_{L^\infty_{\Lambda^+}\pa{m}}}
\end{split}
\end{equation}
\end{itemize}

\bigskip
We recall Lemma \ref{lem:controlRp} and we have $p_R(T0) \leq \bar{N}(\bar{C} T_0 +1)$. Let $\nu_0'$ in $(0,\nu_0)$. Suppose $T_0$ was chosen large enough such that 
$$C_m\pa{\frac{1}{2}}^{\cro{\bar{C}T_0}} \leq \frac{1}{2} \quad\quad\mbox{and}\quad\quad 2 C_m\bar{N}(\bar{C}T_0 +1)e^{-\nu_0 T_0} \leq e^{-\nu_0' T_0}.$$
Applying $\eqref{fnt1>0}$ at $T_0$ gives
$$\norm{\mathbf{1}_{\br{t_1(T_0)> 0}}f^{(n)}(T_0)}_{L^\infty_{x,v}\pa{m}} \leq 2C_m p_R(T_0)e^{-\nu_0 T_0}\norm{f_0}_{L^\infty_{x,v}(m)} \leq e^{-\nu'_0 T_0}\norm{f_0}_{L^\infty_{x,v}(m)},$$
and with $\eqref{fnt1<0}$ we finally have
$$\norm{f^{(n)}(T_0)}_{L^\infty_{x,v}\pa{m}} \leq  e^{-\nu'_0 T_0}\norm{f_0}_{L^\infty_{x,v}(m)}.$$

\par We could now start the proof at $T_0$ up to $2T_0$ and iterating this process we get
\begin{equation*}
\begin{split}
\forall n\in\N,\quad \norm{f^{(n)}(nT_0)\mathbf{1}_{t_1>0}}_{L^\infty_{x,v}(m)} &\leq e^{-\nu_0'T_0} \norm{f^{(n)}((n-1)T_0}_{L^\infty_{x,v}(m)}
\\&\leq e^{-2\nu_0'T_0} \norm{f^{(n)}((n-2)T_0)}_{L^\infty_{x,v}(m)}
\\&\leq \dots \leq e^{-\nu_0'nT_0} \norm{f_0}_{L^\infty_{x,v}(m)}.
\end{split}
\end{equation*}
Finally, for all $t$ in $[nT_0,(n+1)T_0]$ we apply $\eqref{fnt1>0}$ with the above to get
\begin{equation*}
\begin{split}
\norm{f^{(n)}\mathbf{1}_{t_1>0}(t)}_{L^\infty_{x,v}(m)} &\leq 2C_m p_R(T_0) e^{-\nu_0 (t-nT_0)} \norm{f^{(n)}(nT_0)}_{L^\infty_{x,v}(m)}
\\&\leq 2C_m p_R(T_0) e^{-\nu_0't}e^{-(\nu_0-\nu_0')( t-nT_0)}\norm{f_0}_{L^\infty_{x,v}(m)}.
\end{split}
\end{equation*}
Hence the uniform control in $t$, where $C_0>0$ depends on $m$, $T_0$ and $\nu_0'$,
$$\exists C_0 >0,\:\forall t\geq 0, \quad \norm{f^{(n)}\mathbf{1}_{t_1>0}(t)}_{L^\infty_{x,v}} \leq C_0 e^{-\nu'_0 t} \norm{f_0}_{L^\infty_{x,v}(m)}$$
which combined with $\eqref{fnt1<0}$ implies
\begin{equation}\label{frequencyfinal}
\forall n\in\N,\:\forall t \geq 0, \quad \norm{f^{(n)}(t)}_{L^\infty_{x,v}(m)} \leq \max\br{1,C_0} e^{-\nu'_0 t} \norm{f_0}_{L^\infty_{x,v}(m)}.
\end{equation}
\par Since $\eqref{fnt1<0}$ and $\eqref{fnt1>0}$ holds for $x$ in $\bar{\Omega}$, inequality $\eqref{frequencyfinal}$ holds in $L^\infty\pa{\bar{\Omega}\times\R^3}(m)$. Therefore, $\pa{f^{(n)}}_{n\in\N}$ is bounded in $L^\infty_tL^\infty\pa{\bar{\Omega}\times\R^3}(m)$ and converges, up to a subsequence, weakly-* towards $f$ in $L^\infty_tL^\infty\pa{\bar{\Omega}\times\R^3}(m)$ and $f$ is a solution to $\partial_t f = G_{\nu}f$ satisfying the Maxwell boundary condition and with initial datum $f_0$. Moreover, we have the expected exponential decay for $f$  thanks to the uniform $\eqref{frequencyfinal}$. This concludes the proof of Theorem \ref{theo:semigroupGnu} and we now prove Corollary \ref{cor:gainweightGnu}.

\begin{proof}[Proof of Corollary \ref{cor:gainweightGnu}]
Thanks to the convergence properties of $\pa{f^{(n)}}_{n\in\N}$, Lemmas \ref{lem:IpRp}, \ref{lem:controlIp} and \ref{lem:controlRp} are directly applicable to the semigroup $S_{G_\nu}(t)$ with $\Delta_n$ replaced by $\alpha$. Therefore, as usual, for $f_s$ in $L^\infty_{x,v}(m)$ we decompose into $t-s \leq t_{min}(x,v)$ and $t-s \geq t_{min}$, which gives thanks to $\eqref{representationfn}$
\begin{equation}\label{gainweightstart}
\begin{split}
\int_0^t S_{G_\nu}(t-s)f_s(x,v)\:ds =& \int_{\max\br{0,t-t_{min}}}^{t} e^{-\nu(v)(t-s)}f_s(x-(t-s)v,v)\:ds 
\\&+ \int_{0}^{\max\br{0,t-t_{min}}} I_p(f_s)(t-sx,v)\:ds
\\&+\int_{0}^{\max\br{0,t-t_{min}}}R_p(S_{G_\nu}f_s)(t-s,x,v)\:ds.
\end{split}
\end{equation}
Let $\eps$ be in $(0,1)$. We bound $e^{-\nu(v)(t-s)} \leq e^{-\eps\nu_0 t}e^{-(1-\eps)\nu(v)(t-s)}e^{-\eps\nu_0 s}$
and thus, using the estimate with a gain of weight for $I_p$ in Lemma \ref{lem:controlIp}, we control from above the absolute value of the first two terms by
\begin{equation}\label{gainweightIp}
\begin{split}
&\abs{\int_{\max\br{0,t-t_{min}}}^{t} e^{-\nu(v)(t-s)}f_s(x-(t-s)v,v)\:ds + \int_{0}^{\max\br{0,t-t_{min}}} I_p(f_s)(t-sx,v)\:ds}
\\&\quad\leq pC_m\int_0^t \pa{\nu(v)e^{-\nu(v)(t-s)}+e^{-\nu_0(t-s)}}\norm{f_s}_{L^\infty_{x,v}(m\nu^{-1})}\:ds 
\\&\quad\leq pC_m e^{-\eps\nu_0 t}\pa{\int_0^t\nu(v)e^{-(1-\eps)\nu(v)(t-s)} + e^{-(1-\eps)\nu_0(t-s)}} \sup\limits_{s\in[0,t]}\cro{e^{\eps\nu_0s}\norm{f_s}_{L^\infty_{x,v}(m)}}
\\&\quad\leq C_m \frac{p}{1-\eps} e^{-\eps\nu_0 t}\sup\limits_{s\in[0,t]}\cro{e^{\eps\nu_0s}\norm{f_s}_{L^\infty_{x,v}(m\nu^{-1})}}.
\end{split}
\end{equation}

The third term is treated using Lemma \ref{lem:controlRp} with an exponential weight $e^{\eps\nu_0 t}$. This yields

\begin{equation*}
\begin{split}
&\abs{\int_{0}^{\max\br{0,t-t_{min}}}R_p(S_{G_\nu}f_s)(t-s,x,v)\:ds} 
\\&\quad\leq  C_m\pa{\frac{1}{2}}^{\cro{\bar{C}T_0}}\int_{0}^{\max\br{0,t-t_{min}}} e^{-\eps\nu_0(t-s)}\pa{\nu(v)e^{-(1-\eps)\nu(v)(t-s)} + e^{-(1-\eps)\nu_0(t-s)}}
\\&\quad\quad\quad\quad\quad\quad\quad\quad\quad\quad\quad\quad\times\sup\limits_{s_* \in[0,t-s]}\cro{e^{\eps\nu_0s_*}\norm{S_{G_\nu}(s_*)(f_s)}_{L^\infty_{x,v}(m\nu^{-1})}}\:ds
\end{split}
\end{equation*}
which is further bounded as
\begin{equation*}
\begin{split}
&\leq  C_m\pa{\frac{1}{2}}^{\cro{\bar{C}T_0}}e^{-\eps\nu_0t}\int_{0}^{t} \pa{\nu(v)e^{-(1-\eps)\nu(v)(t-s)} + e^{-(1-\eps)\nu_0(t-s)}}
\\&\quad\quad\quad\quad\quad\quad\quad\quad\quad\quad\quad\quad\times\sup\limits_{s_* \in[0,t-s]}\cro{\norm{e^{\eps\nu_0(s+s_*)}S_{G_\nu}(s_*)(f_s)}_{L^\infty_{x,v}(m\nu^{-1})}}\:ds.
\end{split}
\end{equation*}
Since $m\nu^{-1}$ satisfies the requirements of Theorem \ref{theo:semigroupGnu}, we can use the exponential decay of $S_{G_\nu}(s_*)$ with the exponential rate being $\eps\nu_0 <\nu_0$ and obtain
\begin{equation}\label{gainweightRp}
\begin{split}
&\abs{\int_{0}^{\max\br{0,t-t_{min}}}R_p(S_{G_\nu}f_s)(t-s,x,v)\:ds}
\\&\quad\quad\quad\quad\quad\leq \frac{C_m}{1-\eps}\pa{\frac{1}{2}}^{\cro{\bar{C}T_0}}e^{-\eps\nu_0t} \sup\limits_{s\in[0,t]}\cro{e^{\eps\nu_0s}\norm{f_s}_{L^\infty_{x,v}(m\nu^{-1})}}.
\end{split}
\end{equation}

\bigskip
For any $T_0\geq 1$, $2^{\cro{\bar{C}T_0}} \leq 2^{-1}$ and thus $\eqref{gainweightRp}$ becomes independent of $T_0$ and holds for all $t\geq 0$. Plugging $\eqref{gainweightIp}$ and $\eqref{gainweightRp}$ into $\eqref{gainweightstart}$ yields the expected gain of weight with exponential decay.
\end{proof}
\bigskip

%% file: Linftytheory.tex
\section{$L^\infty$ theory for the linear operator with Maxwellian weights}\label{sec:Linftytheory}

As explained in the introduction, the $L^2$ setting is not algebraic for the bilinear operator $Q$. We therefore need to work within an $L^\infty$ framework. This section is devoted to the study of the semigroup generated by the full linear operator together with the Maxwell boundary condition in the space $L^\infty_{x,v}\pa{\mu^{-\zeta}}$ with $\zeta$ in $(1/2,1)$. This weight allows us to obtain sharper estimates on the compact operator $K$ and thus extend the validity of our proof up to $\alpha =2/3$. In this section we establish the following theorem.

\bigskip
\begin{theorem}\label{theo:semigroupLinfty}
Let $\alpha$ be in $(\sqrt{2/3},1]$. There exist $\zeta_\alpha$ in $(1/2,1)$ such that for any $\zeta$ in $(\zeta_\alpha,1)$, the linear perturbed operator $G=L-v\cdot\nabla_x$, together with Maxwell boundary condition, generates a semigroup $S_{G}(t)$ on $L^\infty_{x,v}(\mu^{-\zeta})$. Moreover, there exists $\lambda_\infty$ and $C_\infty>0$ such that 
$$\forall t \geq 0, \quad \norm{S_{G}(t)\pa{\mbox{Id}-\Pi_{G}}}_{L^\infty_{x,v}(\mu^{-\zeta})} \leq C_\infty e^{-\lambda_\infty t},$$
where $\Pi_{G}$ is the orthogonal projection onto $\mbox{Ker}(G)$ in $L^2_{x,v}\pa{\mu^{-1/2}}$ (see $\eqref{PiG}$).
\\The constants $C_\infty$ and $\lambda_\infty$ are explicit and depend on $\alpha$, $\zeta$ and the collision kernel.
\end{theorem}
\bigskip


\subsection{Preliminaries: pointwise estimate on $K$ and $L^2-L^\infty$ theory}\label{subsec:L2Linfty}

We recall that $L = -\nu(v) + K$. The following pointwise estimate on $K$ has been proved in \cite[Lemma 3]{Gu6} for hard sphere models and \cite[Lemma 5.2]{BriDau} for more general kernels.

\bigskip
\begin{lemma}\label{lem:control K}
There exists $k(v,v_*)\geq 0$ such that for all $v$ in $\R^3$,
$$K(f)(v) = \int_{\R^3} k(v,v_*)f(v_*)\:dv_*.$$
Moreover, for $\zeta$ in $[0,1)$ there exists $C_{\zeta} >0$ and $\eps_{\zeta}>0$ such that for all $\eps$ in $[0,\eps_{\zeta})$,
$$ \int_{\R^3}\abs{ k(v,v_*)}e^{\frac{\eps}{8}  \abs{v-v_*}^2+ \frac{\eps}{8} \frac{\abs{\abs{v}^2-\abs{v_*}^2}^2}{\abs{v-v_*}^2}}\frac{ \mu(v)^{-\zeta}}{\mu(v_*)^{-\zeta}}\:dv_* \leq \frac{C_{\zeta}}{1+\abs{v}}.$$
\end{lemma}
\bigskip

We now prove a more precise and more explicit control over the operator $K$. The idea behind it is that as $\zeta$ goes to $1$, the operator $K$ gets closer to the collision frequence $3\nu(v)$.

\bigskip
\begin{lemma}\label{lem:control K<3}
There exists $C_K >0$ such that for all $\zeta$ in $[1/2,1]$,
$$\forall f \in L^\infty_{x,v}(\mu^{-\zeta}), \quad \norm{K(f)}_{L^\infty_{x,v}(\nu^{-1}\mu^{-\zeta})} \leq C_K(\zeta)\norm{f}_{L^\infty_{x,v}(\mu^{-\zeta})}$$
where $C_K(\zeta) = 3+ C_K (1-\zeta)$.
\end{lemma}
\bigskip

\begin{proof}[Proof of Lemma \ref{lem:control K<3}]
The change of variable $\sigma \to -\sigma$ exchanges $v'$ and $v'_*$ and we can so rewrite
\begin{equation}\label{decomposition K}
K(f)(v) =\int_{R^3\times\S^2} b(\cos \theta)\abs{v-v_*}^\gamma \cro{2\mu'f'_*-\mu f_*}d\sigma dv_*
= K_1(f)-K_2(f)
\end{equation}
where $K_1$ and $K_2$ are just the integral divide into the two contributions.

\bigskip
We start with $K_1$. We use the elastic collision identity $\mu \mu_* = \mu'\mu'_*$ to get
\begin{eqnarray*}
\abs{\nu(v)^{-1}\mu(v)^{-\zeta}K_1(f)(v)} &\leq& 2\nu(v)^{-1}\int_{R^3\times\S^2} b(\cos \theta)\abs{v-v_*}^\gamma \mu_*^\zeta \mu(v')^{1-\zeta}\abs{\frac{f_*'}{\mu(v'_*)^{\zeta}}} d\sigma dv_*
\\&\leq& 2\norm{f}_{L^\infty_{x,v}(\mu^{-\zeta})}\nu(v)^{-1}\int_{R^3\times\S^2} b(\cos \theta)\abs{v-v_*}^\gamma \mu_*^\zeta \:d\sigma dv_*.
\end{eqnarray*}
But then by definition of $\nu(v)$,
\begin{eqnarray*}
\int_{R^3\times\S^2} b(\cos \theta)\abs{v-v_*}^\gamma \mu_*^{\zeta} d\sigma dv_* &=& \nu(v) + \int_{R^3\times\S^2} b(\cos \theta)\abs{v-v_*}^\gamma (\mu_*^{\zeta}-\mu_*) d\sigma dv_*
\end{eqnarray*}
which implies, since $b$ is bounded and $\nu(v) \sim (1+\abs{v}^\gamma)$ (see $\eqref{nu0nu1}$),
$$\int_{R^3\times\S^2} b(\cos \theta)\abs{v-v_*}^\gamma \mu_*^{\zeta} d\sigma dv_* \leq \nu(v) + C(1-\zeta)\nu(v)\int_{R^3\times\S^2} (1+\abs{v_*}^\gamma)\abs{v_*}^2\mu_*^{\zeta} d\sigma dv_*.$$
To conclude we recall that $\zeta >1/2$ and the integral above on the right-hand side is uniformly bounded in $v_*$. Hence,
\begin{equation}\label{control K1}
\exists C_{K} >0, \quad\norm{K_1(f)}_{L^\infty_{x,v}(\mu^{-\zeta})} \leq \pa{2+ \frac{C_{K}}{2}(1-\zeta)}\norm{f}_{L^\infty_{x,v}(\mu^{-\zeta})}.
\end{equation}

The term $K_2$ is similar :
\begin{eqnarray*}
\abs{\nu(v)^{-1}\mu(v)^{-\zeta}K_2(f)(v)} &\leq& \mu^{1-\zeta}\nu^{-1}\int_{R^3\times\S^2} b(\cos \theta)\abs{v-v_*}^\gamma \abs{f_*} d\sigma dv_*
\\&\leq& \norm{f}_{L^\infty_{x,v}(\mu^{-\zeta})}\nu(v)^{-1}\int_{R^3\times\S^2} b(\cos \theta)\abs{v-v_*}^\gamma \mu_*^{\zeta} d\sigma dv_*
\\&\leq& \pa{1+ \frac{C_{K}}{2}(1-\zeta)}\norm{f}_{L^\infty_{x,v}(\mu^{-\zeta})}.
\end{eqnarray*}

Plugging the above and $\eqref{control K1}$ inside $\eqref{decomposition K}$ concludes the proof.
\end{proof}
\bigskip

We conclude this preliminary section with a statement of the $L^2-L^\infty$ theory that will be at the core of our main proof. It follows the idea developed in \cite{Gu6} that the $L^2$ theory of previous section could be used to construct a $L^\infty$ one by using the flow of characteristics to transfer pointwise estimates at $x-vt$ into an integral in the space variable. The proof can be found in \cite[Lemma 19]{Gu6} and holds as long as $L^\infty_{x,v}(w) \subset L^2_{x,v}(\mu^{-1/2})$.

\bigskip
\begin{prop}\label{prop:L2Linfty}
Let $\zeta$ be in $[1/2,1)$ and assume that there exist $T_0 >0$ and $C_{T_0}$, $\lambda >0$ such that for all $f(t,x,v)$ in $L^\infty_{x,v}(\mu^{-\zeta})$ solution to
\begin{equation}\label{lineqLinfty}
\partial_t f + v\cdot\nabla_xf = L(f)
\end{equation}
with Maxwell boundary condition $\alpha >0$ and initial datum $f_0$, the following holds
\begin{equation*}
\begin{split}
\forall t\in [0,T_0],\quad \norm{f(t)}_{L^\infty_{x,v}(\mu^{-\zeta})} \leq& e^{\lambda ( T_0-2t)}\norm{f_0}_{L^\infty_{x,v}(\mu^{-\zeta})}+C_{T_0}\int_0^{t} \norm{f(s)}_{L^2_{x,v}\pa{\mu^{-1/2}}}\:ds.
\end{split}
\end{equation*}
Then for all $0 < \tilde{\lambda}< \min\br{\lambda,\lambda_G}$, defined in Theorem \ref{theo:L2}, there exists $C>0$ independent of $f_0$ such that for all $f$ solution to $\eqref{lineqLinfty}$ in $L^\infty_{x,v}(\mu^{-\zeta})$  with $\Pi_{G}(f)=0$,
$$\forall t \geq 0, \quad \norm{f(t)}_{L^\infty_{x,v}(\mu^{-\zeta})} \leq C e^{-\tilde{\lambda} t}\norm{f_0}_{L^\infty_{x,v}(\mu^{-\zeta})}.$$
\end{prop}
\bigskip


\subsection{A crucial estimate between two consecutive collisions}\label{subsec:triplenorm}

The core of the $L^\infty$ estimate is a delicate control over the action of $K$ in between two rebounds. We define the following operator
\begin{equation}\label{triplenorm}
\mathfrak{K}(f)(t,x,v) = \mu^{-\zeta}(v)\int_{\max\br{0,t_{\min}(x,v)}}^t e^{-\nu(v)(t-s)}K(f(s))(x-(t-s)v,v)\:ds.
\end{equation}
We shall prove the following estimate of this functional along the flow of solutions. 

\bigskip
\begin{prop}\label{prop:triplenorm}
Let $\alpha$ in $(\sqrt{2/3},1]$. There exists $\zeta_\alpha$ in $(1/2,1)$, $\eps_\alpha$ in $(0,1)$, $C^{(1)}_\alpha >0$ and  $0<C^{(2)}_\alpha <1$ such that for any $T_0>0$ there exists $C_{T_0}>0$ such that if $f$ is solution to $\partial_t f = Gf$ with Maxwell boundary conditions then for all $t$ in $[0,T_0]$
\begin{equation*}
\begin{split}
\abs{e^{\eps\nu_0 t}\mathfrak{K}(f)(t,x,v)}\leq& C^{(1)}_\alpha \norm{f_0}_{L^\infty_{x,v}(\mu^{-\zeta})} + C_{T_0}\int_0^t\norm{f(s)}_{L^2_{x,v}(\mu^{-1/2})}ds
\\&+C^{(2)}_\alpha\pa{1-e^{-\nu(v)(1-\eps)\min\{t,t_{\min}(x,v)\}}}\sup\limits_{s\in[0,t]}\cro{e^{\eps\nu_0 s}\norm{f(s)}_{L^\infty_{x,v}(\mu^{-\zeta})}}.
\end{split}
\end{equation*}
\end{prop}
\bigskip

\begin{proof}[Proof of Proposition \ref{prop:triplenorm}]
We recall that $G = L-v\cdot\nabla_x = G_\nu +K$. Thanks to Theorem \ref{theo:semigroupGnu} with the weight $\mu^{-\zeta}$, $G_\nu$ generates a semigroup $S_{G_\nu}(t)$ in $L^\infty_{x,v}(\mu^{-\zeta})$. Moreover, Lemma \ref{lem:control K} implies that $K$ is a bounded operator in $L^\infty_{x,v}(\mu^{-\zeta})$. We can therefore write a Duhamel's form for the solution $f$ for almost every $(s,x,v_*)$ in $\R^+\times\bar{\Omega}\times\R^3$ :
\begin{equation}\label{Duhamel}
f(s,x,v_*) = S_{G_\nu}(s)f_0(x,v_*) + \int_0^s S_{G_\nu}(s-s_*)\cro{K(f(s_*))}(x,v_*)\:ds_*.
\end{equation}

\bigskip
\textbf{Step 1: New implicit form for $f$.} In what follows, $C_r$ will stand for any positive constant not depending on $f$ but depending on a parameter $r$. We recall $\eqref{definition Vj}$ the definition of the set of integration $\mathcal{V}_1$. We use the description $\eqref{fncharacteristics}$, $\alpha$ replacing $\Delta_n$, of $S_{G_\nu}(s-s_*)$ along characteristics until the first collision against the boundary that we denote $(t_1^*,x_1^*,v_1^*) =(t_{\min}(x,v_*), X_1(x,v_*),V_1(x,v_*))$ (see Section \ref{subsec:collisionfrequencycharacteristics}) . We deduce
\begin{equation*}
\begin{split}
f(s,x,v_*) =& J_0(s,x,v_*) + J_K(s,x,v_*)+ (1-\alpha)\mathbf{1}_{\br{t>t_1^*}}e^{-\nu(v_*)t_1^*}f(t-t_1^*,x_1^*,v_1^*)
\\&+ \alpha \mathbf{1}_{\br{s>t_1^*}} c_\mu e^{-\nu(v_*)t_1^*} \mu(v_*) \int_{\mathcal{V}_1} f(s-t_1^*,x_1^*,v_{1*})\pa{v_{1*}\cdot n(x_1^*)}\:dv_{1*},
\end{split}
\end{equation*}
where we defined
\begin{equation}\label{J0}
\begin{split}
J_0 = S_{G_\nu}(s)f_0(x,v_*) + \mathbf{1}_{\br{s>t_1^*}}e^{-\nu(v_*)t_1^*}S_{G_\nu}(s-t_1^*)f_0(x_1^*,v_*)
\end{split}
\end{equation}
and
\begin{equation}\label{JK}
J_K = \int_{\max\br{0,s-t_1^*}}^s e^{-\nu(v_*)(s-s_*)}K(f(s_*))(x-(s-s_*)v_*,v_*)\:ds.
\end{equation}
\par We iterate this formula inside the integral over $\mathcal{V}_1$. Using the notation $(\tilde{t_1},\tilde{x_1},\tilde{v_1})$ to denote the first backard collision starting from $(x_1^*,v_{1*})$ and $P_\Lambda$ for the diffuse boundary operator $\eqref{PLambda}$ we end up with a new implicit form for $f$
\begin{equation}\label{implicit f}
f(s,x,v_*) = \bar{J_0}(s,x,v_*) + \bar{J_K}(s,x,v_*) + J_{f}(s,x,v_*) + J_{diff}(s,x,v_*)
\end{equation}
with the following definitions
\begin{equation}\label{barJ0}
\bar{J_0}= J_0(s,x,v_*) + \alpha \mathbf{1}_{\br{s>t_1^*}} e^{-\nu(v_*)t_1^*}P_\Lambda(J_0(s-t_1^*,x_1^*))(v_*),
\end{equation}
\begin{equation}\label{barJK}
\bar{J_K}= J_K(s,x,v_*) + \alpha \mathbf{1}_{\br{s>t_1^*}} e^{-\nu(v_*)t_1^*} P_\Lambda(J_K(s-t_1^*,x_1^*))(v_*),
\end{equation}
\begin{equation}\label{Jf}
\begin{split}
J_f =&(1-\alpha)\mathbf{1}_{\br{s>t_1^*}}\Bigg[e^{-\nu(v_*)t_1^*}f(s-t_1^*,x_1^*,v_1^*) 
\\&+\alpha  e^{-\nu(v_*)t_1^*} c_\mu \mu(v_*)\int_{\mathcal{V}_1} \mathbf{1}_{\br{s-t_1^*>\tilde{t_1}}}e^{-\nu(v_{1*})\tilde{t_1}}f(s-t_1^*-\tilde{t_1},\tilde{x_1},\tilde{v_1})(v_{1*}\cdot n(x_1^*))dv_{1*}\Bigg], 
\end{split}
\end{equation}
and at last
\begin{equation}\label{Jdiff}
\begin{split}
J_{diff} =& \alpha^2\mathbf{1}_{\br{s>t_1^*}} c_\mu^2 e^{-\nu(v_*)t_1^*} \mu(v_*) \Bigg[\int_{\mathcal{V}_1} \mathbf{1}_{\br{s-t_1^*>\tilde{t_1}}} e^{-\nu(v_{1*})\tilde{t_1}} \mu(v_{1*})
\\&\quad\quad\quad\quad\quad\times\int_{\tilde{\mathcal{V}}_1} f(s-t_1^*-\tilde{t_1},\tilde{x_1},\tilde{v}_{1*})\pa{\tilde{v}_{1*}\cdot n(\tilde{x_1})}\pa{v_{1*}\cdot n(x_1)}\:d\tilde{v}_{1*} dv_{1*}\Bigg].
\end{split}
\end{equation}
We now bound the operator $\mathfrak{K}$ in $\eqref{triplenorm}$ for each of the terms above. We will bound most of the terms uniformly because for any $\eps$ in $[0,1]$, from Lemma \ref{lem:control K<3} and $e^{-\nu(v)(t-s)}\leq e^{-\eps\nu_0t}e^{-\nu(v)(1-\eps)(t-s)}e^{\eps\nu_0s}$ the following holds
\begin{eqnarray}
\abs{\mathfrak{K}(F)} &&\leq C_K(\zeta)e^{-\eps\nu_0 t}\pa{\int_{\max\br{0,t-t_1}}^t\nu(v)e^{-\nu(v)(1-\eps)(t-s)}ds} \sup\limits_{s\in[0,t]}\cro{e^{\eps \nu_0s}\norm{F(s)}_{L^\infty_{x,v}(\mu^{-\zeta})}} \nonumber
\\&&\leq \frac{C_K(\zeta)}{1-\eps}\pa{1-e^{-\nu(v)(1-\eps)\min\br{t,t_1}}}e^{-\eps \nu_0t}\sup\limits_{s\in[0,t]}\cro{e^{\eps \nu_0s}\norm{F(s)}_{L^\infty_{x,v}(\mu^{-\zeta})}} \label{important Linfty triple} 
\end{eqnarray}

\bigskip
\textbf{Step 2: Estimate for $\mathbf{\bar{J_0}}$.}
We straightforwardly bound $J_0$ in $\eqref{J0}$ thanks to the exponential decay with rate $\eps\nu_0$ of $S_{G_\nu}(t)$ in Theorem \ref{theo:semigroupGnu} (that also holds on the boundary) for all $(x,v)$ in $\bar{\Omega}\times B(0,R)$.
\begin{equation*}
\begin{split}
\abs{J_{0}} \leq \frac{C_\eps}{\mu^{-\zeta}(v)} e^{-\eps\nu_0 t}\norm{f_0}_{L^\infty_{x,v}(\mu^{-\zeta})}\cro{1+ \frac{\mu^{-\zeta}(v)}{\mu^{-\zeta}(v_1)}+\pa{c_\mu \mu^{-\zeta}(v)\mu(v)\int_{\mathcal{V}_1}\frac{v_{1*}\cdot n(x_1)}{\mu^{-\zeta}(v_{1*})}\:dv_{1*}}}.
\end{split}
\end{equation*}
To conclude we use the fact that $\abs{v_1}=\abs{v}$ and $\mu^{-\zeta}$ is radially symmetric. This yields
\begin{equation}\label{J0 Linfty}
\abs{J_0} \leq C_\eps \mu^{-\zeta}(v)^{-1}e^{-\eps\nu_0 t}\norm{f_0}_{L^\infty_{x,v}(\mu^{-\zeta})}
\end{equation}
Bounding $J_0(t-t_1^*,x_1,v_{1*})$ exactly the same way yields the same bound for the full $\bar{J_0}$ in $\eqref{barJ0}$. We conclude thanks to $\eqref{important Linfty triple}$
\begin{equation}\label{bound barJ0}
\forall 0\leq \eps\leq 1,\:\forall (t,x,v), \quad\abs{\mathfrak{K}(J_0)(t,x,v)} \leq C_{\eps,\zeta}\norm{f_0}_{L^\infty_{x,v}(\mu^{-\zeta})}e^{-\eps\nu_0 t}.  
\end{equation}

\bigskip
\textbf{Step 3: Estimate for $\mathbf{\bar{J_K}}$.}
We write $K$ under its kernel form with Lemma \ref{lem:control K} in $\eqref{JK}$. To shorten notations and as we shall legitimate the following change of variable we use $y(v_*) = x-(t-s)v-(s-s_*)v_*$. Since $\nu(v) \geq \nu_0$ we see
\begin{equation}\label{JK start}
\begin{split}
\abs{\mathfrak{K}(J_K)} \leq \int_{0}^t\int_0^s e^{-\nu_0(t-s_*)}\int_{\R^3}&\abs{k(v,v_*)}\frac{\mu^{-\zeta}}{\mu^{-\zeta}_*}\mathbf{1}_{\br{y(v_*) \in \bar{\Omega}}}
\\&\int_{\R^3}\abs{k(v_*,v_{**})}\frac{\mu^{-\zeta}_*}{\mu^{-\zeta}_{**}}\abs{f(s_*,y(v_*),v_{**})\mu^{-\zeta}_{**}}dv_{**}dv_*ds_*ds
\end{split}
\end{equation}
We take any $R\geq 1$. When $\abs{v}\geq R$ we use Lemma \ref{lem:control K} on $k(v,v_*)$ and when $\abs{v}\leq R$ and $\abs{v_{*}}\geq R$ it follows that $\abs{v-v_*}\geq R$ and we thus use Lemma \ref{lem:control K} again with $\abs{k(v,v_*)} \leq e^{-\frac{\eps_\zeta}{8}R^2}\abs{k(v,v_*)e^{\frac{\eps_\zeta}{8}\abs{v-v_*}^2}}$. At last, if $\abs{v_*}\leq 2R$ and $\abs{v_{**}}\geq 3R$ we us Lemma \ref{lem:control K} one more time on $\abs{k(v_*,v_{**})} \leq e^{-\frac{\eps_\zeta}{8}R^2}\abs{k(v_*,v_{**})e^{\frac{\eps_\zeta}{8}\abs{v_*-v_{**}}^2}}$. Taking the $L^\infty_{x,v}(\mu^{-\zeta})$-norm of $f$ out of the integral and decomposing the exponential decay as for $\eqref{important Linfty triple}$, we infer that for all $\eps$ in $[0,1]$ and any $R\geq 1$:
\begin{equation}\label{JK v or v* or v** large}
\begin{split}
&\br{\mbox{terms in \eqref{JK start} outside $\br{\abs{v}\leq R\cap \abs{v_*}\leq 2R\cap\abs{v_{**}}\leq 3R}$}} 
\\&\quad\quad\quad\quad\quad\leq C_{\eps,\zeta}\pa{\frac{1}{1+R}+ e^{-\frac{\eps_\zeta}{8}R^2}}e^{-\eps\nu_0t} \sup\limits_{s\in[0,t]}\cro{e^{\eps\nu_0 s}\norm{f(s)}_{L^\infty_{x,v}(\mu^{-\zeta})}}
\end{split}
\end{equation}
\par In order to deal with the remaining terms in $\eqref{JK start}$ we first approximate $k(\cdot,\cdot)$ by a smooth and compactly supported function $k_{R}$ uniformly in the following sense:
\begin{equation}\label{kr}
\sup\limits_{\abs{V}\leq 3R}\int_{\abs{v_{*}}\leq 3R} \abs{k(V,v_{*}) - k_R(V,v_{*})}\mu^{-\zeta}(V)\:dv_{*} \leq \frac{1}{1+R}.
\end{equation}
We decompose $k=k_R + (k-k_R)$ and bound the terms where $k-k_R$ are appearing as before and get (remember that $k_R$ is compactly supported): the terms in $\eqref{JK start}$ where $\abs{v}\leq R$ and $\abs{v_*}\leq 2R$ and $\abs{v_{**}}\leq 3R$ are bounded from above by 
\begin{equation*}
\begin{split}
\frac{C_{\eps,\zeta}}{1+R}&e^{-\eps\nu_0t} \sup\limits_{s\in[0,t]}\cro{e^{\eps\nu_0 s}\norm{f(s)}_{L^\infty_{x,v}(\mu^{-\zeta})}} 
\\&+ C_{R,\zeta} \int_{0}^t\int_0^s e^{-\nu_0(t-s_*)}\int_{\br{\abs{v_*}\leq 2R}}\int_{\br{\abs{v_{**}\leq 3R}}}\mathbf{1}_{\br{y(v_*) \in \bar{\Omega}}}\abs{f(s_*,y(v_*),v_{**})}
\end{split}
\end{equation*}
At last, we would like to apply the change of variable $v_* \mapsto y(v_*)$ which has Jacobian $(s-s_*)^{-3}$, and so is legitimate if $s-s_* \geq \eta>0$. We thus consider $\eta >0$ and decompose the integral over $s_*$ into an integral on $[s-\eta,s]$ and an integral on $[0,s-\eta]$. In the first one we bound as before which gives the following upper bound
\begin{equation*}
\begin{split}
C_{\eps,\zeta}\pa{\frac{1}{1+R}+\eta}&e^{-\eps\nu_0t} \sup\limits_{s\in[0,t]}\cro{e^{\eps\nu_0 s}\norm{f(s)}_{L^\infty_{x,v}(\mu^{-\zeta})}} 
\\&+ C_{R,\zeta} \int_{0}^t\int_0^{s-\eta} e^{-\nu_0(t-s_*)}\int_{\overset{\br{\abs{v_*}\leq 2R}}{\br{\abs{v_{**}\leq 3R}}}}\mathbf{1}_{\br{y(v_*) \in \bar{\Omega}}}\abs{f(s_*,y(v_*),v_{**})}.
\end{split}
\end{equation*}
We now perform $v_* \mapsto y(v_*)$ inside the remaining integral term which makes $\norm{f(s_*)}_{L^1_{x,v}}$ appear, which is itself controlled by $\norm{f(s_*)}_{L^2_{x,v}(\mu^{-1/2})}$ thanks to Cauchy-Schwarz inequality. We therefore proved that
\begin{equation}\label{JK v v* and v** small start}
\begin{split}
&\br{\mbox{terms in \eqref{JK start} in $\br{\abs{v}\leq R\cap \abs{v_*}\leq 2R\cap\abs{v_{**}}\leq 3R}$}} 
\\&\quad\leq C_{\eps,\zeta}\pa{\frac{1}{1+R}+\eta}e^{-\eps\nu_0t} \sup\limits_{s\in[0,t]}\cro{e^{\eps\nu_0 s}\norm{f(s)}_{L^\infty_{x,v}(\mu^{-\zeta})}} + C_{R,\zeta,\eta} \int_{0}^t\norm{f(s)}_{L^2_{x,v}(\mu^{-1/2})}.
\end{split}
\end{equation}
Gathering $\eqref{JK v or v* or v** large}$ and $\eqref{JK v v* and v** small start}$ inside $\eqref{JK start}$ finally yields
\begin{equation}\label{bound JK}
\begin{split}
\abs{\mathfrak{K}(J_K)} \leq& C_{\eps,\zeta}\pa{\eta + \frac{1}{1+R}+e^{-\frac{\eps_\zeta}{8}R^2}}e^{-\eps\nu_0t} \sup\limits_{s\in[0,t]}\cro{e^{\eps\nu_0 s}\norm{f(s)}_{L^\infty_{x,v}(\mu^{-\zeta})}}  
\\&+ C_{R,\zeta,\eta} \int_{0}^t\norm{f(s)}_{L^2_{x,v}(\mu^{-1/2})}.
\end{split}
\end{equation}

\bigskip
It remains to deal with the second term in $\bar{J_K}$ given by $\eqref{barJK}$. However, by definition of the boundary operator $P_\Lambda$ and since $\abs{v_{1*}\cdot n(x_1^*)} \leq C_\zeta \mu^{-\zeta}(v_{1*})$ it follows directly that
\begin{equation}\label{PLambda JK}
\abs{\mu^{-\zeta}(v_*)P_\Lambda(J_K(s-t^*_1,x_1^*))(v_*)} \leq c_\mu C_\zeta \mu_*^{1-\zeta}\int_{\R^3} J_K(s-t_1^*,x_1^*,v_{1*})\mu^{-\zeta}(v_{1*})\:dv_{1*}.
\end{equation}
We bound the integral term uniformly exactly as $\eqref{bound JK}$ for $J_K$ and since $\zeta <1$ the integral over $v_*$ of $\mu(v_*)^{1-\zeta}$ is finite. As a conclusion, for all $\eps$ in $(0,1)$, $R\geq 1$ and $\eta >0$,
\begin{equation}\label{bound barJK}
\begin{split}
\abs{\mathfrak{K}(\bar{J_K})} \leq& C_{\eps,\zeta}\pa{\eta + \frac{1}{1+R}+e^{-\frac{\eps_\zeta}{8}R^2}}e^{-\eps\nu_0t} \sup\limits_{s\in[0,t]}\cro{e^{\eps\nu_0 s}\norm{f(s)}_{L^\infty_{x,v}(\mu^{-\zeta})}} 
\\& + C_{R,\zeta,\eta} \int_{0}^t\norm{f(s)}_{L^2_{x,v}(\mu^{-1/2})}.
\end{split}
\end{equation}

\bigskip
\textbf{Step 4: Estimate for $\mathbf{J_f}$.}
The control of $J_f$ given by $\eqref{Jf}$ is straightforward for the first term from $\eqref{important Linfty triple}$ by taking the $L^\infty_{x,v}(\mu^{-\zeta})$-norm of $f$ (remember that $\mu(v_*) = \mu(v_1^*)$ for a specular reflection). The second term is dealt with the same way and noticing that, in the spirit of Lemma \ref{lem:control K<3},
\begin{equation}\label{bound PLambda}
c_\mu \mu^{1-\zeta}(v_*) \int_{\mathcal{V}_1}\mu^{-\zeta}(v_{1*})\pa{v_{1*}\cdot n(x_1^*)}dv_{1*} \leq 1 + C_0(1-\zeta)
\end{equation}
where $C_0>0$ is a universal constant. In the end,
\begin{equation}\label{bound Jf}
\begin{split}
\abs{\mathfrak{K}(J_f)} \leq \mathbf{1}_{\br{t\geq t_1}}&\frac{(1-\alpha)C_K(\zeta)}{1-\eps}\pa{1+\alpha(1+C_0(1-\zeta))}\pa{1-e^{-\nu(v)(1-\eps)\min\br{t,t_1}}}
\\&\times e^{-\eps \nu_0t}\sup\limits_{s\in[0,t]}\cro{e^{\eps \nu_0s}\norm{f(s)}_{L^\infty_{x,v}(\mu^{-\zeta})}}.
\end{split}
\end{equation}

\bigskip
\textbf{Step 5: Estimate for $\mathbf{J_{diff}}$.}
The last term $J_{diff}$ given by $\eqref{Jdiff}$ is treated thanks to a change of variable on the boundary and a trace theorem from \cite{EGKM}. First,
\begin{equation}\label{Jdiff start}
\begin{split}
\abs{\mu^{-\zeta}_* J_{diff}} \leq& \alpha^2 c_\mu^2 \mu^{1-\zeta}(v_*)
\\&\int_{\R^3}d\tilde{v}_{1*}\int_{\R^3}dv_{1*} e^{-\nu_0(t_1^*+\tilde{t_1})}\abs{v_{1*}}\mu(v_{1*})\abs{n(\tilde{x_1})\cdot\tilde{v}_{1*}}\abs{f(s-t_1^*-\tilde{t_1},\tilde{x_1},\tilde{v}_{1*})}.
\end{split}
\end{equation}
\par Using the spherical coordinate $v_{1*} = r_{1*}u_{1*}$ with $u_{1*}$ in $\S^2$ we have by definition 
\begin{equation}\label{tilde tmin and x1}
\tilde{t_1} = t_{\min}^*(x_1,v_{1*}) = \frac{t_1^*(x_1,u_{1*})}{r_{1*}} \quad\mbox{and}\quad \tilde{x_1} = x_1 - \tilde{t_1}v_{1*} = x_1 - t_1^*(x_1,u_{1*})u_{1*}
\end{equation}
and, using the parametrization $(\theta,\phi)$ of the sphere $\S^2$ we compute
\begin{equation}\label{I spherical coordinates}
\begin{split}
\abs{\mu^{-\zeta}_* J_{diff}} \leq& C \int_{\R^3}d\tilde{v}_{1*}\Bigg(\int_{0}^\infty\int_{0}^{2\pi}\int_0^\pi e^{-\nu_0(t_1^*+\tilde{t_1})}\abs{f(s-t_1^*-\tilde{t_1},\tilde{x_1},\tilde{v}_{1*})}
\\&\abs{n(\tilde{x_1})\cdot\tilde{v}_{1*}}\mu(r_{1*})r_{1*}^3\sin \theta \:dr_{1*}d\theta d\phi\Bigg).
\end{split}
\end{equation}
\par We want to perform the change of variable $(r_{1*},\theta,\phi) \mapsto (s-t_1^* - \tilde{t_1},\tilde{x_1})$. Thanks to $\eqref{tilde tmin and x1}$ we have
$$\partial_{r_{1*}}\pa{\tilde{t_1}} = -\frac{1}{r_{1*}^2}t_1^*(x_1,u_1) \quad\mbox{and}\quad \partial_{r_{1*}}\tilde{x_1}=0.$$
The jacobian of $(\theta,\phi) \mapsto \tilde{x_1}$ has been calculated in \cite[Proof of Lemma 2.3]{EGKM}. More precisely, from \cite[(2.8)]{EGKM} we have, calling $\xi(x)$ a parametrization of $\partial\Omega$
$$\abs{\mbox{det}\pa{\partial_{\theta,\phi}\tilde{x_1}}} \geq \frac{t_1^*(x_1,u_{1*})^2 \sin \theta}{\abs{n(\tilde{x_1})\cdot u_{1*}}}\times \frac{\partial_3 \xi(\tilde{x_1})}{\nabla\xi(\tilde{x_1})}.$$
Therefore, we bound from below the full jacobian by
$$\abs{\mbox{det}\pa{\partial_{r_{1*},\theta,\phi}\pa{t-t_1^*-\tilde{t_1},\tilde{x_1}}}}\geq \frac{t_1^*(x_1,u_{1*})^3 \sin \theta}{r_{1*}^2\abs{n(\tilde{x_1})\cdot u_{1*}}}\times \frac{\partial_3 \xi(\tilde{x_1})}{\nabla\xi(\tilde{x_1})}.$$
The important fact is that $x_1$ belongs to $\partial\Omega$ and $u_{1*}$ is on $\S^2$. With these conditions, we know from \cite[(40)]{Gu6}, that $t_1^*(x_1,u_{1*}) \geq C_\Omega\abs{n(x_1)\cdot u_{1*}}$ and hence
$$\abs{\mbox{det}\pa{\partial_{r_{1*},\theta,\phi}\pa{t-t_1^*-\tilde{t_1},\tilde{x_1}}}}\geq C_{\Omega}\frac{\abs{n(x_1)\cdot u_{1*}}^3 \sin \theta}{r_{1*}^2}\times \frac{\partial_3 \xi(\tilde{x_1})}{\nabla\xi(\tilde{x_1})}.$$
\par We therefore need $\abs{n(x_1)\cdot u_{1*}}$ to be non zero and we thus decompose $\eqref{Jdiff start}$ into two integrals. The first one on $\br{\abs{n(x_1)\cdot u_{1*}} \leq \eta}$ and the second on $\br{\abs{n(x_1)\cdot u_{1*}} \geq \eta}$. On the first one we take the $L^\infty_{x,v}(\mu^{-\zeta})$-norm of $f$ out and on the second one we use the spherical coordinates $\eqref{I spherical coordinates}$ and apply, as announced, the change of variable $(r_{1*},\theta,\phi) \mapsto (s-t_1^* - \tilde{t_1},\tilde{x_1})$, which is legitimate on this set. It follows, playing with the exponential decay as for $\eqref{important Linfty triple}$,
\begin{equation*}
\begin{split}
&\abs{\mu^{-\zeta}_* J_{diff}} 
\\&\:\leq C_\eps e^{-\eps\nu_0 s}\sup\limits_{s_*\in[0,s]}\cro{e^{\eps \nu_0 s_*}\norm{f(s_*)}_{L^\infty_{x,v}(\mu^{-\zeta})}}\pa{\int_{\R^3}\int_{\abs{n(x_1)\cdot u_{1*}} \leq \eta}\frac{\abs{\tilde{v}_{1*}}\abs{v_{1*}}\mu(v_{1*})}{\mu^{-\zeta}(\tilde{v}_{1*})}d\tilde{v}_{1*}dv_{1*}}
\\&\:\quad+ C_{\eta,\eps} e^{-\eps \nu_0 s}\int_{\R^3}d\tilde{v}_{1*}\int_0^s\int_{\partial\Omega}e^{\eps \nu_0 s_*}\abs{f(s_*,y,\tilde{v}_{1*})}\pa{r_{1*}^5\mu(r_{1*})}\abs{n(y)\cdot\tilde{v}_{1*}}\:dS(y)ds_*
\end{split}
\end{equation*}
where we recall that $dS(y) = \abs{\partial_3 \xi(y)}^{-1}\abs{\nabla\xi(y)}dy$ is the Lebesgue measure on $\partial\Omega$ and we also denoted $r_{1*}=r_{1*}(s,y)$. Since $\zeta \geq 1/2$, the integral in the first term on the right-hand side above uniformly tends to $0$ as $\eta$ goes to $0$. Hence, for any $\eta >0$,
\begin{equation}\label{conclusion Jdiff}
\begin{split}
\abs{\mu^{-\zeta}_* J_{diff}} \leq& C_\eps \eta e^{-\eps\nu_0 s}\sup\limits_{s_*\in[0,s]}\cro{e^{\eps\nu_0 s_*}\norm{f(s_*)}_{L^\infty_{x,v}(\mu^{-\zeta})}}
\\&+ C_{\eps,\eta} e^{-\eps \nu_0 s}\int_0^s\int_{\Lambda}e^{\eps \nu_0 s_*}\abs{f(s_*,y,\tilde{v}_{1*})}\:d\lambda(y,\tilde{v}_{1*})ds_*
\end{split}
\end{equation}
where $d\lambda(x,v)$ is the boundary measure on the phase space boundary $\Lambda$ (see Section \ref{subsec:notations}).

\bigskip
We now decompose $\Lambda$ into
$$\Lambda_\eta = \br{(x,v)\in \Lambda, \quad \abs{n(x)\cdot v}\leq \eta \:\mbox{or}\: \abs{v}\leq \eta } \quad\mbox{and}\quad \Lambda-\Lambda_\eta.$$
Since $\zeta \geq 1/2$ there exists a uniform $C>0$ such that
\begin{equation}\label{integral minus Lambda eps}
\int_0^s\int_{\Lambda_\eta}e^{\eps\nu_0 s_*}\abs{f(s_*,y,\tilde{v}_{1*})}\:d\lambda(y,\tilde{v}_{1*}) \leq C s \eta \sup\limits_{s_*\in[0,s]}\cro{e^{\eps \nu_0 s_*}\norm{f(s_*)}_{L^\infty_{x,v}(\mu^{-\zeta})}}.
\end{equation}
For the integral on $\Lambda-\Lambda_\eta$ we use the trace lemma \cite[Lemma 2.1]{EGKM} that states
$$\int_{\Lambda-\Lambda_\eta}\abs{f(s_*,y,\tilde{v}_{1*})}\:d\lambda(y,\tilde{v}_{1*}) \leq C_\eta \norm{f(s_*)}_{L^1_{x,v}} + \int_0^{s_*}\cro{\norm{f(s_{**})}_{L^1_{x,v}}+ \norm{L(f(s_{**}))}_{L^1_{x,v}}}.$$ 
Thanks to Cauchy-Schwarz inequality and the boundedness property (see Section \ref{subsec:LL2}) of $L$ in $L^2_{x,v}(\mu^{-1/2})$ we get
\begin{equation}\label{integral Lambda eps}
\int_0^s\int_{\Lambda-\Lambda_\eta}e^{\eps\nu_0 s_*}\abs{f(s_*,y,\tilde{v}_{1*})}\:d\lambda(y,\tilde{v}_{1*}) \leq C_\eta s e^{\eps\nu_0 s}\int_0^s \norm{f(s_*)}_{L^2_{x,v}\pa{\mu^{-1/2}}}ds_* .
\end{equation}
\par Gathering $\eqref{integral minus Lambda eps}$ and $\eqref{integral Lambda eps}$ with $\eqref{conclusion Jdiff}$ we conclude
\begin{equation}\label{Jdiff Linfty}
\abs{\mu^{-\zeta}_* J_{diff}} \leq C_\eps \eta s e^{-\eps\nu_0 s}\sup\limits_{s_*\in[0,s]}\cro{e^{\eps\nu_0 s_*}\norm{f(s_*)}_{L^\infty_{x,v}(\mu^{-\zeta})}} + C_{\eps,\eta} s \int_0^s \norm{f(s_*)}_{L^2_{x,v}\pa{\mu^{-1/2}}}.
\end{equation}
Again, plugging this uniform bound in $\eqref{important Linfty triple}$ yields for any $\eps$ in $(0,1)$ and $\eta >0$,
\begin{equation}\label{bound Jdiff}
\begin{split}
\abs{\mathfrak{K}(J_{diff})} \leq& C_{\eps,\zeta}\eta t e^{-\eps\nu_0 t}\sup\limits_{s\in[0,t]}\cro{e^{\eps\nu_0 s}\norm{f(s)}_{L^\infty_{x,v}(\mu^{-\zeta})}} + C_{\eps,\eta,\zeta} t \int_0^t \norm{f(s)}_{L^2_{x,v}\pa{\mu^{-1/2}}}.
\end{split}
\end{equation}

\bigskip
\textbf{Step 6: Choice of constants and conclusion.}
We consider $T_0>0$ and $t$ in $[0,T_0]$. We bound the full $\mathfrak{K}(f)$ by gathering $\eqref{bound barJ0}$, $\eqref{bound barJK}$, $\eqref{bound Jf}$ and $\eqref{bound Jdiff}$ into $\eqref{implicit f}$. It yields, for any $\eps$, $\eta$ in $(0,1)$ and $R\geq 1$,
\begin{equation*}
\begin{split}
\abs{\mathfrak{K}(f)} \leq& C_{\eps,\zeta} e^{-\eps\nu_0 t} \norm{f_0}_{L^\infty_{x,v}(\mu^{-\zeta})} + C_{\eps,\zeta,\eta} t \int_0^t \norm{f(s)}_{L^2_{x,v}(\mu^{-1/2})}ds
\\&+\pa{\frac{C_\alpha(\zeta)}{1-\eps}(1-e^{-\nu(v)(1-\eps)\min\{t,t_1\}})+C_{\eps,\zeta}(\eta(1+t)+\frac{1}{1+R})} 
\\&\quad\quad\times e^{-\eps\nu_0 t}\sup\limits_{s\in[0,t]}\cro{e^{\eps\nu_0 s}\norm{f(s)}_{L^\infty_{x,v}(\mu^{-\zeta})}}.
\end{split}
\end{equation*}
We used the following definition
$$C_\alpha(\zeta) = (1-\alpha)\pa{3+C_K(1-\zeta)}\cro{1+\alpha(1+C_0(1-\zeta))}.$$
For $\alpha$ in $(\sqrt{2/3},1]$, $\lim_{\zeta\to 1} C_\alpha(\zeta) = 3(1-\alpha)(1+\alpha)<1$. We therefore choose our parameters as
\begin{enumerate}
\item $\zeta$ sufficiently close to $1$ such that $C_\alpha(\zeta)  <1$,
\item $\eps$ sufficiently small so that $C_\alpha(\zeta)/(1-\eps) <1$,
\item $R$ large enough and $\eta$ small enough such that $C_{\eps,\zeta}(\eta(1+T_0)+1/(1+R)) < 1- C_\alpha(\zeta)/(1-\eps)$. 
\end{enumerate}
Such choices terminates the proof.
\end{proof}
\bigskip


\subsection{Semigroup generated by the linear operator}\label{subsec:proofLinfty}

This subsection is dedicated to the proof of Theorem \ref{theo:semigroupLinfty}, that is uniqueness and existence of solutions to $\eqref{lineqLinfty}$ together with the Maxwell boundary condition $\eqref{mixedBC}$ in $L^\infty_{x,v}(\mu^{-\zeta})$. Moreover if $f_0$ satisfies the conservation laws then $f$ will be proved to decay exponentially.

\bigskip
\begin{proof}[Proof of Theorem \ref{theo:semigroupLinfty}]
Let $f_0$ be in $L^\infty_{x,v}(\mu^{-\zeta})$ with $1/2<\zeta <1$. If $f$ is solution to $\eqref{lineqLinfty}$ in $L^\infty_{x,v}(\mu^{-\zeta})$ with initial datum $f_0$ then $f$ belongs to $L^2_{x,v}\pa{\mu^{-1/2}}$ and $f(t) = S_{G}(t)f_0$ is in the latter space. This implies first uniqueness and second that $\mbox{Ker}(G)$ and $\pa{\mbox{Ker}(G)}^\bot$ are stable under the flow of the equation $\eqref{lineqLinfty}$. It suffices to consider $f_0$ such that $\Pi_G(f_0)=0$ and to prove existence and exponential decay of solutions to $\eqref{lineqLinfty}$ in $L^\infty_{x,v}(\mu^{-\zeta})$ with initial datum $f_0$.
\par Thanks to boundedness property of $K$, the Duhamel's form $\eqref{Duhamel}$ is a contraction, at least for small times. We thus have existence of solutions on small times and proving the exponential decay will also imply global existence. For now on we consider $f$ as described in the Duhamel's expression $\eqref{Duhamel}$.

\bigskip
Looking at previous section and using $t_1 = t_{\min}(x,v)$, $f$ can be implicitely written as $\eqref{implicit f}$
\begin{equation*}
\begin{split}
f(t,x,v) =& \bar{J_0}(t,x,v)+ \mu^{\zeta}\mathfrak{K}(f)(t,x,v) + \alpha e^{-\nu(v)t_1}P_\Lambda(J_K(t-t_1,x_1))(t,x,v) 
\\&+ \mathbf{1}_{\br{t>t_1}}\cro{J_f(t,x,v) + J_{diff}(t,x,v)}.
\end{split}
\end{equation*}
The $L^\infty_{x,v}(\mu^{-\zeta})$-norm of each of these terms has already been estimated. More precisely, $\bar{J_0}$ by $\eqref{J0 Linfty}$, $\mathfrak{K}(f)$ by Proposition \ref{prop:triplenorm}, $P_\Lambda(J_K)$ by $\eqref{PLambda JK}$, $J_f$ is direct from $\eqref{Jf}$ and $\eqref{bound PLambda}$ and finally $J_{diff}$ by $\eqref{Jdiff Linfty}$. With the same kind of choices of constant as in Step 6 of the proof of Proposition \ref{prop:triplenorm} (note that $\zeta_\alpha$ and $\eps_\alpha$ are the same) we end up with
\begin{equation*}
\begin{split}
e^{\eps\nu_0 t}\norm{f(t)}_{L^\infty_{x,v}(\mu^{-\zeta})} \leq& C_{\eps,\alpha}\norm{f_0}_{L^\infty_{x,v}(\mu^{-\zeta})}+\sup\limits_{x,v}\pa{C_\infty(x,v)}\sup\limits_{s\in[0,t]}\cro{e^{\eps\nu_0 s}\norm{f(s)}_{L^\infty_{x,v}(\mu^{-\zeta})}}
\\& + C_{T_0}\int_0^t\norm{f(s)}_{L^2_{x,v}(\mu^{-1/2})}ds
\end{split}
\end{equation*}
where
$$C_\infty= C_2^{(\alpha)}(1-e^{-\nu(v)(1-\eps)\min\{t,t_1\}}) + \mathbf{1}_{\br{t>t_1}}(1-\alpha)\pa{1+\alpha(1+C_0(1-\zeta))}e^{-\nu(v)(1-\eps)t_1}$$
and thus, with our choice of constants,
$$C_\infty(x,v) \leq \max\br{C_2^{(\alpha)},\:(1-\alpha)(1+\alpha)} < 1$$
which implies
$$\forall t\in[0,T_0], \norm{f(t)}_{L^\infty_{x,v}(\mu^{-\zeta})} \leq C_1e^{-\eps\nu_0 t}\norm{f_0}_{L^\infty_{x,v}(\mu^{-\zeta})}+ C_{T_0}\int_0^t\norm{f(s)}_{L^2_{x,v}(\mu^{-1/2})}ds.$$
To conclude we choose $T_0$ large enough such that $C_1e^{-\frac{\eps}{2}\nu_0 T_0} \leq 1$ so that assumptions of the $L^2-L^\infty$ theory of Proposition \ref{prop:L2Linfty} are fulfilled so we can apply it, thus concluding the proof.
\end{proof}
\bigskip

%% file: fullcauchy.tex
\section{Perturbative Cauchy theory for the full nonlinear equation}\label{sec:fullcauchy}

This section is dedicated to establishing a Cauchy theory for the perturbed equation
\begin{equation}\label{perturbedBEfull}
\partial_t f + v\cdot\nabla_x f = L(f) + Q(f,f),
\end{equation}
together with the Maxwell boundary condition $\eqref{mixedBC}$ in spaces $L^\infty_{x,v}\pa{m}$, where $m$ is a polynomial or a stretched exponential weight. More precisely we shall prove that for any small initial datum $f_0$ satisfying the conservation of mass $\Pi_G(f_0)=0$ there exists a unique solution to $\eqref{perturbedBEfull}$ and that this solution decays exponentially fast and also satisfies the conservation of mass.
\par We divide our study in two different subsections. First, for any small $f_0$, we build a solution that satisfies the conservation law and decays exponentially fast; this is the purpose of Subsection \ref{subsec:existence}. Second, we prove the uniqueness to $\eqref{perturbedBEfull}$ when the initial datum is small in Subsection \ref{subsec:uniqueness}.
\bigskip


\subsection{Existence of solutions with exponential decay}\label{subsec:existence}

The present subsection is dedicated to the following proof of existence.

\bigskip
\begin{theorem}\label{theo:existenceexpodecay}
Let $\alpha$ be in $(\sqrt{2/3},1]$ and let $m=e^{\kappa_1\abs{v}^{\kappa_2}}$ with $\kappa_1 >0$ and $\kappa_2$ in $(0,2)$ or $m=\langle v \rangle^k$ with $k> k_\infty$. There exists $\eta >0$ such that for any $f_0$ in $L^\infty_{x,v}(m)$ with 
$$\Pi_G(f_0)=0 \quad\mbox{and}\quad \norm{f_0}_{L^\infty_{x,v}(m)}\leq \eta,$$
there exist at least one solution $f$ to the Boltzmann equation $\eqref{perturbedBEfull}$ with Maxwell boundary condition and with $f_0$ as an initial datum. Moreover, $f$ satisfies the conservation of mass and there exist $C$, $\lambda>0$ such that
$$\forall t\geq 0,\quad \norm{f(t)}_{L^\infty_{x,v}(m)}\leq Ce^{-\lambda t}\norm{f_0}_{L^\infty_{x,v}(m)}.$$
\end{theorem}
\bigskip

As explained in Section \ref{subsec:descriptionstrategy}, we decompose $\eqref{perturbedBEfull}$ into a system of differential equations. More precisely, we shall decompose $G = L -v\cdot\nabla_x$ as $G= A+B$ in the spirit of \cite{GMM}, where $B$ is ``small'' compared to $\nu(v)$ and $A$ has a regularising effect. We then shall construct $(f_1,f_2)$ solutions to the following system of equation

\begin{eqnarray}
\partial_t f_1 &=& B f_1 + Q(f_1+f_2,f_1+f_2) \quad\mbox{and}\quad f_1(0,x,v)=f_0(x,v),\label{f1}
\\\partial_t f_2 &=& G f_2 + Af_1 \quad\mbox{and}\quad f_2(0,x,v)=0,\label{f2}
\end{eqnarray}
each of the functions satisfying the Maxwell boundary condition. Note that for such functions, the function $f=f_1+f_2$ would be a solution to $\eqref{perturbedBEfull}$ with Maxwell boundary condition and $f_0$ as initial datum.
\par Subsection \ref{subsubsec:decomposition} explicitly describes the decomposition $G=A+B$ and gives some estimates on $A$, $B$ and $Q$. Subsections \ref{subsubsec:f1} and \ref{subsubsec:f2} deal with each differential equation $\eqref{f1}$ and $\eqref{f2}$ respectively. Finally, Subsection \ref{subsubsec:existence} combines the previous theories to construct a solution to the full nonlinear perturbed Boltzmann equation.
\bigskip


\subsubsection{Decomposition of the linear operator and first estimates}\label{subsubsec:decomposition}

We follow the decomposition proposed in \cite{GMM}.
\par For $\delta$ in $(0,1)$, we consider $\Theta_\delta = \Theta_\delta(v,v_*,\sigma)$ in $C^\infty$ that is bounded by one everywhere, is exactly one on the set
$$\left\{\abs{v}\leq \delta^{-1}    \quad\mbox{and}\quad 2\delta\leq\abs{v-v_*}\leq \delta^{-1}    \quad\mbox{and}\quad \abs{\mbox{cos}\:\theta} \leq 1-2\delta \right\}$$
and whose support is included in
$$\left\{\abs{v}\leq 2\delta^{-1}    \quad\mbox{and}\quad \delta\leq\abs{v-v_*}\leq 2\delta^{-1}    \quad\mbox{and}\quad \abs{\mbox{cos}\:\theta} \leq 1-\delta \right\}.$$
We define the splitting
$$G = A^{(\delta)} +B^{(\delta)},$$
with
$$A^{(\delta)} h (v) = C_\Phi\int_{\R^3\times\mathbb{S}^2}\Theta_\delta\left[\mu'_*h' + \mu'h'_* - \mu h_*\right]b\left(\mbox{cos}\:\theta\right)\abs{v-v_*}^\gamma\:d\sigma dv_*$$
and
$$B^{(\delta)} h (v) = B^{(\delta)}_{2} h (v) -\nu(v) h(v) -v\cdot\nabla_x h(v)= G_\nu h(v) + B^{(\delta)}_{2} h (v),$$
where
$$B^{(\delta)}_{2} h (v) = \int_{\R^3\times\mathbb{S}^2}\left(1-\Theta_\delta\right)\left[\mu'_*h' + \mu'h'_* - \mu h_*\right]b\left(\mbox{cos}\:\theta\right)\abs{v-v_*}^\gamma\:d\sigma dv_*.$$

The following lemmas give control over the operators $A^{(\delta)}$ and $B^{(\delta)}$ for $m=e^{\kappa_1\abs{v}^{\kappa_2}}$ with $\kappa_1 >0$ and $\kappa_2$ in $(0,2)$ or $m=\langle v \rangle^k$ with $k> k_\infty$. Their proofs can be found in \cite[Section 4]{GMM} in the specific case of hard sphere ($b=\gamma=1$) and for more general hard potential with cutoff kernels in \cite[Section 6.1.1]{BriDau} (polynomial weight) or \cite[Section 2]{Bri9} (exponential weight).

\bigskip
\begin{lemma}\label{lem:controlA}
Let $\zeta$ be in $(1/2,1]$. There exists $C_A>0$ such that for all $f$ in $L^\infty_{x,v}\pa{m}$
$$\norm{A^{(\delta)}(f)}_{L^\infty_{x,v}(\mu^{-\zeta})} \leq C_A\norm{f}_{L^\infty_{x,v}\pa{m}}.$$
The constant $C_A$ is constructive and only depends on $m$, $\zeta$, $\delta$ and the collision kernel.
\end{lemma}
\bigskip
\begin{lemma}\label{lem:controlB2}
$B^{(\delta)}_2$ satisfies
$$\forall f \in L^\infty_{x,v}\pa{m}, \quad \norm{B^{(\delta)}_2(f)}_{L^\infty_{x,v}\pa{\nu^{-1}m}} \leq C_B(\delta)\norm{f}_{L^\infty_{x,v}\pa{m}},$$
where $C_B(\delta)>0$ is a constructive constant such that
\begin{itemize}
\item if $m=\langle v \rangle^k$ then
$$\lim\limits_{\delta\to 0}C_B(\delta) = \frac{4}{k-1-\gamma}\frac{4\pi b_\infty}{l_b};$$
\item if $m=e^{\kappa_1\abs{v}^{\kappa_2}}$ then
$$\lim\limits_{\delta\to 0}C_B(\delta) = 0.$$
\end{itemize}
\end{lemma}
\bigskip

The operator $B_2^{(\delta)}$ also has a smallness property as an operator from $L^\infty_{x,v}(m)$ to $L^1_vL^\infty_v(\abs{v}^2)$. The following lemma is from \cite[Lemma 5.7]{Bri6} in the case of polynomial weights $m$ but is utterly applicable in the case of stretched exponential weights.

\bigskip
\begin{lemma}\label{lem:mixingcontrolB2}
For any $\delta >0$ there exists $\tilde{C}_B(\delta)$ such that for all $f$ in $L^\infty_{x,v}\pa{m}$,
$$\norm{B^{(\delta)}_2 (f)}_{L^1_vL^\infty_x\pa{\langle v \rangle^2}} \leq \tilde{C}_B(\delta) \norm{f}_{L^\infty_{x,v}\pa{m}}.$$
Moreover, the following holds: $\lim_{\delta \to 0} \tilde{C}_B(\delta) = 0.$
\end{lemma}
\bigskip

\begin{remark}\label{rem:kq*}
We emphasize here that for our choices of weights (see definition of $k_\infty$ $\eqref{kinfty}$), $\lim_{\delta\to 0}C_B(\delta)=C_B(m) <1$. Until the end of the present section we only consider $0<\delta$ small enough such that $C_B(\delta)<1$.
\end{remark}
\bigskip

We conclude this subsection with a control on the bilinear term in the $L^\infty_{x,v}$ setting. 
\bigskip
\begin{lemma}\label{lem:controlQ}
For all $h$ and $g$ such that $Q(h,g)$ is well-defined, $Q(h,g)$ belongs to $\cro{\mbox{Ker}(L)}^\bot$ in $L^2_v$:
\begin{equation*}
\pi_L\pa{Q(h,g)}=0.
\end{equation*}
Moreover, there exists $C_Q >0$ such that for all $h$ and $g$,
\begin{equation*}
\norm{Q(h,g)}_{L^\infty_{x,v}\pa{ \nu^{-1}m}} \leq C_Q \norm{h}_{L^\infty_{x,v}\pa{m}}\norm{g}_{L^\infty_{x,v}\pa{m}}.
\end{equation*}
The constant $C_Q$ is explicit and depends only on $m$ and the kernel of the collision operator.
\end{lemma}

\bigskip
\begin{proof}[Proof of Lemma \ref{lem:controlQ}]
Since we use the symmetric definition of $Q$ $\eqref{Qfg}$ the orthogonality property can be found in \cite[Appendix A.2.1]{Bri1}. The estimate follows directly from \cite[Lemma 5.16]{GMM} and the fact that $\nu(v) \sim \langle v \rangle^\gamma$ (see $\eqref{nu0nu1}$).
\end{proof}
\bigskip


\subsubsection{Study of equation $\eqref{f1}$ in $L^\infty_{x,v}\pa{m}$}\label{subsubsec:f1}

In the section we study the differential equation $\eqref{f1}$. We prove well-posedness for this problem and above all exponential decay as long as the initial datum is small. We deal with the different types of weights $m$ in the same way. Since the ``operator norm'' of $B_2^{(\delta)}$ tends to zero as $\delta$ tends to zero in the case of stretched exponential weight, one has a more direct proof in this case and we refer to \cite[Section 5.2.2]{Bri6}, where the author dealt with pure diffusion, for the interested reader.

\begin{prop}\label{prop:f1}
Let $m=e^{\kappa_1\abs{v}^{\kappa_2}}$ with $\kappa_1 >0$ and $\kappa_2$ in $(0,2)$ or $m=\langle v \rangle^k$ with $k> k_\infty$. Let $f_0$ be in $L^\infty_{x,v}\pa{m}$ and $g(t,x,v)$ in $L^\infty_{x,v}\pa{m}$. Then there exists $\delta_m>0$ such that for any $\delta$ in $(0,\delta_m]$ there exist $C_1$, $\eta_1$ and $\lambda_1>0$ such that if
$$ \norm{f_0}_{L^\infty_{x,v}\pa{m}}\leq \eta_1 \quad\mbox{and}\quad \norm{g}_{L^\infty_t L^\infty_{x,v}\pa{m}}\leq \eta_1,$$
then there exists a solution $f_1$ to 
\begin{equation}\label{f1prop}
\partial_tf_1 = G_\nu f_1 + B^{(\delta)}_2f_1 + Q(f_1+g,f_1+g),
\end{equation}
with initial datum $f_0$ and satisfying the Maxwell boundary condition $\eqref{mixedBC}$.
Moreover, this solution satisfies
$$\forall t\geq 0,\quad \norm{f_1(t)}_{L^\infty_{x,v}\pa{m}}\leq C_1 e^{-\lambda_1t}\norm{f_0}_{L^\infty_{x,v}\pa{m}}.$$
The constants $C_1$, $\eta_1$ and $\lambda_1$ are constructive and only depend on $m$, $\delta$ and the kernel of the collision operator.
\end{prop}
\bigskip

\begin{proof}[Proof of Proposition \ref{prop:f1}]
Thanks to Proposition \ref{theo:semigroupGnu}, $G_\nu$ combined with the Maxwell boundary condition generates a semigroup $S_{G_\nu}(t)$ in $L^\infty_{x,v}\pa{m}$. Therefore if $f_1$ is solution to $\eqref{f1prop}$ then it has the following Duhamel representation almost everywhere in $\R^+\times\Omega\times\R^3$
\begin{equation*}
\begin{split}
f_1(t,x,v) =& S_{G_\nu}(t)f_0(x,v) + \int_0^{t} S_{G_\nu}(t-s)\cro{B^{(\delta)}_2 (f_1(s))}(x,v)\:ds
\\&+\int_0^{t} S_{G_\nu}(t-s)\cro{Q(f_1(s)+g(s),f_1(s)+g(s))}(x,v))\:ds.
\end{split}
\end{equation*}
\par 
To prove existence and exponential decay we use the following iteration scheme starting from $h_0 = 0$.
\begin{equation}\label{iterativeschemef1}
\left\{\begin{array}{l} \disp{h_{l+1} = S_{G_\nu}(t)f_0 + \int_0^{t} S_{G_\nu}(t-s)\cro{B^{(\delta)}_2 (h_{l+1}) + Q(h_l+g,h_l+g)}\:ds } \vspace{2mm}\\ \disp{h_{l+1}(0,x,v) = f_0(x,v).}\end{array}\right.
\end{equation}
A contraction argument on the Duhamel representation above would imply that $\pa{h_l}_{l\in\N}$ is well-defined in $L^\infty_{x,v}\pa{m}$ and satisfies the Maxwell boundary condition (because $S_{G_\nu}$ does). The computations to prove this contraction property are similar to the ones we are about to develop in order to prove that $\pa{h_l}_{l\in\N}$ is a Cauchy sequence and we therefore only write down the latter.

\bigskip
Considering the difference $h_{l+1}-h_l$ we write, since $Q$ is a symmetric bilinear operator,
\begin{equation}\label{differenceiteration}
\begin{split}
h_{l+1}(t,x,v)-h_l(t,x,v) =& \int_0^{t} S_{G_\nu}(t-s)\cro{B^{(\delta)}_2 (h_{l+1}(s)-h_l(s))}ds 
\\&+ \int_0^{t} S_{G_\nu}(t-s)\cro{Q(h_{l}-h_{l-1},h_{l}+h_{l-1}+g)(s,x,v)}ds
\end{split}
\end{equation}
As now usual, we write $S_{G_\nu}(t-s)$ under its implicit form after one rebound  against the boundary (see $\eqref{fncharacteristics}$ or Step 1 of the proof of Proposition \ref{prop:triplenorm}). It reads
\begin{equation*}
\begin{split}
&\int_0^{t} S_{G_\nu}(t-s)\cro{B^{(\delta)}_2 (h_{l+1}(s)-h_l(s))}(x,v)\:ds
\\& = \int_{\max\br{0,t-t_1}}^t e^{-\nu(v)(t-s)}B_2^{(\delta)}\pa{h_{l+1}(s)-h_l(s)}(x-(t-s)v,v)\:ds
\\&+ (1-\alpha) e^{-\nu(v)t_1}\mathbf{1}_{\br{t> t_1}} \int_0^{t-t_1} S_{G_\nu}((t-t_1)-s)B_2^{(\delta)}\pa{h_{l+1}(s)-h_l(s)}(x_1,v_1)\:ds
\\&+ \alpha\mu e^{-\nu(v)t_1}\mathbf{1}_{\br{t> t_1}}  \int_0^{t-t_1} \int_{\mathcal{V}_1} \frac{1}{\mu(v_{1*})} S_{G_\nu}(t-t_1-s)B_2^{(\delta)}\pa{h_{l+1}-h_l}(x_1,v_{1*})d\sigma_{x_1}ds,
\end{split}
\end{equation*}
where $t_1 = t_{min}(x,v)$, $x_1 = x -t_1v$ and $v_1 = \mathcal{R}_{x_1}(v)$. Using the decomposition $S_{G_\nu}(f) = I_p(f) + R_p(S_{G_\nu}(f) )$ given by Lemma \ref{lem:IpRp}, we obtain for all $p\geq 1$:
\begin{equation}\label{differenceiterationineq}
m(v)\abs{h_{l+1}-h_l}(t,x,v) = J_B + J_{t_1} + J_{IB} + J_{RB} + J_Q
\end{equation}
with the following definitions
\begin{equation}\label{JB}
J_B = \int_{\max\br{0,t-t_1}}^t m(v)e^{-\nu(v)(t-s)}\abs{B_2^{(\delta)}\pa{h_{l+1}(s)-h_l(s)}}(x-(t-s)v,v)\:ds, 
\end{equation}
\begin{equation}\label{Jf0}
J_{t_1} = (1-\alpha) e^{-\nu(v)t_1}\mathbf{1}_{\br{t> t_1}}m(v)\abs{h_{l+1}-h_l}(t-t_1,x_1,v_1),
\end{equation}
\begin{equation}\label{JIB}
\begin{split}
J_{IB} =& \alpha\mu e^{-\nu(v)t_1}\mathbf{1}_{\br{t> t_1}}  \int_0^{t-t_1} \int_{\mathcal{V}_1} \frac{m(v)}{\mu(v_{1*})}
\\&\quad\quad\quad\quad\quad\times \abs{I_p\pa{B_2^{(\delta)}\pa{h_{l+1}-h_l}}(t-t_1-s,x_1,v_{1*})}d\sigma_{x_1}(v_{1*})ds,
\end{split}
\end{equation}
\begin{equation}\label{JRB}
\begin{split}
J_{RB} =& \alpha\mu e^{-\nu(v)t_1}\mathbf{1}_{\br{t> t_1}}  \int_0^{t-t_1} \int_{\mathcal{V}_1} \frac{m(v)}{\mu(v_{1*})}
\\&\quad\quad\quad\quad\quad\times \abs{R_p\pa{S_{G_\nu}B_2^{(\delta)}\pa{h_{l+1}-h_l}}(t-t_1-s,x_1,v_{1*})}d\sigma_{x_1}(v_{1*})ds,
\end{split}
\end{equation}
\begin{equation}\label{JQ}
\begin{split}
J_Q =& \int_0^{t} m(v)\abs{S_{G_\nu}(t-s)\cro{Q(h_{l}-h_{l-1},h_{l}+h_{l-1}+g)(s,x,v)}}ds+\alpha\mathbf{1}_{\br{t> t_1}}
\\& \times\mu e^{-\nu(v)t_1} \int_0^{t-t_1} \abs{S_{G_\nu}(t-t_1-s)\cro{Q(h_{l}-h_{l-1},h_{l}+h_{l-1}+g)(s,x_1,v_1)}}ds.
\end{split}
\end{equation}
We now estimate each of these terms separately.

\bigskip
\textbf{Estimate for $\mathbf{J_B}$ and $\mathbf{J_{t_1}}$}
These two terms are connected \textit{via} $t_1$ and it is important to understand that the contributions of $J_B$ and $J_{t_1}$ are interchanging.
\par By crudely bounding the integrand of $\eqref{JB}$ by the $L^\infty_{x,v}$-norm and controlling $B_2^{(\delta)}$ by Lemma \ref{lem:controlB2},
\begin{eqnarray*}
m(v)\abs{B_2^{(\delta)}\pa{h_{l+1}-h_l}}(x-(t-s)v,v) &\leq& \nu(v)\norm{B_2^{(\delta)}\pa{h_{l+1}(s)-h_l(s)}}_{L^\infty_{x,v}(m\nu^{-1})}
\\&\leq& C_B(\delta)\nu(v)\norm{h_{l+1}(s)-h_l(s)}_{L^\infty_{x,v}(m)}.
\end{eqnarray*}
For $\eps$ in $(0,1)$ we have $e^{-\nu(v)(t-s)}\leq e^{-\eps\nu_0t}e^{-\nu(v)(1-\eps)(t-s)}e^{\eps\nu_0s}$ and thus
\begin{equation}\label{boundJB}
\begin{split}
J_B \leq& \frac{C_B(\delta)}{1-\eps}\pa{1-e^{-\nu(v)(1-\eps)\min\{t,t_1\}}}e^{-\eps\nu_0t}\sup\limits_{s\in[0,t]}\cro{e^{\eps\nu_0s}\norm{h_{l+1}(s)-h_l(s)}_{L^\infty_{x,v}(m)}}.
\end{split}
\end{equation}
\par For $J_{t_1}$ we notice in $\eqref{Jf0}$ that $\abs{v} = \abs{v_1}$ and therefore $m(v)=m(v_1)$. Also, for $t>t_1$ and all $\eps$ in $(0,1)$ we have $e^{-\nu(v)t_1} \leq e^{-(1-\eps)\nu(v)t_1}e^{-\eps\nu_0t}e^{\eps\nu_0(t-t_1)}$ and therefore
\begin{equation}\label{boundJf0}
J_{t_1} \leq (1-\alpha) e^{-(1-\eps)\nu(v)t_1} e^{-\eps\nu_0t}\mathbf{1}_{\br{t> t_1}}\cro{e^{\eps\nu_0(t-t_1)}\norm{(h_{l+1}-h_l)(t-t_1)}_{L^\infty_{x,v}(m)}}.
\end{equation}
\par From $\eqref{boundJB}$ and $\eqref{boundJf0}$ we deduce
\begin{equation}\label{boundJBJf0}
\forall \eps\in (0,1),\quad J_B+J_{t_1} \leq \min\br{(1-\alpha),\frac{C_B(\delta)}{1-\eps}}e^{-\eps\nu_0t}\sup\limits_{s\in[0,t]}\cro{e^{\eps\nu_0s}\norm{h_{l+1}-h_l}_{L^\infty_{x,v}(m)}}.
\end{equation}

\bigskip
\textbf{Estimate for $\mathbf{J_{IB}}$}
Using Lemma \ref{lem:IpRp} (replacing $\Delta_n$ by $\alpha$) to control $I_p$ in $\eqref{JIB}$ and bounding the exponential decay in $d\Sigma^k_l$ by $e^{-\nu_0(t-t_1-s-t_k^{(l)}}$ gives the following control
\begin{equation*}
\begin{split}
J_{IB} \leq \int_0^{t-t_1}e^{-\nu_0(t-s)}&\int_{\mathcal{V}_1}\sum\limits_{k=0}^p \sum\limits_{i=0}^k\sum\limits_{l \in \vartheta_k(i)}(1-\alpha)^i\alpha^{k-i}\int_{\prod_{j=1}^p\mathcal{V}_j}\frac{\mu(v)m(v)}{\mu(v_k^{(l)})}
\\&\times\abs{B_2^{(\delta)}\pa{h_{l+1}-h_l}(s,x_k^{(l)}-t_k^{(l)}v_k^{(l)},v_k^{(l)})}\:\prod_{0\leq j\leq k}d\sigma_{x_j^{(l)}}(v_{j*})ds,
\end{split}
\end{equation*}
where the sequence $(t_k^{(l)},x_k^{(l)},v_k^{(l)})$ is associated to the initial point $(t-t_1-s,x_1,v_{1*})$.
\par For a given $(k,i,l)$ there exists a unique $J$ in $\br{0,p}$ such that $v_k^{(l)}=V_1(\dots (V_1(x_j,v_{J*}))))$ $k-J$ iterations and thus $(t_k^{(l)},x_k^{(l)},v_k^{(l)})$ only depends on $(t-t_1-s,x,v,v_{1*},\dots,v_{J*})$ and we can integrate the remaining variables. We remind here that $d\sigma_{x_j}$ is a probability measure on $\mathcal{V}_j$ and also
$$d\sigma_{x_J}(v_{J*}) =c_\mu \mu(v_{J*}) v_{J*}\cdot n(x_J) dv_{J*}.$$
Since $\abs{v_k^{(l)}}=\abs{v_{J*}}$ and $\mu(v)m(v) \leq C$, we infer the following,
\begin{equation*}
\begin{split}
J_{IB} \leq &C\int_0^{t-t_1}e^{-\nu_0(t-s)}\int_{\mathcal{V}_1}\sum\limits_{k=0}^p \sum\limits_{i=0}^k\sum\limits_{l \in \vartheta_k(i)}(1-\alpha)^i\alpha^{k-i}\int_{\prod_{j=1}^{J-1}\mathcal{V}_j}\:ds\prod_{0\leq j\leq J-1}d\sigma_{x_j^{(l)}}(v_{j*})
\\&\times\pa{\int_{\R^3}\abs{B_2^{(\delta)}\pa{h_{l+1}-h_l}(s,x_k^{(l)}-t_k^{(l)}v_k^{(l)},v_k^{(l)})}\abs{v_{J*}}dv_{J*}}.
\end{split}
\end{equation*}
We now make the change of variable $v_k^{(l)} \mapsto v_{J*}$ which preserves the norm and then the integral in $v_{J*}$ can be bounded by the $L^1_vL^\infty_x(\langle v\rangle^2)$-norm of $B^{(\delta)}_2$. As in the proof of Lemma \ref{lem:controlIp} $\sum_{k,i,l}(1-\alpha)^i\alpha^{k-i} \leq p$ and this yields the following estimate
$$J_{IB} \leq p C \int_0^{t}e^{-\nu_0(t-s)}\norm{B_2^{(\delta)}\pa{h_{l+1}-h_l}(s)}_{L^1_vL^\infty_x\pa{\langle v\rangle^2}}\:ds.$$
\par To conclude the estimate on $J_{IB}$ we choose $p=p(T_0)$ defined in Lemma \ref{lem:controlRp} (which makes $p$ bounded by $C(1+T_0)$) and we control $B_2^{(\delta)}$ thanks to Lemma \ref{lem:mixingcontrolB2}. 
\begin{equation}\label{boundJIB}
J_{IB} \leq C(1+T_0)\tilde{C}_B(\delta)e^{-\eps\nu_0t}\sup\limits_{s\in[0,t]}\cro{e^{\eps\nu_0s}\norm{(h_{l+1}-h_l)(s)}_{L^\infty_{x,v}(m)}}
\end{equation}
for all $\eps$ in $(0,1)$, $T_0 \geq 0$ and $t$ in $[0,1]$.

\bigskip
\textbf{Estimate for $\mathbf{J_{RB}}$}
The term $\eqref{JRB}$ is dealt with by crudely bounding the integrand of $\eqref{JRB}$ by its $L^\infty_{x,v}$-norm. Using Lemma \ref{lem:controlRp} to estimate $R_p$ we get for all $t$ in $[0,T_0]$,
\begin{equation*}
\begin{split}
&\frac{\mu(v)}{\mu(v_{1*})}e^{-\nu(v)t_1} m\abs{R_p\pa{S_{G_\nu}B_2^{(\delta)}(h_{l+1}-h_l)}(t-t_1-s,x_1,v_{1*})}
\\&\:\leq C \frac{\mu(v)m(v)}{\mu(v_{1*})m(v_{1*})}\nu(v_{1*})e^{-\nu_0(t-s)}\pa{\frac{1}{2}}^{\cro{\bar{C}T_0}}
\\&\quad\quad\times\sup\limits_{s_*\in[0,t-t_1-s]}\cro{e^{\nu_0 s_*}\norm{S_{G_\nu}(s_*)\cro{B_2^{(\delta)}(h_{l+1}-h_l)}}_{L^\infty_{x,v}(m\nu^{-1})}}.
\end{split}
\end{equation*}
Then, we apply Theorem \ref{theo:semigroupGnu} about the exponential decay of $S_{G_\nu}(s_*)$ in $L^\infty_{x,v}(m\nu^{-1})$  with a rate $(1-\eps)\nu_0$. At last, we control $B_2^{(\delta)}$ thanks to Lemma \ref{lem:controlB2}. This yields,
\begin{equation}\label{boundJRB}
J_{RB} \leq C \pa{\frac{1}{2}}^{\cro{\bar{C}T_0}}C_B(\delta)e^{-\eps\nu_0t}\sup\limits_{s\in[0,t]}\cro{e^{\eps\nu_0s}\norm{(h_{l+1}-h_l)(s)}_{L^\infty_{x,v}(m)}}
\end{equation}
for all $\eps$ in $(0,1)$, $T_0 \geq 0$ and $t$ in $[0,1]$.

\bigskip
\textbf{Estimate for $\mathbf{J_Q}$}
The term $\eqref{JQ}$ is dealt with using the gain of weight of $S_{G_\nu}(t)$ upon integration in time: Corollary \ref{cor:gainweightGnu}. A direct application of this corollary yields for all $\eps$ in $(0,1)$
$$J_Q \leq \frac{C_0}{1-\eps}e^{-\eps\nu_0t}\sup\limits_{s\in[0,t]}\cro{e^{\eps\nu_0s}\norm{Q(h_{l}-h_{l-1},h_{l}+h_{l-1}+g)(s)}_{L^\infty_{x,v}(m\nu^{-1})}}.$$
We control $Q$ thanks to Lemma \ref{lem:controlQ} and we infer
\begin{equation}\label{boundJQ}
\begin{split}
\forall \eps\in (0,1), \quad J_Q \leq& \frac{C}{1-\eps}\pa{\norm{h_{l}}_{L^\infty_{[0,t]}L^\infty_{x,v}(m)} + \norm{h_{l-1}}_{L^\infty_{[0,t]}L^\infty_{x,v}(m)}+\norm{g}_{L^\infty_tL^\infty_{x,v}(m)}}
\\&\times e^{-\eps\nu_0t}\sup\limits_{s\in[0,t]}\cro{e^{\eps\nu_0s}\norm{(h_{l}-h_{l-1})(s)}_{L^\infty_{x,v}(m)}}.
\end{split}
\end{equation}

\bigskip
\textbf{Conclusion of the proof.} We gather $\eqref{boundJBJf0}$, $\eqref{boundJIB}$, $\eqref{boundJRB}$ and $\eqref{JQ}$ inside $\eqref{differenceiterationineq}$. This gives for all $0<\eps<1$, all $T_0>0$ and all $t$ in $[0,T_0]$,
\begin{equation*}
\begin{split}
&\sup\limits_{s\in[0,t]}\cro{e^{\eps\nu_0s}\norm{(h_{l+1}-h_l)(s)}_{L^\infty_{x,v}(m)}}
\\&\:\leq \pa{\min\br{(1-\alpha),\frac{C_B(\delta)}{1-\eps}}+C(1+T_0)\tilde{C}_B(\delta) + 2^{-\cro{\bar{C}T_0}}C}
\\&\quad\quad\quad\quad\quad\times\sup\limits_{s\in[0,t]}\cro{e^{\eps\nu_0s}\norm{h_{l+1}-h_l}_{L^\infty_{x,v}(m)}}
\\&\:\quad+ \frac{C}{1-\eps}\pa{\norm{h_{l}}_{L^\infty_{[0,t]}L^\infty_{x,v}(m)} + \norm{h_{l-1}}_{L^\infty_{[0,t]}L^\infty_{x,v}(m)}+\norm{g}_{L^\infty_tL^\infty_{x,v}(m)}}
\\&\:\quad\quad\quad\quad\quad\times \sup\limits_{s\in[0,t]}\cro{e^{\eps\nu_0s}\norm{h_{l}-h_{l-1}}_{L^\infty_{x,v}(m)}}
\end{split}
\end{equation*}
\par We choose our constants as follow.
\begin{itemize}
\item From Lemma \ref{lem:controlB2} we define $\delta_0>0$ such that for all $\delta<\delta_0$, $C_B(\delta) \leq C_B(\delta_0)<1$;
\item Since $C_B(\delta_0)<1$ we fix $\eps$ in $(0,1)$ such that $C_B(\delta_0)+\eps <1$;
\item We choose $T_0$ large enough such that 
$$2^{-\cro{\bar{C}T_0}}C \leq \frac{1}{4}\pa{1-\min\br{(1-\alpha),\frac{C_B(\delta_0)}{1-\eps}}};$$
\item At last, we take $\delta < \delta_0$ such that, from Lemma \ref{lem:mixingcontrolB2},
$$C(1+T_0)\tilde{C}_B(\delta) \leq \frac{1}{4}\pa{1-\min\br{(1-\alpha),\frac{C_B(\delta_0)}{1-\eps}}}.$$
\end{itemize}
Denoting 
$$C_0 = 2C\pa{1-\min\br{(1-\alpha),\frac{C_B(\delta_0)}{1-\eps}}}^{-1},$$
our choice of constants implies that for all $t$ in $[0,T_0]$,
\begin{equation}\label{fincauchy}
\begin{split}
&\sup\limits_{s\in[0,t]}\cro{e^{\eps\nu_0s}\norm{h_{l+1}-h_l}_{L^\infty_{x,v}(m)}}
\\&\quad\quad\quad\leq C_0\pa{\norm{h_{l}}_{L^\infty_{[0,t]}L^\infty_{x,v}(m)} + \norm{h_{l-1}}_{L^\infty_{[0,t]}L^\infty_{x,v}(m)}+\norm{g}_{L^\infty_tL^\infty_{x,v}(m)}}
\\&\quad\quad\quad\quad\quad\times\sup\limits_{s\in[0,t]}\cro{e^{\eps\nu_0s}\norm{h_{l}-h_{l-1}}_{L^\infty_{x,v}(m)}}.
\end{split}
\end{equation}
Of important note is the fact that $C_0$ and $\eps$ do not depend on $T_0$ and therefore we can iterate the process on $[T_0,2T_0]$ and so on. The inequality above thus holds for all $t \geq 0$.
\par To conclude, we first prove that $\norm{h_{l+1}}_{L^\infty_{[0,t]}L^\infty_{x,v}(m)}$ is uniformly bounded. We could do exactly the same computations but subtracting $S_{G_\nu}(t)f_0(x,v)$ instead of $h_l(t,x,v)$ to $h_l(t,x,v)$ in $\eqref{differenceiteration}$. Thus, $\eqref{fincauchy}$ would become for all $t\geq 0$,
\begin{equation*}
\begin{split}
\sup\limits_{s\in[0,t]}\cro{e^{\eps\nu_0s}\norm{(h_{l+1}-S_{G_\nu}f_0)(s)}_{L^\infty_{x,v}(m)}}\leq& C_0\pa{\norm{h_{l}}_{L^\infty_{[0,t]}L^\infty_{x,v}(m)}+\norm{g}_{L^\infty_tL^\infty_{x,v}(m)}}
\\&\times\sup\limits_{s\in[0,t]}\cro{e^{\eps\nu_0s}\norm{h_{l}}_{L^\infty_{x,v}(m)}}.
\end{split}
\end{equation*}
Using the exponential decay of $S_{G_\nu}(t)f_0$, which is faster than $\eps\nu_0$ (see Theorem \ref{theo:semigroupGnu}), and assuming that the norms of $g$ and $f_0$ are bounded by $\eta_1$ to be determined later,
\begin{equation*}
\begin{split}
&\sup\limits_{s\in[0,t]}\cro{e^{\eps\nu_0s}\norm{h_{l+1}(s)}_{L^\infty_{x,v}(m)}} 
\\&\quad\leq C_0^{(1)}\norm{f_0}_{L^\infty_{x,v}(m)} + C^{(2)}_0\pa{\norm{h_{l}}_{L^\infty_{[0,t]}L^\infty_{x,v}(m)} + \eta_1}\sup\limits_{s\in[0,t]}\cro{e^{\eps\nu_0s}\norm{h_{l}(s)}_{L^\infty_{x,v}(m)}},
\end{split}
\end{equation*}
where $C_0^{(1)}$ and $C_0^{(2)}$ are two positive constants independent of $h_{l+1}$ and $\eta_1$.
\par If we choose $\eta_1>0$ small enough such that
$$C_0^{(1)}\eta_1 + C_0^{(2)}\pa{2+C_0^{(1)}}\pa{1+C_0^{(1)}}\eta_1^2 \leq (1+C_0^{(1)}) \eta_1$$
then we obtain by induction that
\begin{equation}\label{f1expodecay}
\forall l \in\N,\:\forall t\geq 0,\quad\sup\limits_{s\in[0,t]}\cro{e^{\eps\nu_0s}\norm{h_{l}(s)}_{L^\infty_{x,v}(m)}} \leq \pa{1+C_0^{(1)}}\norm{f_0}_{L^\infty_{x,v}(m)}.
\end{equation}
\par We now plug $\eqref{f1expodecay}$ into $\eqref{fincauchy}$ and use the fact that $f_0$ and $g$ are bounded by $\eta_1$. This gives
$$\norm{h_{l+1}-h_l}_{L^\infty_{t,x,v}(m)} \leq 3C_0\pa{1+C_0^{(1)}}\eta_1\norm{h_{l}-h_{l-1}}_{L^\infty_{t,x,v}(m)}.$$
The latter implies that for $\eta_1$ small enough, $\pa{h_l}_{l\in\N}$ is a Cauchy sequence in $L^\infty_{t,x,v}(m)$ and therefore converges towards $f_1$ in $L^\infty_{t,x,v}(m)$. Since $\gamma < k$, we can take the limit inside the iterative scheme $\eqref{iterativeschemef1}$ and $f_1$ is a solution to $\eqref{f1prop}$. Moreover, by taking the limit inside $\eqref{f1expodecay}$, $f_1$ has the desired exponential decay. This concludes the proof of Proposition \ref{prop:f1}.
\end{proof}
\bigskip


\subsubsection{Study of equation $\eqref{f2}$ in $L^\infty_{x,v}(\mu^{-\zeta})$}\label{subsubsec:f2}

We turn to the differential equation $\eqref{f2}$ in $L^\infty_{x,v}(\mu^{-\zeta})$ with $\zeta$ in $(1/2,1)$ so that Theorem \ref{theo:semigroupLinfty} holds.

\bigskip
\begin{prop}\label{prop:f2}
Let $m=e^{\kappa_1\abs{v}^{\kappa_2}}$ with $\kappa_1 >0$ and $\kappa_2$ in $(0,2)$ or $m=\langle v \rangle^k$ with $k> k_\infty$. Let $g=g(t,x,v)$ be in $L^\infty_tL^\infty_{x,v}\pa{m}$. Then there exists a unique function $f_2$ in $L^\infty_tL^\infty_{x,v}(\mu^{-\zeta})$ such that
$$\partial_t f_2 = G\pa{f_2} + A^{(\delta)}\pa{g} \quad\mbox{and}\quad f_2(0,x,v)= 0.$$
Moreover, if $\:\Pi_G\pa{f_2+g}=0$ and if
$$\exists\: \lambda_g,\:\eta_g>0,\:\forall t\geq 0, \: \norm{g(t)}_{L^\infty_{x,v}\pa{m}}\leq \eta_ge^{-\lambda_g t},$$
then  for any $0<\lambda_2<\min\br{\lambda_g,\:\lambda_\infty}$, with $\lambda_\infty$ defined in Theorem \ref{theo:semigroupLinfty}, there exist $C_2>0$ such that
$$\forall t\geq 0, \quad \norm{f_2(t)}_{L^\infty_{x,v}(\mu^{-\zeta})} \leq C_2\eta_g e^{-\lambda_2 t}.$$
The constant $C_2$ only depends on $\lambda_2$.
\end{prop}
\bigskip

\begin{proof}[Proof of Proposition \ref{prop:f2}]
Thanks to the regularising property of $A$ (Lemma \ref{lem:controlA}) $A^{(\delta)}\pa{g}$ belongs to $L^\infty_tL^\infty_{x,v}(\mu^{-\zeta})$. Theorem \ref{theo:semigroupLinfty} implies that there is indeed a unique $f_2$ solution to the differential equation and it is given by
$$f_2 = \int_0^{t}S_G(t-s)\cro{A^{(\delta)}\pa{g}(s)}\:ds,$$
where $S_G(t)$ is the semigroup generated by $G=L-v\cdot\nabla_x$ in $L^\infty_{x,v}(\mu^{-\zeta})$.
\par Suppose now that $\Pi_G\pa{f_2+g}=0$ and that there exists $ \eta_2>0$ such that $\norm{g(t)}_{L^\infty_{x,v}\pa{m}}\leq \eta_g e^{-\lambda t}$. Using the definition of $\Pi_G$ $\eqref{PiG}$, the projection part of $f_2$ is straightforwardly bounded for all $t\geq 0$:
\begin{equation}\label{PiGf2}
\begin{split}
\norm{\Pi_G\pa{f_2}(t)}_{L^\infty_{x,v}(\mu^{-\zeta})} &= \norm{\Pi_G\pa{g}(t)}_{L^\infty_{x,v}(\mu^{-\zeta})}\leq C_{\Pi_G} \norm{g}_{L^\infty_{x,v}\pa{m}}
\\&\leq C_{\Pi_G}\eta_g\:e^{-\lambda_g t}.
\end{split}
\end{equation}

\par Applying $\Pi_G^\bot = \mbox{Id}-\Pi_G$ to the equation satisfied by $f_2$ we get, thanks to $\eqref{PiG}$,
$$\partial_t \cro{\Pi_G^\bot\pa{f_2}} = g\cro{\Pi_G^\bot\pa{f_2}} + \Pi_G^\bot\pa{A^{(\delta)}\pa{g}}.$$
This yields
$$\Pi_G^\bot\pa{f_2} = \int_0^{t}S_G(t-s)\cro{ \Pi_G^\bot\pa{A^{(\delta)}\pa{g}}(s)}\:ds.$$
We use the exponential decay of $S_{G}(t)$ on $\pa{\mbox{Ker}(g)}^\bot$ (see Theorem \ref{theo:semigroupLinfty}).
$$\norm{\Pi_G^\bot\pa{f_2}}_{L^\infty_{x,v}(\mu^{-\zeta})} \leq C_\infty \int_0^{t} e^{-\lambda_\infty (t-s)}\norm{A^{(\delta)}\pa{g}(s)}_{L^\infty_{x,v}(\mu^{-\zeta})}\:ds.$$
Using the definition of $\Pi_G$ $\eqref{PiG}$ and then the regularising property of $A$ Lemma \ref{lem:controlA} we further can further bound. Fix $\lambda_2 < \min\br{\lambda_\infty,\:\lambda_g}$,
\begin{eqnarray}
\norm{\Pi_G^\bot\pa{f_2}}_{L^\infty_{x,v}(\mu^{-\zeta})} &\leq& C_GC_\infty C_{\Pi_G}C_AC_g\eta_g \int_0^{t} e^{-\lambda_\infty(t-s)}e^{-\lambda_g s}\:ds \nonumber
\\&\leq& C_GC_\infty C_{\Pi_G}C_AC_g\eta_g \:  te^{-\min\br{\lambda_g,\lambda_\infty} t}\nonumber
\\&\leq& C_2(\lambda_2)\eta_g e^{-\lambda_2 t}.\label{PiGbotf2}
\end{eqnarray}
\par Gathering $\eqref{PiGf2}$ and $\eqref{PiGbotf2}$ yields the desired exponential decay.
\end{proof}
\bigskip


\subsubsection{Proof of Theorem \ref{theo:existenceexpodecay}}\label{subsubsec:existence}

Take $f_0$ in $L^\infty_{x,v}\pa{m}$ such that $\Pi_G(f_0)=0$.
\par The existence will be proved by an iterative scheme. We start with $f^{(0)}_1=f^{(0)}_2=0$ and we approximate the system of equations $\eqref{f1}-\eqref{f2}$ as follows.
\begin{eqnarray*}
\partial_t f^{(n+1)}_1&=& B^{(\delta)} \pa{f^{(n+1)}_1} + Q\pa{f^{(n+1)}_1+f^{(n)}_2}
\\\partial_t f^{(n+1)}_2 &=& G\pa{f^{(n+1)}_2} + A^{(\delta)}\pa{f^{(n+1)}_1},
\end{eqnarray*}
with the following initial data
$$f^{(n+1)}_1(0,x,v)=f_0(x,v) \quad\mbox{and}\quad f^{(n+1)}_2(0,x,v)=0.$$
\par Assume that $(1+C_2)\norm{f_0}\leq \eta_1$, where $C_2$ was defined in Proposition \ref{prop:f2} and $\eta_1$ was defined in Proposition \ref{prop:f1}. Thanks to Proposition \ref{prop:f1} and Proposition \ref{prop:f2}, an induction proves first that $\pa{f^{(n)}_1}_{n\in\N}$ and $\pa{f^{(n)}_2}_{n\in\N}$ are well-defined sequences and second that for all $n$ in $\N$ and all $t\geq 0$
\begin{eqnarray}
\norm{f^{(n)}_1(t)}_{L^\infty_{x,v}\pa{m}} &\leq& e^{-\lambda_1 t}\norm{f_0}_{L^\infty_{x,v}\pa{m}} \label{expodecayfn1}
\\\norm{f^{(n)}_2(t)}_{L^\infty_{x,v}(\mu^{-\zeta})} &\leq& C_2 e^{-\lambda_2 t}\norm{f_0}_{L^\infty_{x,v}\pa{m}},\label{expodecayfn2}
\end{eqnarray}
with $\lambda_2 <\min\br{\lambda_1,\lambda_\infty}$. Indeed, if we constructed $f^{(n)}_1$ and $f^{(n)}_2$ satisfying the exponential decay above then we can construct $f^{(n+1)}_1$, which has the required exponential decay $\eqref{expodecayfn1}$, and then construct $f^{(n+1)}_2$. Finally, we have the following equality
$$\partial_t\pa{f^{(n+1)}_1+f^{(n+1)}_2} = g\pa{f^{(n+1)}_1+f^{(n+1)}_2} + Q\pa{f^{(n+1)}_1+f^{(n)}_2}.$$
Thanks to orthogonality property of $Q$ in Lemma \ref{lem:controlQ} and the definition of $\Pi_G$ $\eqref{PiG}$ we obtain that the projection is constant with time and thus
$$\Pi_G\pa{f^{(n+1)}_1+f^{(n+1)}_2} = \Pi_G(f_0)=0.$$
Applying Proposition \ref{prop:f2} we obtain the exponential decay $\eqref{expodecayfn2}$ for $f^{(n+1)}_2$.

\bigskip
We recognize exactly the same iterative scheme for $f^{(n+1)}_1$ as in the proof of Proposition \ref{prop:f1} with $g$ replaced by $f_2^{(n)}$. Moreover, the uniform bound $\eqref{expodecayfn2}$ allows us to derive the same estimates as in the latter proof independently of $f^{(n)}_2$. As a conclusion, $\pa{f^{(n)}_1}_{n\in\N}$ is a Cauchy sequence in $L^\infty_tL^\infty_{x,v}\pa{m}$ and therefore converges strongly towards a function $f_1$.
\par By $\eqref{expodecayfn2}$, the sequence $\pa{f^{(n)}_2}_{n\in\N}$ is bounded in $L^\infty_tL^\infty_{x,v}(\mu^{-\zeta})$ and is therefore weakly-* compact and therefore converges, up to a subsequence, weakly-* towards $f_2$ in $L^\infty_tL^\infty_{x,v}(\mu^{-\zeta})$.
\par Since the kernel inside the collision operator behaves like $\abs{v-v_*}^\gamma$ and that our weight $m(v)$ is either exponetial or of degree $k>2>\gamma$, we can take the weak limit inside the iterative scheme. This implies that $(f_1,f_2)$ is solution to the system $\eqref{f1}-\eqref{f2}$ and thus $f = f_1+f_2$ is solution to the perturbed equation $\eqref{perturbedBEfull}$. Moreover, taking the limit inside the exponential decays $\eqref{expodecayfn1}$ and $\eqref{expodecayfn2}$ yields the expected exponential decay for $f$.
\bigskip


\subsection{Uniqueness in the perturbative framework}\label{subsec:uniqueness}

We conclude the proof of our main Theorem \ref{theo:perturbativeCauchy} stated in Section \ref{sec:mainresults} by proving the uniqueness of solutions in the perturbative regime.

\bigskip
\begin{theorem}\label{theo:uniqueness}
Let $m=e^{\kappa_1\abs{v}^{\kappa_2}}$ with $\kappa_1 >0$ and $\kappa_2$ in $(0,2)$ or $m=\langle v \rangle^k$ with $k> k_\infty$. There exits $\eta>0$ such that for any $f_0$ in $L^\infty_{x,v}(m)$ such that $\norm{f_0}_{L^\infty_{x,v}(m)}\leq \eta$ there exists at most one solution $f(t,x,v)$ in $L^\infty_tL^\infty_{x,v}(m)$ to the perturbed Boltzmann equation $\eqref{perturbedBEfull}$ with Maxwell boundary condition and with $f_0$ as initial datum.
\end{theorem}
\bigskip

\begin{proof}[Proof of Theorem \ref{theo:uniqueness}]
Let $f_0$ be in $L^\infty_{x,v}(m)$ such that $\norm{f_0}_{L^\infty_{x,v}(m)}\leq \eta$, $\eta$ to be chosen later. Suppose that there exist two solutions $f$ and $\tilde{f}$ in $L^\infty_tL^\infty_{x,v}(m)$ associated to the initial datum $f_0$.
\par Subtracting the equations satisfied by $f$ and $\tilde{f}$ we get
$$\partial_t \pa{f-\tilde{f}} = G\pa{f-\tilde{f}} + Q\pa{f-\tilde{f},f+\tilde{f}}$$
and following the decomposition of the previous subsection
$$\partial_t \pa{f-\tilde{f}} = \cro{G_\nu\pa{f-\tilde{f}} + B^{(\delta}_2\pa{f-\tilde{f}} + Q\pa{f-\tilde{f},f+\tilde{f}}} + A^{(\delta)}\pa{f-\tilde{f}}.$$
Since $G_\nu$ generates a semigroup in $L^\infty_{x,v}(m)$ we can write the equation above under its Duhamel form:
\begin{equation}\label{uniquenessDuhamel}
\begin{split}
f-\tilde{f} =& \int_0^{t} S_{G_\nu}(t-s)\cro{B^{(\delta}_2\pa{f-\tilde{f}} + Q\pa{f-\tilde{f},f+\tilde{f}}}ds 
\\&+ \int_0^{t} S_{G_\nu}(t-s)\cro{A^{(\delta)}\pa{f-\tilde{f}}}ds.
\end{split}
\end{equation}

\bigskip
The first term on the right-hand side can be treated the same way as in the proof of Proposition \ref{prop:f1} and therefore, for $\delta$ small enough, there exists $0<C_1<1$ such that
\begin{equation}\label{uniqueness1}
\begin{split}
&\abs{\int_0^{t} S_{G_\nu}(t-s)\cro{B^{(\delta}_2\pa{f-\tilde{f}} + Q\pa{f-\tilde{f},f+\tilde{f}}}ds} 
\\&\quad\leq \cro{(1-C_1) + C_Q\pa{\norm{f}_{L^\infty_{[0,t]}L^\infty_{x,v}(m)} + \norm{\tilde{f}}_{L^\infty_{[0,t]}L^\infty_{x,v}(m)}}}\norm{f-\tilde{f}}_{L^\infty_{[0,t]}L^\infty_{x,v}(m)}.
\end{split}
\end{equation}
Since $S_{G_\nu}(t)$ is bounded on $L^\infty_{t,x,v}(m)$ (see Theorem \ref{theo:semigroupGnu}), as well as $A^{(\delta)}$ is (see Lemma \ref{lem:controlA}, we can bound the second term on the right-hand side of $\eqref{uniquenessDuhamel}$ by
\begin{equation}\label{uniqueness2}
\abs{\int_0^{t} S_{G_\nu}(t-s)\cro{A^{(\delta)}\pa{f-\tilde{f}}}ds}\leq C_2 t\norm{f-\tilde{f}}_{L^\infty_{[0,t]}L^\infty_{x,v}(m)}.
\end{equation}
\par Plugging $\eqref{uniqueness1}$ and $\eqref{uniqueness2}$ into $\eqref{uniquenessDuhamel}$ we obtain
\begin{equation}\label{uniquenessdifference}
\begin{split}
\norm{f-\tilde{f}}_{L^\infty_{[0,t]}L^\infty_{x,v}(m)} \leq &\cro{(1-C_1) + C_Q\pa{\norm{f}_{L^\infty_{[0,t]}L^\infty_{x,v}(m)} + \norm{\tilde{f}}_{L^\infty_{[0,t]}L^\infty_{x,v}(m)}}+ C_2t}
\\&\times\norm{f-\tilde{f}}_{L^\infty_{[0,t]}L^\infty_{x,v}(m)}.
\end{split}
\end{equation}

\bigskip
We now need to prove that the $L^\infty_{x,v}(m)$-norm of $f$ and $\tilde{f}$ must be bounded by the one of $f_0$. But this follows from $\eqref{uniquenessdifference}$ when one subtracts $S_{G_\nu}(t)f_0$ instead of $\tilde{f}$ to $f$. This yields, after controlling $S_{G_\nu}(t)f_0$ by its $L^\infty_{x,v}(m)$-norm,
$$\norm{f}_{L^\infty_{[0,t]}L^\infty_{x,v}(m)} \leq C_0 \norm{f_0}_{L^\infty_{x,v}(m)}+ \cro{(1-C_1) + C_Q\norm{f}_{L^\infty_{[0,t]}L^\infty_{x,v}(m)}+ C_2t}\norm{f}_{L^\infty_{[0,t]}L^\infty_{x,v}(m)}.$$
Since $C_1<1$ we fix $T_0$ such that $C_2T_0 < C_1/4$. We deduce that for all $t$ in $[0,T_0]$
$$\forall t \in [0,T_0], \quad\frac{3C_1}{4}\norm{f}_{L^\infty_{[0,t]}L^\infty_{x,v}(m)} \leq C_0 \norm{f_0}_{L^\infty_{x,v}(m)}+ C_Q\norm{f}^2_{L^\infty_{[0,t]}L^\infty_{x,v}(m)}$$
and therefore,
 if $\norm{f_0}_{L^\infty_{x,v}(m)}\leq \eta$ with $\eta$ small enough such that
$$\frac{3C_1}{4}-2\frac{C_QC_0}{C_1}\eta>\frac{C_1}{2}$$
then
\begin{equation}\label{uniquenessbound}
\forall t \in [0,T_0], \quad\norm{f}_{L^\infty_{[0,t]}L^\infty_{x,v}(m)} \leq \frac{2C_0}{C_1} \norm{f_0}_{L^\infty_{x,v}(m)}.
\end{equation}

\bigskip
To conclude the proof of uniqueness we see that $\eqref{uniquenessbound}$ is also valid for $\tilde{f}$ and $\eqref{uniquenessdifference}$ thus becomes
$$\forall t\in[0,T_0],\quad \norm{f-\tilde{f}}_{L^\infty_{[0,t]}L^\infty_{x,v}(m)} \leq \cro{\pa{1-\frac{3C_1}{4}} + 4\frac{C_0C_Q}{C_1}\eta}\norm{f-\tilde{f}}_{L^\infty_{[0,t]}L^\infty_{x,v}(m)}.$$
We can choose $\eta$ even smaller such that the term on the right-hand side can be absorbed by the left-hand side. This implies that $f=\tilde{f}$ on $[0,T_0]$. Starting at $T_0$ we can iterate the process and finally getting that $f=\tilde{f}$ on $\R^+$; which concludes the proof of Theorem \ref{theo:uniqueness}.
\end{proof}
\bigskip

%% file: qualitativestudy.tex
\section{Qualitative study of the perturbative solutions to the Boltzmann equation}\label{sec:qalitativestudy}

In this last section, we address the issue of positivity and continuity of the solutions to the Boltzmann equation
\begin{equation}\label{BEpositivity}
\partial_t F + v \cdot\nabla_xF = Q\pa{F,F}.
\end{equation}
\par Note that even if our arguments are constructive, we only prove qualitative behaviours and we do not tackle the issue of quantitative estimates. For instance, we prove the positivity of the solutions but do not give any explicit lower bound. Such explicit lower bounds have been recently obtained in the case of pure specular reflections \cite{Bri2} and in the case of pure Maxwellian diffusion \cite{Bri5}. We think that the proofs can be adapted to fit the case of Maxwell boundary condition as it is a convex combination of these boundary conditions. However, the techniques required to deal with this are very different from the one developed throughout this paper and we therefore did not looked into it much further.

\bigskip


\subsection{Positivity of solutions}\label{subsec:positivity}

This subsection is dedicated to proving the following positivity property.

\bigskip
\begin{prop}\label{prop:positivity}
Let $m=e^{\kappa_1\abs{v}^{\kappa_2}}$ with $\kappa_1 >0$ and $\kappa_2$ in $(0,2)$ or $m=\langle v \rangle^k$ with $k> k_\infty$. Let $f_0$ be in $L^\infty_{x,v}\pa{m}$ with $\Pi_G(f_0)=0$ and
$$\norm{f_0}_{L^\infty_{x,v}\pa{m}}\leq \eta,$$
where $\eta>0$ is chosen such that Theorems \ref{theo:existenceexpodecay} and \ref{theo:uniqueness} hold and denote $f$ the unique solution of the perturbed equation associated to $f_0$.
\\ Suppose that $F_0= \mu + f_0\geq 0$ then $F=\mu + f \geq 0$.
\end{prop}
\bigskip

\begin{proof}[Proof of Proposition \ref{prop:positivity}]
Since we are working with hard potential kernels we can decompose the nonlinear operator into
$$Q(F,F) = -Q^-(F,F) + Q^+(F,F)$$
where
\begin{eqnarray*}
Q^-(F,F)(v) &=& \pa{\int_{\R^3\times \mathbb{S}^{2}}B\left(|v - v_*|,\mbox{cos}\:\theta\right) F_*\:dv_*d\sigma}F(v)=q(F)(v) F(v),
\\Q^+(F,F)&=&\int_{\R^3\times \mathbb{S}^{2}}B\left(|v - v_*|,\mbox{cos}\:\theta\right)F'F_*'\:dv_*d\sigma.
\end{eqnarray*}

\par Following the idea of \cite{Bri2}\cite{Bri5}, we obtain an equivalent definition of being a solution to $\eqref{BEpositivity}$ by applying the Duhamel formula along backward characteristics that is stopped right after the first collision against the boundary. If $F$ is solution to the Boltzmann equation then
for almost all $(x,v)$ in $\Omega\times\R^3$,
\begin{equation}\label{defpositivity1}
\begin{split}
F(t,x,v) =& F_0(x-vt,v)\mbox{exp}\left(-\int_{0}^t q[F(s,x-(t-s)v,\cdot)](v)\:ds\right) 
\\ &+ \int_{0}^t \mbox{exp}\left(-\int_s^t q[F(s',x-(s-s')v,\cdot)](v)\:ds'\right)
\\&\quad\quad\times Q^+\cro{F(s,x-(t-s)v,\cdot), F(s,x-(t-s)v,\cdot)}(v)\: ds
\end{split}
\end{equation}
if $t \leq t_{min}(x,v):=t_0$ or else
\begin{equation}\label{defpositivity2}
\begin{split}
F(t,x,v) =& F_\Lambda(t_0,x-t_0v,v)\mbox{exp}\pa{-\int_{t_0}^t q[F(s,x-(t-s)v,\cdot)](v)\:ds}
\\ &+ \int_{t_0}^t \mbox{exp}\pa{-\int_s^t q[F(s',x-(s-s')v,\cdot)](v)\:ds'}
\\&\quad\quad \times Q^+\cro{F(s,x-(t-s)v,\cdot), F(s,x-(t-s)v,\cdot)}(v)\: ds.
\end{split}
\end{equation}
We denoted by $F_\Lambda$ the Maxwell boundary condition for $(t',x',v)$ in $\R^+\times\Lambda$
$$F_\Lambda(t',x',v) = (1-\alpha) F(t',x',\mathcal{R}_{x'}(v))+ \alpha P_\Lambda\pa{F(t',x',\cdot)}(v).$$

\bigskip
We construct an iterative scheme $(F^{(n)})_{n\in\N}$ with $F^{(0)}=\mu$ and $F^{(n+1)}(t,x,v)$ being defined by $\eqref{defpositivity1}$ and $\eqref{defpositivity2}$ with all the $F$ on the right-hand side being replaced by $F^{(n)}$ except in the definition of $F_\Lambda$ where we keep $F^{(n+1)}$ instead. In other terms, $F^{(n+1)}$ is solution to
$$\cro{\partial_t + v\cdot\nabla_x + q(F^{(n)})}F^{(n+1)} = Q(F^{(n)},F^{(n)})$$
with the Maxwell boundary condition; which is an approximative scheme to the Boltzmann equation $\eqref{BEpositivity}$.
\par Defining $f^{(n)}= F^{(n)} -\mu$ we have the following differential iterative scheme
$$\partial_t f^{(n+1)} + v\cdot\nabla_xf^{(n+1)} = -\nu(v)\pa{f^{(n+1)}} + K\pa{f^{(n)}} + Q^+\pa{f^{(n)}}-q\pa{f^{(n)}}f^{(n+1)}.$$
As before, we prove that $\pa{f^{(n)}}_{n\in\N}$ is well-defined and converges in $L^\infty_{t,x,v}(m)$ towards $f$, the unique solution of the perturbed Boltzmann equation. Therefore, the same holds for $\pa{F^{(n)}}_{n\in\N}$ converging towards $F$ the unique perturbed solution of the original Boltzmann equation $\eqref{BEpositivity}$.
\par From the positivity of $q$ and $Q^+$ and $F_0$, a straightforward induction from $\eqref{defpositivity1}$ shows that $F^{(n)}(t,x,v)\geq 0$ for all $n$ when $t\leq t_0$. This implies that for all $n$ and all $(x,v)$, $F_\Lambda^{(n+1)}(t_0,x-t_0v,v)\geq 0$ and therefore $\eqref{defpositivity2}$ gives $F^{(n+1)}(t,x,v)\geq 0$ for all $(t,x,v)$ and all $n$. The positivity of $F$ follows by taking the limit as $n$ tends to infinity.
\end{proof}
\bigskip


\subsection{Continuity of solutions}\label{subsec:continuity}

The last issue tackled in the present article is the continuity of the solutions described in Section \ref{sec:fullcauchy}. More precisely, we prove the following proposition. 

\bigskip
\begin{prop}\label{prop:continuityqualitative}
Let $m=e^{\kappa_1\abs{v}^{\kappa_2}}$ with $\kappa_1 >0$ and $\kappa_2$ in $(0,2)$ or $m=\langle v \rangle^k$ with $k> k_\infty$. Let $f_0$ be in $L^\infty_{x,v}\pa{m}$ with $\Pi_G(f_0)=0$ and
$$\norm{f_0}_{L^\infty_{x,v}\pa{m}}\leq \eta,$$
where $\eta>0$ is chosen such that Theorems \ref{theo:existenceexpodecay} and \ref{theo:uniqueness} hold and denote $f$ the unique solution of the perturbed equation associated to $f_0$.
\\ Suppose that $F_0= \mu + f_0$ is continuous on $\Omega\times\R^3\cup\br{\Lambda^+\cup \mathfrak{C}^-_\Lambda}$ and satisfies the Maxwell boundary condition $\eqref{mixedBC}$ then $F=\mu + f$ is continuous on the continuity set $\mathfrak{C}$.
\end{prop}
\bigskip

We recall the definition of inward inflection grazing boundary
$$\Lambda_0^{(I-)} = \Lambda_0 \cap \br{t_{min}(x,v)=0,\: t_{min}(x,-v)\neq 0 \:\mbox{and}\: \exists \delta>0, \:\forall \tau \in [0,\delta],\: x-\tau v \in \bar{\Omega}^c}.$$
We also rewrite the boundary continuity set
$$\mathfrak{C}^-_\Lambda = \Lambda^- \cup \Lambda_0^{(I-)}$$
and the continuity set
\begin{equation*}
\begin{split}
\mathfrak{C} =& \Big\{\br{0}\times\br{\Omega\times\R^3\cup\pa{\Lambda^+\cup \mathfrak{C}^-_\Lambda}}\Big\} \cup \Big\{(0,+\infty)\times\mathfrak{C}^-_\Lambda \Big\}
\\&\cup \Big\{(t,x,v) \in (0,+\infty)\times\pa{\Omega\times\R^3\cup \Lambda^+}:
\\&\quad\quad \forall 1\leq k\leq N(t,x,v) \in\N, \pa{X_{k+1}(x,v),V_k(x,v)}\in \mathfrak{C}^-_\Lambda \Big\}.
\end{split}
\end{equation*}
The sequence $(T_k(x,v),X_k(x,v),V_k(x,v))_{k\in\N}$ is the sequence of footprints of the backward characteristic trajectory starting at $(x,v)$; $N(t,x,v)$ is almost always finite and such that $T_{N(t,x,v)} \leq t < T_{N(t,x,v)+1}(x,v)$. We refer to Subsection \ref{subsec:collisionfrequencycharacteristics} for more details.

\bigskip
As explained in Lemma \ref{lem:boundarycontinuityset}, the set $\mathfrak{C}^-_\lambda$ describes the boundary points in the phase space that lead to continuous specular reflections.
\par The proof of Proposition \ref{prop:continuityqualitative} relies on a continuity result for the non-homogeneous transport equation with a mixed specular and in-flow boundary conditions when $\Omega$ is not necessarily convex.
\par Similar results have been obtained in \cite[Lemma 12]{Kim1} or \cite[Lemma 13]{Gu6} (when $\Omega$ is convex) for purely in-flow boundary condition as well as for purely bounce-back reflections \cite[Lemma 15]{Kim1}. We recover their results when $\alpha=1$ or by replacing $(T_k,X_k,V_k)_k$ by the sequence associated to bounce-back characteristics. The continuity for pure specular reflections has been tackled in \cite[Lemma 21]{Gu6} but required strict convexity of $\Omega$.
\par The following lemma therefore improves and extends the existing results.

\bigskip
\begin{lemma}\label{lem:continuityspecularinflow}
Let $\Omega$ be a $C^1$ bounded domain of $\R^3$ and let $f_0(x,v)$ be continuous on $\Omega\times\R^3\cup\br{\Lambda^+\cup \mathfrak{C}^-_\Lambda}$ and $g(t,x,v)$ be a boundary datum continuous on $[0,T]\times \mathfrak{C}^-_\Lambda$. At last, let $q_1(t,x,v)$ and $q_2(t,x,v)$ be two continuous function in the interior of $[0,T]\times\Omega\times\R^3$ satisfying
$$\sup\limits_{t \in [0,T]}\norm{q_1(t,x,v)}_{L^\infty_{x,v}(m)}<\infty \quad\mbox{and}\quad \sup\limits_{t \in [0,T]}\norm{q_2(t,x,v)}_{L^\infty_{x,v}(m)}<\infty.$$
Assume $f_0$ satisfies the mixed specular and in-flow boundary condition:
$$\forall (x,v)\in \mathfrak{C}^-_\Lambda,\quad f_0(x,v)=(1-\alpha) f_0(x,\mathcal{R}_x(v))+ g(0,x,v)$$
and suppose $f(t,x,v)$ is the solution to
$$\left\{\begin{array}{l}\disp{\cro{\partial_t + v\cdot\nabla_x + q_1(t,x,v)}f(t,x,v) = q_2(t,x,v)} \vspace{2mm}\\ \vspace{2mm} \disp{\forall (t,x,v)\in [0,T]\times\Lambda^-,\: f(t,x,v) = (1-\alpha) f(t,x,\mathcal{R}_x(v)) +g(t,x,v)}\end{array}\right.$$
associated to the initial datum $f_0$. Then $f(t,x,v)$ is continuous on the continuity set $\mathfrak{C}$.
\end{lemma}
\bigskip

\begin{proof}[Proof of Lemma \ref{lem:continuityspecularinflow}]
As now standard, in the homogeneous case $q_2=0$, we can use a Duhamel formula along the backward characteristics because $q_1$ belongs to $L^\infty_{t,x,v}$. More precisely, as in Subsection \ref{subsubsec:implicit} with $q_1(t,x,v)$ replacing $\nu(v)$ we obtain that if $h(t,x,v)$ is solution to
$$\cro{\partial_t + v\cdot\nabla_x + q_1(t,x,v)}h(t,x,v) = 0$$
with the mixed specular and in-flow boundary conditions then $h$ takes the form
\begin{itemize}
\item if $t\leq t_{min}(x,v)=T_1$, 
$$h(t,x,v) = h_0(x-tv,v)e^{-\int_0^t q_1(s,x-(t-s)v,v)ds};$$
\item if $t> T_1$
$$h(t,x,v)= \cro{(1-\alpha) h(t-T_1,X_1,V_1)+g(t-T_1,X_1,V_1)}e^{-\int_{t-T_1}^{t} q_1(s,x-(t-s)v,v)ds}.$$
\end{itemize}
Unlike the case of Maxwell boundary condition, we see that in the case of mixed specular and in-flow boundary condition we can always reach the initial plane $\br{t=0}$. We obtain an explicit form for $h(t,x,v)$ for almost every $(t,x,v)$ by iterating the property above (see \cite[lemma 20]{Gu6} for more details or \cite[Lemma 15]{Kim1} replacing the bounce-back trajectories by the specular ones). It reads with $N=N(t,x,v))$ and the usual notation $t_k = t-T_k(x,v)$
\begin{equation}\label{homogeneouscontinuity}
\begin{split}
h(t,x,v)=& (1-\alpha)^{N}h_0(X_{N}-t_NV_{N},V_N)\:e^{-\sum\limits_{k=0}^{N}\int_{\min\br{0,t_{k+1}}}^{t_k} q_1(s,X_k-(t_k-s)V_k,V_k)\:ds}
\\&+\sum\limits_{k=0}^{N-1} g(t_{k+1},X_{k+1},V_{k+1})\:e^{-\int_{\min\br{0,t_{k+1}}}^{t_k} q_1(s,X_k-(t_k-s)V_k,V_k)\:ds}
\end{split}
\end{equation}
for almost every $(t,x,v)$. Note that this expression is indeed well-defined since $N(t,x,v)$ is finite almost everywhere and $q_1$ belongs to $L^\infty_{t,x,v}$. We also emphasize that $\min\br{0,t_{k+1}}$ only plays a role when $k=N(t,x,v)$; it encodes the fact that we integrate all the complete lines between $t_k$ and $t_{k+1}$ and only the remaining part $[t-T_N,t]$ of the last line.

\bigskip
Since the source term $q_2$ also belongs to $L^\infty_{t,x,v}$, we obtain an explicit formula for $f(t,x,v)$ from $\eqref{homogeneouscontinuity}$. It reads, for almost every $(t,x,v)$,
\begin{equation}\label{nonhomogeneouscontinuity}
\begin{split}
f(t,x,v)=& (1-\alpha)^{N}f_0(X_{N}-t_NV_{N},V_N)\:e^{-\sum\limits_{k=0}^{N}\int_{\min\br{0,t_{k+1}}}^{t_k} q_1(s,X_k-(t_k-s)V_k,V_k)\:ds}
\\&+\sum\limits_{k=0}^{N-1} g(t_{k+1},X_{k+1},V_{k+1})\:e^{-\int_{\min\br{0,t_{k+1}}}^{t_k} q_1(s,X_k-(t_k-s)V_k,V_k)\:ds}
\\&+ \sum\limits_{k=0}^{N-1}\int_{\min\br{0,t_{k+1}}}^{t_k} \mbox{exp}\cro{-\sum\limits_{l=0}^{k}\int_{\max\br{s,t_{l+1}}}^{t_l} q_1(s_1,X_l-(t_l-s_1)V_l,V_l)ds_1} 
\\&\quad\quad\quad\quad\quad \times q_2(s,X_k-(t_k-s)V_k,V_k)\:ds.
\end{split}
\end{equation}
Note that in the expression above we used the change of variable $s_1 \mapsto t-s_1$ to recover exactly the sequence $(t_l,X_l,V_l)$ associated to $(t,x,v)$ instead of $(\tilde{t}_l,\tilde{X}_l,\tilde{V}_l)$ associated with $(t-s,x,v)$.

\bigskip
By assumptions on $f_0$ and $g$, we deduce that $f$ is continuous on
$$\Big\{\br{0}\times\br{\Omega\times\R^3\cup\pa{\Lambda^+\cup \mathfrak{C}^-_\Lambda}}\Big\} \cup \Big\{(0,+\infty)\times\mathfrak{C}^-_\Lambda \Big\}.$$
Now if $(t,x,v)$ belongs to 
\begin{equation*}
\begin{split}
&\Big\{(t,x,v) \in (0,+\infty)\times\pa{\Omega\times\R^3\cup \Lambda^+}:
\\&\quad\forall 1\leq k\leq N(t,x,v) \in\N, \pa{X_{k+1}(x,v),V_k(x,v)}\in \mathfrak{C}^-_\Lambda \Big\}
\end{split}
\end{equation*}
we have by iterating Lemma \ref{lem:boundarycontinuityset} that the finite sequence $(T_k,X_k,V_k)_{0\leq k \leq N(t,x,v)}$ is continuous around $(x,v)$.
\par Let $(t',x',v')$ be in the same set as $(t,x,v)$. In the case $T_{N(t,x,v)}\leq t\leq t'<T_{N(t,x,v)+1}$ or $T_{N(t,x,v)}\leq t'\leq t<T_{N(t,x,v)+1}$, by continuity of the $t-T_k(x,v)$ we have that for $(t',x',v')$ sufficiently close to $t$, $N(t',x',v')=N(t,x,v)$ and the continuity of $(T_k,X_k,V_k)_{0\leq k \leq N(t,x,v)}$, $g$, $q_1$ and $q_2$ implies $f(t',x',v')\to f(t,x,v)$. It remains to deal with the case $t'\leq t = T_{N(t,x,v)}$ where $N(t',x,v) = N(t,x,v)-1$. Exactly as proved in \cite[Lemma 21]{Gu6}, in that case $t_{N(t,x,v)}=0$ and the integrals from $0$ to $t_{N}$ are null in formula $\eqref{nonhomogeneouscontinuity}$. Moreover, $(X_{N(t,x,v)-1}-(t'-T_{N(t,x,v)-1})V_{N(t,x,v)-1})$ converges to $X_{N(t,x,v)}$ as $t'$ tends to $t$. Finally, since $f_0$ satisfies the boundary condition, we obtain here again that $f(t',x',v')\to f(t,x,v)$. Which concludes the proof.
\end{proof}
\bigskip

We now prove the continuity of the solutions constructed in Section \ref{sec:fullcauchy}.

\bigskip
\begin{proof}[Proof of Proposition \ref{prop:continuityqualitative}]
We use a sequence to approximate the solution  of the full Boltzmann equation with initial datum $F_0 = \mu +f_0$. We start from $F^{(0)}=\mu$ and define by induction $F^{(n+1)}=\mu + f^{(n+1)}$ such that
$$\cro{\partial_t + v\cdot\nabla_x + q(F^{(n)})}F^{(n+1)} = Q(F^{(n)},F^{(n)})$$
with the mixed specular and diffusive boundary conditions
$$\forall (x,v)\in \Lambda^-,\quad F^{(n+1)}(x,v) = (1-\alpha) F^{(n+1)}(x,\mathcal{R}_x(v)) + \alpha P_\Lambda(F^{(n)}(x,\cdot))(v).$$

\bigskip
Since we impose a specular part in the boundary condition, similar computations as in Section \ref{sec:frequencycollision} show that $\pa{f^{(n)}}_{n\in\N}$ is well-defined in $L^\infty_{t,x,v}(m)$. Moreover, similar computations as Subsection \ref{subsubsec:f1} prove that $\pa{f^{(n)}}_{n\in\N}$ is a Cauchy sequence, at least on $[0,T]$ for $T$ sufficiently small, as well as $\pa{F^{(n)}}_{n\in\N}$. Therefore $\pa{f^{(n)}}_{n\in\N}$ converges towards $f$ the unique solution of the perturbed Boltzmann equation with initial datum $f_0$ and $\pa{F^{(n)}}_{n\in\N}$ converges to $F$ the unique solution of the full Boltzmann equation with initial datum $F_0=\mu+f_0$.
\par We apply Lemma \ref{lem:continuityspecularinflow} inductively on $\nu(v)^{-1}F^{(n+1)}$. Indeed, \cite[Theorem 4 and Corollary 5]{Kim1} showed that $q_1 = \nu(v)^{-1}q(F^{(n)})$ and $q_2 = \nu(v)^{-1}Q(F^{(n)},F^{(n)})$ are continuous in the interior of $[0,T]\times\bar{\Omega}\times\R^3$ if $F^{(n)}$ is continous on $\mathfrak{C}$ (see also Lemma \ref{lem:controlQ}). And \cite[Proof of 2 of Theorem 3, Step 1]{Kim1} proved that $P_\Lambda(F^{(n)})$ is continuous on $[0,T]\times\mathfrak{C}^-_\Lambda$ even if $F^{(n)}$ is only continuous on
$$[0,T]\times\bar{\Omega}\times\R^3 -\br{(x,v)\in \bar{\Omega}\times\R^3,\quad n(X_1(x,v))\cdot v =0}$$
which is included in $\mathfrak{C}$.
\par Hence, by induction $F^{(n)}$ is continuous on $\mathfrak{C}$ for all $n$ and is a Cauchy sequence. Therefore its limit $F$ is continuous as well. 
\end{proof}
\bigskip